\documentclass[12pt]{article}

\input diagram.tex
\usepackage{amsmath} % eg. for \begin{equation}
\usepackage[mathscr]{eucal}
\usepackage{amssymb} % eg. for \leqslant
\usepackage{theorem}

\theoremstyle{change}  % puts numbers IN FRONT of "Theorem"
\newtheorem{theorem}{Theorem}[section] % defines environment "Theorem".
\newtheorem{lemma}[theorem]{Lemma}  % defines environment "Lemma", that
\newtheorem{proposition}[theorem]{Proposition}
\newtheorem{corollary}[theorem]{Corollary}
\theorembodyfont{\rmfamily}  % has the effect that the content of Remark,
\newtheorem{remark}[theorem]{Remark}
\newtheorem{example}[theorem]{Example}

\newtheorem{definition}[theorem]{Definition}
\newtheorem{notation}[theorem]{Notation}
\newtheorem{nothing}[theorem]{} % empty Theoremumgebung.

\newtheorem{hypothesis}[theorem]{Hypothesis}

\newenvironment{proof}{\noindent{\bf Proof}\ }{\qed\bigskip}

\renewcommand{\le}{\leqslant} % needs amssymb-Paket
\renewcommand{\ge}{\geqslant}

\newcommand{\mylabel}[1]{\label{#1}}

\newcommand{\abar}{\bar{a}}
\newcommand{\alphabar}{\bar{\alpha}}
\newcommand{\Aut}{\mathrm{Aut}}
\newcommand{\bbar}{\bar{b}}

\newcommand{\calC}{\mathcal{C}}
\newcommand{\calCkA}{\calC_k^A}                       
\newcommand{\calE}{\mathcal{E}}
\newcommand{\calF}{\mathcal{F}}
\newcommand{\calFkA}{\calF_k^A}
\newcommand{\calH}{\mathcal{H}}
\newcommand{\calI}{\mathcal{I}}
\newcommand{\calK}{\mathcal{K}}
\newcommand{\calM}{\mathcal{M}}
\newcommand{\calMc}{\calM^c}

\newcommand{\calR}{\mathcal{R}}
\newcommand{\calRtilde}{\tilde{\calR}}
\newcommand{\calS}{\mathcal{S}}
\newcommand{\calSbar}{\bar{\cal{S}}}
\newcommand{\calStilde}{\tilde{\calS}}
\newcommand{\calSkA}{\calS_k^A}                                      
\newcommand{\calSbarkA}{\bar{\calS}_k^A}                        
\newcommand{\can}{\mathrm{can}}
\newcommand{\catfont}{\mathsf}

\newcommand{\CC}{\mathbb{C}}
\newcommand{\cdotG}{\mathop{\cdot}\limits_{G}}
\newcommand{\cdotH}{\mathop{\cdot}\limits_{H}}

\newcommand{\Def}{\mathrm{Def}}
\newcommand{\Del}{\mathrm{Del}}
\newcommand{\dia}{\mathrm{diag}}

\newcommand{\Ebar}{\bar{E}}
\newcommand{\End}{\mathrm{End}}
\newcommand{\etabar}{\bar{\eta}}
\newcommand{\Ext}{\mathrm{Ext}}
\newcommand{\fbar}{\bar{f}}

\newcommand{\Gbar}{\bar{G}}

\newcommand{\GsetA}{\lset{G}^A}
\newcommand{\GsetGA}{\lset{G}_G^A}
\newcommand{\GsetHA}{\lset{G}_H^A}
\newcommand{\GsetKA}{\lset{G}_K^A}

\newcommand{\Gtilde}{\tilde{G}}
\newcommand{\Hbar}{\bar{H}}
\newcommand{\HsetGA}{\lset{H}_G^A}
\newcommand{\HsetKA}{\lset{H}_K^A}
\newcommand{\Hom}{\mathrm{Hom}}
\newcommand{\Htilde}{\tilde{H}}
\newcommand{\id}{\mathrm{id}}
\newcommand{\im}{\mathrm{im}}
\newcommand{\Ind}{\mathrm{Ind}}
\newcommand{\ind}{\mathrm{ind}}
\newcommand{\Inf}{\mathrm{Inf}}
\newcommand{\Inn}{\mathrm{Inn}}
\newcommand{\Ins}{\mathrm{Ins}}
\newcommand{\iotatilde}{\tilde{\iota}}
\newcommand{\Irr}{\mathrm{Irr}}

\newcommand{\kappabar}{\bar{\kappa}}
\newcommand{\kappatilde}{\tilde{\kappa}}
\newcommand{\kappattilde}{\tilde{\kappatilde}}
\newcommand{\Kbar}{\bar{K}}
\newcommand{\Khat}{\hat{K}}
\newcommand{\kMod}{\lMod{k}}
\newcommand{\KsetLA}{\lset{K}_L^A}
\newcommand{\Ktilde}{\tilde{K}}
\newcommand{\Kttilde}{\tilde{\Ktilde}}
\newcommand{\ktilde}{\tilde{k}}
\newcommand{\lambdabar}{\bar{\lambda}}
\newcommand{\lDelta}[1]{\llap{\phantom{|}}_{#1}\Delta}
\newcommand{\Lbar}{\bar{L}}
\newcommand{\lexp}[2]{\setbox0=\hbox{$#2$} \setbox1=\vbox to
                 \ht0{}\,\box1^{#1}\!#2}
\newcommand{\lGamma}[1]{\llap{\phantom{|}}_{#1}\Gamma}
\newcommand{\Lhat}{\hat{L}}
\newcommand{\lin}{\mathrm{lin}}
\newcommand{\liso}{\buildrel\sim\over\longrightarrow}

\newcommand{\lMod}[1]{\llap{\phantom{|}}_{#1}\catfont{Mod}}
\newcommand{\lmon}[1]{\llap{\phantom{|}}_{#1}\catfont{mon}}

\newcommand{\lset}[1]{\llap{\phantom{|}}_{#1}\catfont{set}}
\newcommand{\Ltilde}{\tilde{L}}
\newcommand{\ltilde}{\tilde{l}}
\newcommand{\Mat}{\mathrm{Mat}}
\newcommand{\myiso}{\buildrel\sim\over\to}
\newcommand{\mubar}{\bar{\mu}}
\newcommand{\mutilde}{\tilde{\mu}}
\newcommand{\NN}{\mathbb{N}}
\newcommand{\omegabar}{\bar{\omega}}
\newcommand{\Out}{\mathrm{Out}}
\newcommand{\pbar}{\bar{p}}
\newcommand{\phibar}{\bar{\phi}}
\newcommand{\phiop}{\phi^{\mathrm{op}}}
\newcommand{\phitilde}{\tilde{\phi}}
\newcommand{\pitilde}{\tilde{\pi}}

\newcommand{\psitilde}{\tilde{\psi}}
\newcommand{\qed}{\nobreak\hfill
                  \vbox{\hrule\hbox{\vrule\hbox to 5pt
                  {\vbox to 8pt{\vfil}\hfil}\vrule}\hrule}}
\newcommand{\QQ}{\mathbb{Q}}
\newcommand{\Res}{\mathrm{Res}}
\newcommand{\sbar}{\bar{s}}

\newcommand{\sigmatilde}{\tilde{\sigma}}
\newcommand{\Stilde}{\tilde{S}}
\newcommand{\thetatilde}{\tilde{\theta}}
\newcommand{\torA}{\mathrm{tor}A}
\newcommand{\Ubar}{\bar{U}}
\newcommand{\Uop}{U^{\mathrm{op}}}
\newcommand{\Utilde}{\tilde{U}}
\newcommand{\Vtilde}{\widetilde{V}}
\newcommand{\Wtilde}{\widetilde{W}}

\newcommand{\xop}{x^{\mathrm{op}}}
\newcommand{\Xop}{X^{\mathrm{op}}}

\newcommand{\Yop}{Y^{\mathrm{op}}}

\newcommand{\ZZ}{\mathbb{Z}}

\title{Fibered Biset Functors\footnote{MR Subject Classification 19A22, 20C15, 20C20}}
\author{\small Robert Boltje\thanks{Research supported by the NSF, DMS-0200592}\\
        \small Department of Mathematics\\
        \small University of California\\
        \small Santa Cruz, CA 95064\\
        \small U.S.A.\\
        \small boltje@math.ucsc.edu %
        \and
        \small Olcay Co\c{s}kun\thanks{Research supported partially by  T\"ubitak-1001-113F240.}\\
        \small Department of Mathematics\\
        \small Bo\u{g}azi\c{c}i University\\
        \small Bebek, 34342 Istanbul\\
        \small T\"{u}rkiye\\
        \small olcay.coskun@boun.edu.tr}
\date{December 4, 2016}

\begin{document}

\sloppy

\maketitle

%%%%%%%%%%%%%%%%%%%%%%%%%%%% ABSTRACT %%%%%%%%%%%%%%%%%%%%%%%%%%%%%%%%

\begin{abstract}
The theory of biset functors, introduced by Serge Bouc, gives a unified treatment of operations in representation theory that are induced by permutation bimodules. In this paper, by considering fibered bisets, we introduce and describe the basic theory of fibered biset functors which is a natural framework for operations induced by monomial bimodules. The main result of this paper is the classification of simple fibered biset functors. 
\end{abstract}

%%%%%%%%%%%%%%%%%%%%%% INTRODUCTION %%%%%%%%%%%%%%%%%%%%%%%%%%%%%%%%%%

\section*{Introduction}

In \cite{Bouc1996a} and \cite{BoucBook} Serge Bouc introduced and developed the theory of {\em biset functors}. This notion provides a framework for situations where structural maps that behave like {\em restriction}, {\em induction}, {\em inflation}, and {\em deflation}, or a subset of them, are present. Typical examples of biset functors are the various representation rings, as for instance the Burnside ring, the character ring, the Green ring, and the trivial source ring of a finite group $G$. The Brauer character ring of $G$ and cohomology groups in fixed degree of $G$ are other examples (no deflation in these cases). The theory of biset functors proved to be the right framework to prove striking results: The determination of the Dade group of endo-permutation modules of a $p$-group (see \cite{BoucThevenaz2000a} for the final result, or \cite{Thevenaz2007} for an overview) and the determination of the unit group of the Burnside ring of a $p$-group (see \cite{Bouc2010a} or \cite{BoucBook}). Both the Dade group and the unit group of the Burnside ring provide further examples of biset functors. 

Biset functors are additive, abelian group valued functors on the biset category, whose objects are finite groups and whose morphism sets are given by the Grothendieck groups $B(G,H)$ of finite $(G,H)$-bisets. The composition is induced by a construction that imitates the tensor product of bimodules. Restriction, induction, inflation and deflation can be seen as particular transitive bisets.

Recently, a category analogous to the biset category, whose objects are finite sets and morphisms are correspondences between them, together with functors on this category have been considered by Bouc and Th\'evenaz, see~\cite{BoucThevenaz2015}. Again surprising connections to other areas, for instance to lattice theory, came to light.

In this article we strive to systematically develop a similar theory of {\em fibered biset functors}. The motivation comes from the fact that representation rings of finite groups carry more structure, by considering multiplication with one-dimensional representations as structural maps. Monomial Burnside rings have been introduced by Dress in \cite{Dress} in great generality and utilized by the first author in the formalism of canonical induction formulas, see \cite{Boltje1989} and \cite{Boltje1998}. See also \cite{Barker} and \cite{Romero2013a} for some basic properties. Let $A$ be a multiplicatively written abelian group. The $A$-fibered Burnside ring $B^A(G)$ is the Grothendieck group of $A$-fibered $G$-sets, i.e., $G\times A$-sets with finitely many orbits which are free as $A$-sets. A $\ZZ$-basis of $B^A(G)$ is parametrized by $G$-conjugacy classes of pairs $(H,\phi)$, where $H$ is a subgroup of $G$ and $\phi\in\Hom(H,A)$. Similarly one defines $A$-fibered $(G,H)$-bisets and their Grothendieck groups $B^A(G,H)=B^A(G\times H)$. Also fibered bisets allow a tensor product construction which gives rise to the $A$-fibered biset category over a commutative ring $k$, denoted by $\calC_k^A$. Its objects are again finite groups and its morphism sets are given by $B_k^A(G,H):=k\otimes B^A(G,H)$. The $k$-linear functors from $\calCkA$ to $\lMod{k}$ together with natural transformations form the abelian category $\calFkA$ of {\em $A$-fibered biset functors over $k$}. If $k=\ZZ$, we usually suppress the index $k$. In the case that $A$ is the unit group of a field $K$, one obtains a natural linearization map from $B^A(G)$ to various representation rings of $KG$-modules, by mapping the class of $(H,\phi)$ to $\Ind_H^G(K_\phi)$, where $K_\phi$ denotes the one-dimensional $KH$-module associated with the homomorphism $\phi\colon H\to K^\times$. Using the linearization map one can interpret these representation rings as additive functors on the $A$-fibered biset category with values in the category of abelian groups, i.e., as {\em $A$-fibered biset functors}.

One has a natural embedding $B(G)\to B^A(G)$ and also a natural splitting map $B^A(G)\to B(G)$ of this embedding in the category of rings. The embedding allows to view the biset category as a subcategory of the $A$-fibered biset category. To view an $A$-fibered biset functor $F$ via restriction as a biset functor means forgetting some of its structure. Thus, if a biset functor comes via restriction from an $A$-fibered biset functor it is worth trying to understand its richer structure as an $A$-fibered biset functor. 

The main goal of this paper is to determine the simple $A$-fibered biset functors. For this we first need to study the tensor product of $A$-fibered bisets and how transitive $A$-fibered bisets can be factored through smaller groups. This is done in Sections~1 and 2. The situation for bisets is as follows: It is shown in \cite{BoucBook} that a transitive $(G,H)$-biset
$(G\times H)/ U$, with $U\le G\times H$, is equal to the product of five canonical bisets. More precisely, one has
\begin{equation}\label{eqn biset decomposition}
\Big(\frac{G\times H}{U}\Big) = \
\Ind_P^G\times_P\Inf_{P/K}^P\times_{P/K}\mbox{\rm c}_{P/K,Q/L}^\eta\times_{Q/L}\Def^Q_{Q/L}\times_{Q}\Res^H_Q.
\end{equation}
Here, $P=p_1(U)$ and $Q = p_2(U)$ are the first and the second projections of the subgroup $U\le G\times H$ and $K = p_1(U\cap (P\times 1))$ and 
$L = p_2(U\cap (1\times Q))$. In this case, the groups $P/K$ and $Q/L$ are isomorphic and a canonical isomorphism $\eta$ is 
determined by $U$. Moreover, $\times_?$ denotes the product of bisets. See \ref{4 operations} for the description of 
the factors appearing in (\ref{eqn biset decomposition}). We usually write a transitive {\em $A$-fibered} $(G,H)$-biset $X$ in the form $\left(\frac{G\times H}{U,\phi}\right)$, where $U\le G\times H$ is the stabilizer of the $A$-orbit of a chosen element $x\in X$ and $\phi\colon U\to A$ is the homomorphism  arising from the action of $U$ on this $A$-orbit. Different choices of $x$ lead to $G\times H$-conjugates of $(U,\phi)$. In Corollary~\ref{Mackey formula for fibered bisets 2} we derive a Mackey formula for the tensor product of two transitive $A$-fibered bisets and use it to show an analogue of the decomposition in (\ref{eqn biset decomposition}). As above, $U$ determines a quintuple $(P,K,\eta,L,Q)$. Further, the homomorphism $\phi\colon U\to A$ determines the following data: Let $\phi_1\times \phi_2^{-1}$ be the decomposition of the restriction of $\phi$ to $K\times L$ and let $\hat K$ (resp.~$\hat L$) be the kernel of $\phi_1$ (resp. $\phi_2^{-1}$). We write $l(U,\phi) = (P,K,\phi_1)$ and $r(U,\phi) = (Q,L,\phi_2)$ and call them the left and the right invariants of $(U,\phi)$. In Proposition~\ref{decomposition} we prove the following decomposition
\begin{equation}\label{eqn fibered biset decomposition 1}
\left(\frac{G\times H}{U,\phi}\right) = 
\Ind_P^G\otimes_{AP}\Inf_{P/\hat K}^P\otimes_{AP/\hat K}\mbox{\rm Y}\otimes_{AQ/\hat L}\Def^Q_{Q/\hat L}\otimes_{AQ}\Res^H_Q.
\end{equation}
Here, $Y$ is a transitive $A$-fibered $(P/\hat K, Q/\hat L)$-biset with stabilizing pair $(\bar U,\bar \phi)$ such that the first and the second projections of $\bar U$ are surjective, and the homomorphisms $\bar\phi_1$ and $\bar\phi_2$ are faithful. In contrast to the theory of bisets is that in general the $A$-fibered biset $Y$ can be factored through groups of smaller order than $P/\Khat$ and $Q/\Lhat$. The problem remains to decompose $Y$ further by factoring through a smallest possible group. It is solved in Section~10 under additional assumptions on $A$, see Hypothesis~\ref{hypothesis}, which hold for instance if $A$ is the group of units of an algebraically closed field. Under this assumption, we decompose $X$ as follows. For simplicity, assume $P=G, Q=H, \hat K =1$ and $\hat L=1$. Set $\tilde K = G^\prime\cap K$ and $\tilde L = H^\prime \cap L$, where $G^\prime$ (resp.~$H^\prime$)  denotes the derived subgroup of $G$ (resp.~H). We construct a triple $(\tilde G, M, \mu)$, where $\tilde G$ is a central extension of $G/K$ by $\tilde K$, $M$ is a subgroup of $G\times \tilde G$, and $\mu\in\Hom(M,A)$, with additional properties as described in Section \ref{sec completely reduced pairs}, and also a triple $(\tilde H, N,\psi)$ with analogous properties related to $H$, $L$, and $\tilde{L}$, such that 
\begin{equation*}
Y \cong \Ins_{\tilde G}^G\otimes_{A\tilde G} X\otimes_{A\tilde H} \Del_{\tilde H}^H\,,
\end{equation*}
where $X$ is an $A$-fibered $(\tilde G,\tilde H)$-biset which is reduced, that is, cannot be factored further through a group of order smaller than $|\Gtilde|=|\Htilde|$. 
Here the \emph{insertion} fibered biset $\Ins_{\tilde G}^G$ \emph{inserts} the section $K/\tilde K$ into the group $\tilde G$ and the 
\emph{deletion} fibered biset $\Del^H_{\tilde H}$ \emph{deletes} the section $L/\tilde L$ from the group $H$. Unfortunately, the construction of $\Ins_{\tilde G}^G$ and $\Del^H_{\tilde H}$ uses choices which make them not canonical.

In Section~3 we introduce $A$-fibered biset functors and recall basic properties of such functor categories from \cite{Bouc1996a}. In order to determine the simple $A$-fibered biset functors over a commutative ring $k$ one needs to find the simple modules of the endomorphism algebra $E_k^A(G) = B_k^A(G,G)$ of a finite group $G$ which are annihilated by the ideal $I^A_k(G)$ of endomorphisms that factor through groups of smaller order. In order to determine the simple $E_k^A(G)/I_k^A(G)$-modules we consider the subalgebra $E_k^{A,c}(G)$ of $E_k^A(G)$ spanned over $k$ by the isomorphism classes of {\em covering} transitive $A$-fibered $(G,G)$-bisets $\left(\frac{G\times G}{U,\phi}\right)$, i.e., satisfying $p_1(U)=G=p_2(U)$. This subalgebra of $E_k^A(G)$ covers $E_k^A(G)/I_k^A(G)$ and it is isomorphic to a product of matrix rings over group algebras $k\Gamma_{(G,K,\kappa)}$. Here, $\Gamma_{(G,K,\kappa)}$ is the group of isomorphism classes of covering transitive $A$-fibered $(G,G)$-bisets $\left(\frac{G\times G}{U,\phi}\right) $ with $l(U,\phi)=(G,K,\kappa)=r(U,\phi)$. This result is proved in the course of Sections~4, 5, and 6. In Section~4 we introduce the central idempotents of $E_G^c$ which split the algebra $E_G^c$ into matrix rings. In Section~5 we determine an equivalence relation on the pairs $(K,\kappa)$ whose equivalence classes will parametrize the various matrix components. And in Section~6 we prove that $E_k^{A,c}(G)$ has the announced structure. In Section~7, the groups $\Gamma_{(G,K,\kappa)}$ are identified as extensions of a subgroup of the outer automorphism group $\Out(G/K)$ (determined by $\kappa$) and the group $\Hom(G/K,A)$. Section~8 is devoted to understanding which simple modules of $E_G^c$ are annihilated by $E_G\cap I_G$. For a given matrix component, indexed by $(K,\kappa)$, this only depends on the pair $(K,\kappa)$. Pairs that will lead to simple $E_G/I_G$-modules are called {\em reduced} pairs. Unfortunately, for general $A$ we do not have a handy criterion when $(K,\kappa)$ is reduced. Proposition~\ref{conditions for reduced} gives a necessary and a (different) sufficient condition for $(K,\kappa)$ to be reduced.

In Section~9 we show that the simple $A$-fibered biset functors are parametrized by equivalence classes of quadruples $(G,K,\kappa, [V])$, where $G$ is a finite group, $(K,\kappa)$ is a reduced pair, and $[V]$ is the isomorphism class of a simple $k\Gamma_{(G,K,\kappa)}$-module. This equivalence relation involves the notion of {\em linkage} of two tripes $(G,K,\kappa)$ and $(H,L,\lambda)$ by a pair $(U,\phi)$ with $U\le G\times H$ and $\phi\in\Hom(U,A)$ such that $l(U,\phi)=(G,K,\kappa)$ and $r(U,\phi)=(H,L,\lambda)$. In Proposition~\ref{linkage criterion} a group cohomology criterion is given to determine if two triples are linked. There exist examples of such linked triples with $G$ and $H$ not isomorphic. They lead to a negative answer of a question asked by Bouc, see Remark~\ref{rem Romero}, which has been observed independently by Romero in \cite{Romero2012a}.

In Section~10, we show that under a further hypothesis on the group $A$, one can show that a pair $(K,\kappa)$ is reduced if and only if $K\le Z(G)\cap G'$ and $\kappa$ is injective. This uses lengthy computations in the cohomology groups $H^2(G/K,A)$ and related cohomology groups. Finally, in Section~11 we realize some representation rings as simple fibered biset functors over appropriate fields of coefficients.

\bigskip
{\bf Notation}\quad
Throughout the paper, we will adopt the following conventions:

If $G$ is a group, $H\le G$ is a subgroup and $x\in G$, then $\lexp{x}{H}:=xHx^{-1}$. By $H<G$ we indicate that $H$ is a proper subgroup of $G$, i.e., $H\neq G$. If $k$ is a commutative ring we denote by $kG$ or $k[G]$ the group algebra of $G$ over $k$. If $H\le G$ and $M$ is a left $k[H]$-module then $\lexp{x}{M}$ is the $k[\lexp{x}{H}]$-module with underlying $k$-module equal to $M$ and $\lexp{x}{H}$-action given by $xhx^{-1}\cdot m := h m$, for $h\in H$ and $m\in M$. If $A$ is an abelian group, $\phi\colon H\to A$ is a group homomorphism and $x\in G$, then we define the homomorphism $\lexp{x}{\phi}\colon\lexp{x}{H}\to A$ by $xhx^{-1}\mapsto \phi(h)$ for $h\in H$.

For any ring $R$, the isomorphism class of an irreducible left $R$-module $V$ is denoted by $[V]$ and the set of isomorphism classes of simple left $R$-modules is denoted by $\Irr(R)$.

\bigskip
{\bf Acknowledgement} Work on this paper began 2007 during a visit of the second author at UC Santa Cruz. Since then both authors have enjoyed the support of the following institutions while working on this projects: UC Santa Cruz, MSRI,  Bo\u{g}azi\c{c}i University, Nesin Mathematics Village, Centre Interfacultaire Bernoulli. We would like to express our gratitude to these institutions.

%If $M$ is an $R[G]$-module and $H\le G$, the $H$-coinvariant module $M_H$ is defined as $M/IM$, where $I\subset R[H]$ is the augmentation ideal, i.e., the set of all $R$-linear combinations of elements of the form $h-1$, $h\in H$. For $m\in M$ we abbreviate $m+IM$ just by $\overline{m}$. Note that $M_H$ is an $R[N_G(H)/H]$-module in a natural way by setting $(nH)\cdot \overline{m}:=\overline{nm}$, for $n\in N_G(H)$ and $m\in M$. Note also that $M\mapsto M_H$ is a functor from the category of $R[G]$-modules to the category of $R[N_G(H)/H]$-modules which is equivalent to the tensor functor $R\otimes_{R[H]}-$, where $R$ is considered as the trivial right $R[H]$-module.

%%%%%%%%%%%%%%%%%%%%%%% SECTION 1 %%%%%%%%%%%%%%%%%%%%%%%%%%%%%%%%%%%%

\section{$A$-Fibered Bisets}\label{sec fibered bisets}

Throughout this article, we fix a multiplicatively written abelian
group $A$. For every finite group $G$, we set
\begin{equation*}
  G^*:=\Hom(G,A)
\end{equation*}
and view $G^*$ as an abelian group with point-wise multiplication.

\begin{nothing} {\bf The categories $\GsetA$ and $\GsetHA$.}\quad
Let $G$ be a finite group. An {\em $A$-fibered $G$-set} is a left
$A\times G$-set $X$ which is free as an $A$-set and has finitely many
$A$-orbits. A morphism between two $A$-fibered $G$-sets is just an
$A\times G$-equivariant map. The $A$-fibered $G$-sets and their
morphisms form a category which we denote by $\GsetA$.

We often consider a left $G$-set $X$ also as a right $G$-set via
\begin{equation*}
  xg=g^{-1}x
\end{equation*} 
for any $x\in X$ and $g\in G$
and vice-versa. However, when we view a left $A$-set as a right
$A$-set we always do this via
\begin{equation}\label{two-sided A-action}
  xa=ax
\end{equation}
and vice-versa. Elements of $A$ play a similar role as the elements of a commutative ring $k$, when
switching sides for modules over the group algebra $kG$.

If also $H$ is a finite group we denote by the category $\GsetHA$ of {\em
$A$-fibered $(G,H)$-bisets} as the category $\lset{G\times H}^A$. By
the above convention, we can view an object $X\in\lset{G\times H}^A$
as equipped with a left $G$-action, a right $H$-action and a two-sided
$A$-action through (\ref{two-sided A-action}), all three commuting with each other.
\end{nothing}

\begin{nothing}\label{noth M_G}{\bf The set $\calM_G=\calM^A(G)$.}\quad
For a finite group $G$, we denote by $\calM^A(G)$ the set of all pairs
$(K,\kappa)$ where $K$ is a subgroup of $G$ and $\kappa\in K^*$. Often we will just write $\calM_G$ for $\calM^A(G)$. The
set $\calM_G$ has a poset structure given by $(L,\lambda)\le
(K,\kappa)$ if $L\le K$ and $\lambda=\kappa|_L$. Moreover, $G$ acts on
$\calM_G$ by conjugation via
$\lexp{g}{(K,\kappa)}:=(\lexp{g}{K},\lexp{g}{\kappa})$ for $g\in G$.
Note that conjugation respects the poset structure. We denote the
$G$-orbit of $(K,\kappa)$ by $[K,\kappa]_G$.

Let also $H$ be a finite group and let $(U,\phi)\in\calM_{G\times
H}$. We denote by
\begin{equation*}
  p_1\colon G\times H\to G\quad\text{and}\quad p_2\colon G\times H\to
  H
\end{equation*}
the projection maps and set
\begin{equation*}
  k_1(U):=\{g\in G\mid (g,1)\in U\}\quad\text{and}\quad
  k_2(U):=\{h\in H\mid (1,h)\in U\}\,.
\end{equation*}
We also set
\begin{equation*}
  k(U):=k_1(U)\times k_2(U)\le U\,,
\end{equation*}
the largest \lq rectangular\rq\ subgroup of $U$. Furthermore, for
$i=1,2$, we define homomorphisms $\phi_i\in k_i(U)^*$ by
\begin{equation*}
  \phi|_{k(U)}=\phi_1\times\phi_2^{-1}\,.
\end{equation*}
The reason we use $\phi_2^{-1}$ in the above equation is that later
formulas will look nicer. Finally, we associate to $(U,\phi)$ its {\em
left invariants} and {\em right invariants}
\begin{equation*}
  l(U,\phi):=(p_1(U),k_1(U),\phi_1)\quad\text{and}\quad
  r(U,\phi):=(p_2(U),k_2(U),\phi_2)\,.
\end{equation*}
Sometimes we will only be interested in part of these invariants and
also define
\begin{equation*}
  l_0(U,\phi):=(k_1(U),\phi_1)\quad\text{and}\quad
  r_0(U,\phi):=(k_2(U),\phi_2)\,.
\end{equation*}
\end{nothing}

\begin{proposition}\mylabel{eta and zeta}
Let $(U,\phi)\in\calM_{G\times H}$ and set
\begin{equation*}
  (P,K,\kappa):=l(U,\phi)\quad\text{and}\quad (Q,L,\lambda):=r(U,\phi)\,.
\end{equation*}
Moreover, set
\begin{equation*}
  \Khat:=\ker(\kappa)\,,\quad \Lhat:=\ker(\lambda)\,,\quad \Ktilde:=\Khat P'\cap
K\quad \text{and}\quad \Ltilde:=\Lhat Q'\cap L\,.
\end{equation*}
Then, clearly,
\begin{equation*}
  1 \le \Khat \le \Ktilde \le K \le P \le G\quad\text{and}\quad
  1 \le \Lhat \le \Ltilde \le L \le Q \le H\,.
\end{equation*}
Moreover:

\smallskip
{\rm (a)} $\Khat$, $\Ktilde$ and $K$ are normal in $P$, and $\Lhat$,
$\Ltilde$ and $L$ are normal in $Q$.

\smallskip
{\rm (b)} $K/\Khat$ is central in $P/\Khat$ and $L/\Lhat$ is central in
$Q/\Lhat$.

\smallskip
{\rm (c)} One has group isomorphisms
\begin{equation*}
  P/K \cong U/(K\times L) \cong Q/L
\end{equation*}
which are induced by the projection maps $U\to P$ and $U\to Q$. The
resulting isomorphism
\begin{equation*}
  \eta=\eta_U\colon Q/L\to P/K
\end{equation*}
is characterized by
\begin{equation*}
  \eta(qL)=pK \iff (p,q)\in U\,,
\end{equation*}
for $q\in Q$ and $p\in P$.

\smallskip
{\rm (d)} One has group isomorphisms
\begin{equation*}
  \Ktilde/\Khat \cong
  \ker(\phi|_{\Ktilde\times\Ltilde})/(\Khat\times\Lhat) \cong
  \Ltilde/\Lhat
\end{equation*}
induced by the projection maps $\Ktilde\times\Ltilde\to\Ktilde$ and
$\Ktilde\times\Ltilde\to\Ltilde$. The resulting isomorphism
\begin{equation*}
  \zeta=\zeta_{U,\phi}\colon \Ltilde/\Lhat\to\Ktilde/\Khat
\end{equation*}
is characterized by
\begin{equation*}
  \zeta(\ltilde\Lhat)=\ktilde\Khat \iff
  \kappa(\ktilde)=\lambda(\ltilde)\,,
\end{equation*}
for $\ltilde\in\Ltilde$ and $\ktilde\in\Ktilde$.

\smallskip
{\rm (e)} If $\Khat=1$, $\Ktilde=K$ and $P=G$ then $|G|\le|H|$.
\end{proposition}

\begin{proof}
Parts~(a), (b) and (c) are easy verifications and are left to the
reader. Some parts of these statements are also known as Goursat's
theorem.

\smallskip
(d) We only show that the projection
\begin{equation*}
  \pbar_1\colon \ker(\phi|_{\Ktilde\times\Ltilde})/(\Khat\times\Lhat)
  \to \Ktilde/\Khat\,, \quad (\ktilde,\ltilde)(\Khat\times\Lhat)
  \mapsto \ktilde\Khat\,,
\end{equation*}
is an isomorphism. (The other isomorphism is proved with the same
arguments and the last statement follows easily.) Clearly, $\pbar_1$
is a well-defined group homomorphism, and it is easy to see that
$\pbar_1$ is injective. To show that $\pbar_1$ is surjective, let
$\ktilde\in\Ktilde$. After multiplying $\ktilde$ with an element from
$\Khat$ we may assume that $\ktilde\in P'$. Then, there exist elements
$g_1,g'_1,\ldots,g_n,g'_n\in P$ such that
\begin{equation*}
  \ktilde =  [g_1,g'_1]\cdots[g_n,g'_n]\,.
\end{equation*}
By the definition of $P$ there exist elements
$h_1,h'_1,\ldots,h_n,h'_n\in Q$ such that $(g_i,h_i),(g'_i,h'_i)\in U$
for $i=1,\ldots,n$. Set
\begin{equation*}
  \ltilde:=[h_1,h'_1]\cdots[h_n,h'_n]\in Q'\,.
\end{equation*}
It follows that $(\ktilde,\ltilde)\in U'$, which implies that
$\phi(\ktilde,\ltilde)=1$. Since $\Ktilde\le K$ we have $\ltilde\in
L$. Since $\ltilde$ belongs also to $Q'$, we have $\ltilde\in\Ltilde$
and $\pbar_1(\ktilde,\ltilde)=\ktilde\Khat$.

\smallskip
(e) This follows immediately from parts (c) and (d).
\end{proof}

\begin{nothing}\mylabel{stabilizing pairs}{\bf Stabilizing pairs.}\quad
Let $X$ be an $A$-fibered $(G,H)$-biset. We will denote the $A$-orbit
of the element $x\in X$ by $[x]$. Note that $G\times H$ acts on the
set of $A$-orbits of $X$ by $(g,h)[x]:=[(g,h)x]$. For $x\in X$, let
$S_x\le G\times H$ denote the stabilizer of $[x]$.

We define a map
\begin{equation}\label{stabilizing pair}
  X\to\calM_{G\times H}\,,\quad x\mapsto (S_x,\phi_x)\,,
\end{equation}
with
\begin{equation*}
  \phi_x\colon S_x \to A
\end{equation*}
determined by the equation
\begin{equation*}
  (g,h)x=\phi_x(g,h)x
\end{equation*}
for $(g,h)\in S_x$. We call $(S_x,\phi_x)$ the {\em stabilizing pair}
of $x$. With the notation introduced in the previous paragraph, we also
obtain group homomorphisms
\begin{equation*}
  \phi_{x,1}\colon k_1(S_x)\to A\quad\text{and}\quad
  \phi_{x,2}\colon k_2(S_x)\to A
\end{equation*}
determined by the equations
\begin{equation*}
  \phi_{x,1}(g)x=gx\quad\text{and}\quad \phi_{x,2}(h)x=xh\,,
\end{equation*}
for $g\in k_1(S_x)$ and $h\in k_2(S_x)$. Since the action of $G\times
H$ and $A$ on $X$ commute, the definitions of $\phi_x$, $\phi_{x,1}$
and $\phi_{x,2}$ do not depend on the choice of $x$ in its $A$-orbit.
Note also that for $(g,h)\in G\times H$, we have
\begin{equation*}
  (S_{(g,h)x},\phi_{(g,h)x})=\lexp{(g,h)}{(S_x,\phi_x)}\,,\quad
  \phi_{(g,h)x,1}=\lexp{g}{\phi_{x,1}}\,,\quad
  \phi_{(g,h)x,2}=\lexp{h}{\phi_{x,2}}\,.
\end{equation*}
Thus, the map defined in (\ref{stabilizing pair}) is constant on the
$A$-orbits of $X$, and considered as a map on the $A$-orbits of $X$, it
is $G\times H$-equivariant.
\end{nothing}

\begin{nothing} {\bf Transitive $A$-fibered $(G,H)$-bisets.}\quad
Let $X$ be an $A$-fibered $(G,H)$-biset. It is clear that the $A\times G\times
H$-action on $X$ is transitive if and only if the $G\times
H$-action on the set of $A$-orbits of $X$ is transitive. In this case
we call $X$ a {\em transitive} $A$-fibered $(G,H)$-biset. There exists
a bijective correspondence between the isomorphism classes of
transitive $A$-fibered $(G,H)$-bisets and the $G\times H$-conjugacy
classes of $\calM_{G\times H}$. We describe this correspondence. If
$X$ is a transitive $A$-fibered $(G,H)$-biset, choose an element $x\in
X$ and associate to $X$ the class $[S_x,\phi_x]_{G\times H}$ as
defined in (\ref{stabilizing pair}). Conversely, given a pair
$(U,\phi)$ we construct an $A$-fibered $(G,H)$-biset $X$ by
$X:=(A\times G\times H)/U_\phi$, where $U_\phi$ is the subgroup of
$G\times H \times A$ consisting of all elements $(\phi(u^{-1}), u)$,
with $u\in U$. If we start with the conjugate pair
$\lexp{(g,h)}{(U,\phi)}$ instead of $(U,\phi)$, we obtain the conjugate subgroup
$\lexp{(1,g,h)}{U_\phi}$ of $G\times H\times A$ and therefore, we
obtain isomorphic $A$-fibered $(G,H)$-bisets. It is easy to see that
these two constructions are mutually inverse.

For $(U,\phi)\in\calM_{G\times H}$ we denote the corresponding
transitive $A$-fibered $(G,H)$-biset and its isomorphism class by
\begin{equation*}
  \left(\frac{G\times H}{U,\phi}\right)\quad\text{and}\quad
  \left[\frac{G\times H}{U,\phi}\right]\,,
\end{equation*}
respectively.
\end{nothing}

\begin{nothing} {\bf Opposite of an $A$-fibered biset.}\quad
Let $X$ be an $A$-fibered $(G,H)$-biset. We define the {\em opposite}
$\Xop\in\HsetGA$ of $X$ as the $A$-fibered $(H,G)$-biset which has $X$ as underlying set, but endowed with the left $A\times H\times G$-action given by
\begin{equation*}
  (a,h,g)x:=(a^{-1},g,h)x\,.
\end{equation*}
In particular, if $X=\left(\frac{G\times H}{U,\phi}\right)$ then
$\Xop\cong \left(\frac{H\times G}{\Uop,\phiop}\right)$ with
\begin{equation*}
  \Uop:=\{(h,g)\in H\times G\mid (g,h)\in U\}\quad\text{and}\quad
  \phiop(h,g):=\phi(g,h)^{-1}\,.
\end{equation*}
Thus, $k_1(\Uop)=k_2(U)$, $k_2(\Uop)=k_1(U)$, $(\phiop)_1 = \phi_2$, and $(\phiop)_2=\phi_1$.
%If additionally $Y$ is an $A$-fibered $(H,K)$-biset then $(X\cdotH Y)^{\mathrm{op}}\cong \Yop\cdotH\Xop$ as $A$-fibered $(K,G)$-bisets.
\end{nothing}

\begin{nothing}{\bf The Grothendieck group.}\quad
We denote the isomorphism class of an $A$-fibered $(G,H)$-biset $X$ by
$[X]$. Let $X$ and $Y$ be two $A$-fibered $(G,H)$-bisets. The disjoint
union $X\coprod Y$ of $X$ and $Y$ is the categorical coproduct. On the
set of isomorphism classes it induces the structure of a monoid
$[X]+[Y]:=[X\coprod Y]$. The corresponding Grothendieck group is called
the {\em Burnside group} of $A$-fibered $(G,H)$-bisets and is denoted
by $B^A(G,H)$. Every $A$-fibered $(G,H)$-biset $X$ is represented
by a class $[X]$ in $B^A(G,H)$ and the elements
$\left[\frac{G\times H}{U,\phi}\right]$, with $[U,\phi]_{G\times H}\in 
G\times H\backslash\calM_{G\times H}$, form a $\ZZ$-basis of the abelian
group $B^A(G,H)$. We will denote again by $-^{\mathrm{op}}\colon B^A(G,H)\to B^A(H,G)$ the isomorphism induced by taking opposite fibered bisets. With this we have a canonical isomorphism $B^A(G,H)\cong B^A(H,G)$ of abelian groups.

As a special case, when $H$ is the trivial group we obtain the Grothendieck group $B^A(G)$ of the category $\GsetA$ with respect to disjoint unions. 
\end{nothing}

\begin{remark}\label{rem kG-mon}  (a) The group $B^A(G)$ can be interpreted as the result of the $-_+$ construction, see  \cite{Boltje1998}, applied to the group rings $\ZZ G^*=\ZZ\Hom(G,A)$.

\smallskip
(b) Assume that $A=k^\times$, the unit group of an integral domain $k$. Then $B^A(G)$ can also be interpreted as the Grothendieck group of the $k$-linear additive category $\lmon{kG}$ of finite $G$-equivariant line bundles over $k$, which was introduced in \cite{Boltje2001a}. %This category comes with a $k$-linear forgetful functor $\lmon{kG}\to\lmod{kG}$. Replacing $G$ with $G\times H$, one obtains a category $\lmon{kG}_{kH}$ whose Grothendieck group is $B(G,H;A)$ and a $k$-linear functor $\lmon{kG}_{kH}\to\lmod{kG}_{kH}$. With respect to the latter functor and the homomorphism it induces between $B(G,H;A)$ and the Grothendieck group of $\lmod{kG}_{kH}$, the functor $-^{\mathrm{op}}\colon \GsetHA\to\HsetGA$ is a lift of the $k$-dual functor $\lmod{kG}_{kH}\to\lmod{kH}_{kG}$, $M\mapsto \Hom_k(M,k)$. Moreover, the tensor product bifunctor $-\otimes_{AH}-\colon \GsetHA\times \HsetKA\to \GsetKA$ which we will define in the following section is a lift of the tensor product functor $-\otimes_{kH}-\colon \lmod{kG}_{kH}\times\lmod{kH}_{kK}\to \lmod{kG}_{kK}$.
\end{remark}

%%%%%%%%%%%%%%%% SECTION 2 %%%%%%%%%%%%%%%%%%%%%%%%%%%%%%%%%

\section{Tensor product of fibered bisets}\label{sec tensor product}

In this section, we introduce the tensor product of fibered bisets. The usual product construction of 
bisets does not work in the fibered case, since under this product, the $A$-action may not remain free. The remedy
is to consider the subset of free $A$-orbits in the product. We will call this construction the tensor product. After establishing associativity of the tensor product, we prove an explicit formula (see Corollary~\ref{Mackey formula for fibered bisets 2}), called the Mackey formula, for the tensor product of two transitive $A$-fibered bisets, and a first level decomposition of a transitive fibered biset into a product of standard (fibered) bisets (see Proposition~\ref{decomposition}). The latter we call the standard decomposition. The middle factor of this decomposition will be decomposed further in Section~\ref{sec completely reduced pairs}. Both the Mackey formula and the standard decomposition are completely analogous to formulas in the biset category, see \cite[Section~2.3]{BoucBook}.

\begin{nothing} {\bf The tensor product.}\quad
Given finite groups $G,H,K$ and objects $X\in\GsetHA$ and
$Y\in\HsetKA$, we will define their {\em tensor product}
$X\otimes_{AH} Y\in\GsetKA$. First recall the definition of the product
$X\times_{AH}Y$ of the right $A\times H$-set $X$ with the left
$A\times H$-set $Y$:  It is the set of $A\times H$-orbits of $X\times
Y$ under the action
\begin{equation*}
  \lexp{(a,h)}{(x,y)}= (x(a^{-1},h^{-1}),(a,h)y)
\end{equation*}
for $(a,h)\in A\times H$ and $(x,y)\in X\times Y$. We will denote the
$A\times H$-orbit of $(x,y)$ by $[x,y]_{AH}$. Note that we have $[xa,y]_{AH}=[x,ay]_{AH}$ and
$[xh,y]_{AH}=[x,hy]_{AH}$ for $a\in A$ and $h\in H$.

The set $X\times_{AH}Y$ is an $A$-set via $a[x,y]_{AH}:=[ax,y]_{AH}=[x,ay]_{AH}$ and
it is a $(G,K)$-biset via $(g,k)[x,y]_{AH}:=[gx,yk^{-1}]_{AH}$. These two
actions commute so that $X\times_{AH}Y$ is a left $A\times G\times
K$-set. However, in general the $A$-action is not free. Note that the
action of $G\times K$ permutes the free $A$-orbits of $X\times_{AH}Y$.
This allows us to define
\begin{equation*}
  X\otimes_{AH} Y\in\GsetKA
\end{equation*}
as the union of the free $A$-orbits of $X\times_{AH} Y$. If the
stabilizer of $[x,y]_{AH}$ in $A$ is trivial, i.e., if $[x,y]_{AH}\in
X\otimes_{AH}Y$, we will write $x\otimes_{AH}y$ instead of
$[x,y]_{AH}$. Note that the construction $X\otimes_{AH} Y$ is
functorial in $X$ and $Y$. It is clear from the definitions that the
tensor product respects disjoint unions:
\begin{align*}
  (X\coprod X')\otimes_{AH}Y
  & \cong (X\otimes_{AH}Y)\coprod (X'\otimes_{AH}Y)\,,\\
  X\otimes_{AH}(Y\coprod Y')
  & \cong (X\otimes_{AH}Y)\coprod (X\otimes_{AH}Y')\,,
\end{align*}
for $X,X'\in\GsetHA$ and $Y,Y'\in\HsetKA$. It is also straightforward to verify that
\begin{equation*}
  (X\otimes_{AH} Y)^{\mathrm{op}} \cong \Yop\otimes_{AH}\Xop\,,
\end{equation*}
as $A$-fibred $(K,G)$-bisets, under the map $[x,y]\mapsto [y,x]$.
\end{nothing}

\begin{nothing} {\bf Associativity.}\quad
Let $G,H,K,L$ be finite groups, let $X\in\GsetHA$, $Y\in\HsetKA$ and
$Z\in\KsetLA$ be $A$-fibered bisets. We will show that there exists an
isomorphism
\begin{equation*}
  (X\otimes_{AH} Y)\otimes_{AK} Z \cong X\otimes_{AH}(Y\otimes_{AK}Z)
\end{equation*}
which maps $(x\otimes_{AH} y)\otimes_{AK} z$ to
$x\otimes_{AH}(y\otimes_{AK} z)$ for $x\in X$, $y\in Y$ and $z\in Z$.
It is well-known that the map
\begin{align}\label{associativity isomorphism}
  (X\times_{AH} Y)\times_{AK} Z
  & \to X\times_{AH} (Y\times_{AK} Z)\,,\\
  \notag
  \bigl[[x,y]_{AH},z\bigr]_{AK}
  & \mapsto \bigl[x,[y,z]_{AK}\bigr]_{AH}\,,
\end{align}
is a bijection. It is also clear that it is an isomorphism of $A\times
G\times L$-sets. Since $(X\otimes_{AH}Y)\otimes_{AK}Z$ is a subset of
the left hand side and $X\otimes_{AH}(Y\otimes_{AK}Z)$ is a subset of
the right hand side, it suffices to show that the above isomorphism
restricts to these subsets. Note that $[[x,y]_{AH},z]_{AK}\in
(X\otimes_{AH}Y)\otimes_{AK}Z$ if and only if the stabilizer $S_1$ of
$[x,y]_{AH}$ in $A$ is trivial and the stabilizer $S_2$ of
$[[x,y]_{AH},z]_{AK}$ in $A$ is trivial. Since $S_1\le S_2$, this is
equivalent to the statement that $S_2=1$. Similarly,
$[x,[y,z]_{AK}]_{AH}\in X\otimes_{AH}(Y\otimes_{AK} Z)$ if and only if
the stabilizer $T_2$ of $[x,[y,z]_{AK}]_{AH}$ in $A$ is trivial. But
since the map in (\ref{associativity isomorphism}) is an
$A$-equivariant isomorphism, $S_2=T_2$, and the proof is complete.
\end{nothing}

\begin{nothing}\mylabel{*notation}
Our next aim is to find an explicit formula for the tensor product of
two transitive $A$-fibered bisets. For this purpose, we will need the
following notation. Let $G$, $H$ and $K$ be finite groups, let $U\le
G\times H$ and $V\in H\times K$ be subgroups, and let $\phi\in U^*$
and $\psi\in V^*$ be homomorphisms satisfying
\begin{equation}\label{compatibility condition}
  \phi_2|_{k_2(U)\cap k_1(V)} = \psi_1|_{k_2(U)\cap k_1(V)}.
\end{equation}
Following \cite{Bouc1996a}, we set
\begin{equation*}
  U*V:=\{(g,k)\in G\times K \mid \text{there exists $h\in H$ with $(g,h)\in
  U$ and $(h,k)\in V$}\}\,.
\end{equation*}
Moreover, we define $\phi*\psi\in(U*V)^*$ by
\begin{equation*}
  (\phi*\psi)(g,k):=\phi(g,h)\psi(h,k)\,,
\end{equation*}
where $h\in H$ is chosen such that $(g,h)\in U$ and $(h,k)\in V$, a construction that has also been used in \cite{Bouc2010c}.
Note that, by the condition in (\ref{compatibility condition}), this does not depend
on the choice of $h\in H$.
\end{nothing}

For the following proposition, we fix transitive $A$-fibered bisets
$X\in\GsetHA$ and $Y\in\HsetKA$, and also elements $x\in X$ and $y\in Y$.
Let $(S_x,\phi_x)\in\calM_{G\times H}$ and
$(S_y,\phi_y)\in\calM_{H\times K}$ denote their stabilizing pairs.
Also, set $H_x:=p_2(S_x)$ and $H_y:=p_1(S_y)$. Note that every
$A\times G\times K$-orbit of $X\times_{AH}Y$ has an element of the
form $[x,hy]_{AH}$, with $h\in H$.

\begin{proposition}\mylabel{Mackey formula for fibered bisets 1}
Assume the above notation.

\smallskip
{\rm (a)} Let $t,t' \in H$. The elements $[x,ty]_{AH}$ and
$[x,t'y]_{AH}$ belong to the same $A\times G\times K$-orbit of
$X\times_{AH}Y$ if and only if $H_xtH_y=H_xt'H_y$.

\smallskip
{\rm (b)} Let $t\in H$. The stabilizer of $[x,ty]_{AH}$ in $A$ is
trivial if and only if
\begin{equation*}
  \phi_{x,2}|_{H_t}=\lexp{t}{\phi_{y,1}}|_{H_t}\,,
\end{equation*}
where $H_t:=k_2(S_x)\cap\lexp{t}{k_1(S_y)}$.

\smallskip
{\rm (c)} The elements $[x,ty]_{AH}\in X\times_{AH} Y$, where $t$ runs
through a set of representatives of the double cosets $H_x\backslash
H/H_y$ such that $\phi_{x,2}|_{H_t}=\lexp{t}{\phi_{y,1}}|_{H_t}$, form
a set of representatives for the $A\times G\times K$-orbits of
$X\otimes_{AH} Y$. For $t\in H$ with
$\phi_{x,2}|_{H_t}=\lexp{t}{\phi_{y,1}}|_{H_t}$, the stabilizing pair of
$[x,ty]_{AH}$ is equal to
\begin{equation*}
  (S_x*\!\lexp{(t,1)}{S_y},\ \phi_x*\!\lexp{(t,1)}{\phi_y})\,.
\end{equation*}
\end{proposition}

\begin{proof}
(a) This follows from the following chain of equivalences:
\begin{align*}
  & \text{$[x,ty]_{AH}$ and $[x,t'y]_{AH}$ belong to the same
  $A\times G\times K$-orbit}\\
  \iff & \exists g\in G, k\in K, a\in A\colon [x,t'y]_{AH} =
  [agx,tyk]_{AH} \\
  \iff & \exists g\in G, k\in K, a,b\in A, h\in H\colon
  (x,t'y)=(agxh^{-1}b^{-1},hbtyk)\\
  \iff & \exists g\in G, k\in K, a,b\in A, h\in H\colon
    (g,h)\in S_x, ({t'}^{-1}ht,k^{-1})\in S_y,\\
    &\qquad    \phi_x(g,h)=a^{-1}b, \phi_y({t'}^{-1}ht,k^{-1})=b^{-1}\\
  \iff & \exists g\in G, h\in H, k\in K\colon (g,h)\in S_x,
  ({t'}^{-1}ht,k^{-1})\in S_y \\
  \iff & \exists h\in H \colon h\in H_x, {t'}^{-1}ht\in H_y\\
  \iff & \exists h\in H_x \colon t\in ht'H_y\\
  \iff & t\in H_xt'H_y\,.
\end{align*}

\smallskip
(b) Let $a\in A$. Then we have the following chain of equivalences:
\begin{align*}
  & [x,ty]_{AH} = [x,aty]_{AH} \\
  \iff & \exists h\in H, b\in A\colon (x,ty)=(xh^{-1}b^{-1},bhaty)\\
  \iff & \exists h\in H, b\in A\colon h\in k_2(S_x), t^{-1}ht\in
  k_1(S_y)\\
  & \qquad \phi_{x,2}(h^{-1})=b, \phi_{y,1}(t^{-1}ht)=a^{-1}b^{-1}\\
  \iff & \exists h\in k_2(S_x)\cap\lexp{t}{k_1(S_y)} \colon
  \lexp{t}{\phi_{y,1}}(h)=a^{-1}\phi_{x,2}(h)\\
  \iff & a \in\im\bigr(\phi_{x,2}|_{H_t}\cdot
  \lexp{t}{\phi_{y,1}^{-1}}|_{H_t}\bigr)\,.
\end{align*}
Thus, the stabilizer of $[x,ty]_{AH}$ in $A$ is trivial if and only if
$\phi_{x,2}|_{H_t}=\lexp{t}{\phi_{y,1}}|_{H_t}$.

\smallskip
(c) With the results from (a) and (b), we only have to verify the
statement about the stabilizing pair of $[x,ty]_{AH}$. Let
$(W,\lambda)\in\calM_{G\times K}$ denote the stabilizing pair of
$[x,ty]_{AH}$, and let $(g,k)\in G\times K$ and $a\in A$. Then we have
the chain of equivalences
\begin{align*}
  & \text{$(g,k)\in W$ and $\lambda(g,k)=a$}\\
  \iff &  [gx,tyk^{-1}]_{AH}=[x,aty]_{AH} \\
  \iff & \exists b\in A, h\in H \colon (gxh^{-1}b^{-1},hbtyk^{-1})=
  (x,aty)\\
  \iff & \exists b\in A, h\in H \colon (g,h)x=bx,
  (t^{-1}ht,k)y=ab^{-1}y\\
  \iff & \text{$(g,k)\in S_x*\!\lexp{(t,1)}{S_y}$ and there exists $h\in H$
  such that}\\
  & \qquad \text{$(g,h)\in S_x$, $(h,k)\in\lexp{(t,1)}{S_y}$ and
  $\lexp{(t,1)}{\phi_y}(h,k)=a\phi_x(g,h)^{-1}$}\\
  \iff & \text{$(g,k)\in S_x*\lexp{(t,1)}{S_y}$ and $(\phi_x * \lexp{(t,1)}{\phi_y})(g,k) = a$.} 
\end{align*}
Thus, $W=S_x*\lexp{(t,1)}{S_y}$ and $\lambda=\phi_x * \lexp{(t,1)}{\phi_y}$.
\end{proof}

The following Mackey formula is completely analogous to the formula for bisets, see \cite[Lemma~2.3.24]{BoucBook}.

\begin{corollary}[Mackey formula]\mylabel{Mackey formula for fibered bisets 2}
Let $(U,\phi)\in\calM_{G\times H}$ and $(V,\psi)\in\calM_{H\times
K}$. There exists an isomorphism of $A$-fibered $(G,K)$-bisets,
\begin{equation*}
  \left(\frac{G\times H}{U,\phi}\right)\otimes_{AH}
  \left(\frac{H\times K}{V,\psi}\right) \cong
  \coprod_{\substack{t\in p_2(U)\backslash H/p_1(V)\\
  \phi_2|_{H_t}=\lexp{t}{\psi_1}|_{H_t}}}
  \left(\frac{G\times K}{U*\lexp{(t,1)}{V},\phi*\lexp{(t,1)}{\psi}}\right)\,,
\end{equation*}
where $H_t:=k_2(U)\cap \lexp{t}{k_1(V)}$. The above isomorphism maps
the trivial coset of the $t$-component of the right hand side to the
element $U_\phi\otimes_{AH}(t,1,1)V_\psi$.
\end{corollary}

\begin{proof}
This is an immediate consequence of Proposition \ref{Mackey formula
for fibered bisets 1}(c) by choosing $x=U_\phi$ and $y=U_\psi$ to be
the trivial cosets of $(G\times H\times A)/U_\phi$ and $(H\times
K\times A)/U_\psi$, respectively.
\end{proof}

The proof of the following proposition is straightforward and is left
to the reader, see also \cite[Lemma~2.3.22]{BoucBook} for Parts (a) and (c).

\begin{proposition}\mylabel{first consequences}
Let $(U,\phi)\in\calM_{G\times H}$ and $(V,\psi)\in\calM_{H\times
K}$. 

\smallskip
{\rm (a)} One has
\begin{equation*}
  k_1(U)\le k_1(U*V)\le p_1(U*V)\le p_1(U)\quad\text{and}\quad
  k_2(V)\le k_2(U*V)\le p_2(U*V)\le p_2(V)
\end{equation*}

\smallskip
{\rm (b)} Assume that the restrictions of $\phi_2$ and $\psi_1$ to $k_2(U)\cap k_1(V)$ coincide. Then
\begin{equation*}
  l_0(U,\phi) \le l_0(U*V,\phi*\psi)
  \quad\text{and}\quad
  r_0(V,\psi) \le r_0(U*V,\phi*\psi) \,.
\end{equation*}

\smallskip
{\rm (c)} Let $\eta_U\colon p_2(U)/k_2(U)\to p_1(U)/k_1(U)$ be the
isomorphism from Proposition~{\rm \ref{eta and zeta}(c)}. Then one has
\begin{equation*}
  k_1(U*V)/k_1(U) \cong \eta_U\bigr((p_2(U)\cap
  k_1(V)k_2(U))/k_2(U)\bigr)\,.
\end{equation*}

\smallskip
{\rm (d)} If $r(U,\phi)=l(V,\psi)$ then $l(U*V,\phi*\psi)=l(U,\phi)$ and $r(U*V,\phi*\psi)=r(V,\psi)$.
\end{proposition}

Next we want to show that every transitive $A$-fibered $(G,H)$-biset
$X$ is the tensor product of the canonical operations restriction,
deflation, inflation, induction (which themselves are bisets) and a
transitive $A$-fibered $(\Gbar,\Hbar)$-biset $Y$ for certain sections
$\Gbar$ and $\Hbar$ of $G$ and $H$, respectively. First we need the following notations from \cite{Bouc1996a}.

\begin{notation}\mylabel{4 operations}
(a) Let $G$ and $H$ be finite groups. For any subgroups $H_2\le H_1\le H$ and any group homomorphism
$f\colon H_1\to G$, we set
\begin{equation*}
  \lDelta{f}(H_2):=\{(f(h),h)\mid h\in H_2\}\le G\times H\,.
\end{equation*}
For any subgroups $G_2\le G_1\le G$ and any group
homomorphism $f\colon G_1\to H$, we set
\begin{equation*}
  \Delta_f(G_2):=\{(g,f(g))\mid g\in G_2\}\le G\times H\,.
\end{equation*}
Moreover, if $f$ is the inclusion map of a subgroup, we only write
$\Delta(H_2)$, resp.~$\Delta(G_2)$.

\smallskip
(b) For a finite group $G$, a subgroup $H$ of $G$ and a normal subgroup
$N$ of $G$ one defines the objects
\begin{gather*}
  \Ind_H^G:=\left(\frac{G\times H}{\Delta(H),1}\right) \in \GsetHA\,,
  \quad
  \Res^G_H:=\left(\frac{H\times G}{\Delta(H),1}\right) \in
  \HsetGA\,,\\
  \Inf_{G/N}^G:=\left(\frac{G\times G/N}{\Delta_{\pi}(G),1}\right)
  \in\lset{G}_{G/N}^A\,,\quad
  \Def_{G/N}^G:=\left(\frac{G/N\times G}{\lDelta{\pi}(G),1}\right)
  \in\lset{G/N}_G^A\,,
\end{gather*}
where $\pi\colon G\to G/N$ denotes the canonical epimorphism. These
objects are called {\em induction, restriction, inflation} and {\em
deflation}.
\end{notation}

The following proposition is a straightforward verification, using the
explicit formula in Corollary~\ref{Mackey formula for fibered bisets
2}. It is completely analogous to the proof of Lemma~3 in
\cite[Section~3]{Bouc1996a} or \cite[Lemma~2.3.26]{BoucBook}.

\begin{proposition}\mylabel{decomposition}
Let $G$ and $H$ be finite groups and let $(U,\phi)\in\calM_{G\times
H}$. Then, with the notation from Proposition~\ref{eta and zeta},
one has the decomposition
\begin{equation*}
  \left(\frac{G\times H}{U,\phi}\right) \cong
  \Ind_P^G \otimes_{AP}
  \Inf_{P/\Khat}^P \otimes_{A(P/\Khat)}
  X \otimes_{A(Q/\Lhat)}
  \Def_{Q/\Lhat}^Q \otimes_{AQ}
  \Res^H_Q\,,
\end{equation*}
where
\begin{equation*}
  X=\left(\frac{P/\Khat \times Q/\Lhat}
  {U/(\Khat\times\Lhat),\phibar}\right)\,.
\end{equation*}
\end{proposition}

In the above proposition the homomorphism $\phibar\colon
U/(\Khat\times\Lhat)\to A$ is induced by $\phi$ and the group
$U/(\Khat\times\Lhat)$ is viewed as a subgroup of
$(P/\Khat)\times(Q/\Lhat)$ via the canonical isomorphism
\begin{equation*}
  \can\colon  (P\times Q)/(\Khat\times\Lhat) \to
  (P/\Khat) \times (Q/\Lhat)\,.
\end{equation*}

Note that setting $\Gbar:=P/\Khat$, $\Hbar:=Q/\Lhat$,
$\Ubar:=\can(U/(\Khat\times\Lhat))\le \Gbar\times \Hbar$, and
$\phibar\colon\Ubar\to A$ (using the above isomorphism), we have
\begin{equation*}
  p_1(\Ubar)=\Gbar\,,\quad \ker(\phibar_1)=\{1\}\,, \quad
  p_2(\Ubar)=\Hbar\,,\quad \ker(\phibar_2)=\{1\}\,.
\end{equation*}

%%%%%%%%%%%%%%%%%%%%%% SECTION 3 %%%%%%%%%%%%%%%%%%%%%%%%%%%%%%%%%%%%%

\section{$A$-fibered Biset Functors}\label{sec fibered biset functors}

Throughout this section, let $k$ denote a commutative ring. In this section we recall and use the
general results in \cite[Sections~2 and 4]{Bouc1996a} on functor
categories, and we specialize the approach in
\cite[Section~3]{Bouc1996a} to our situation.

\begin{definition}
(a) By $\calC=\calCkA$ we denote the following category. Its objects
are the finite groups. For finite groups $G$ and $H$, we set
\begin{equation*}
  \Hom_{\calC}(G,H):= B_k^A(H,G):=k\otimes B^A(H,G)\,.
\end{equation*}
If also $K$ is a finite group then composition in $\calC$ is defined by
\begin{equation*}
-\cdotH-\colon B_k^A(K,H) \times B_k^A(H, G) \to
  B_k^A(K, G)\,,\quad (y,x)\mapsto y\cdotH x\,,
\end{equation*}
the $k$-linear extension of the bilinear map induced by taking the
tensor product of $A$-fibered bisets. The identity element of $G$ is the element $\left[\frac{G\times G}{\Delta(G),1}\right]$.

(b) An {\em $A$-fibered biset functor} over $k$ is a $k$-linear functor
from the $k$-linear category $\calCkA$ to the $k$-linear category
$\kMod$ of left $k$-modules. By $\calF=\calFkA$ we denote the
category of $A$-fibered biset functors over $k$. Their morphisms are
the natural transformations. Since the category $\kMod$ is abelian,
also the category $\calFkA$ is abelian, with point-wise
constructions of kernels, cokernels, etc. As explained in Section~2 of
\cite{Bouc1996a}, this allows to define subfunctors, quotient
functors, simple functors, projective functors, etc. If $M\in\calF$, $m\in M(H)$ and $b\in B_k^A(G,H)$, we will usually write $bm$ instead of $(\calF(b))(m)$.
\end{definition}

\begin{remark}\label{rem isomorphisms in C}
(a) Two finite groups $G$ and $H$ are isomorphic as groups if and only if they are isomorphic in $\calC$. In fact, if $\phi\colon G\to H$ is a group isomorphism then $\left[\frac{G\times H}{\Delta_\phi(G),1}\right]$ is an isomorphism between $G$ and $H$ in $\calC$. Conversely, assume $a\in B_k^A(G,H)$ and $b\in B_k^A(H,G)$ satisfy $ab=1$ and $ba=1$. Then, the equation $ab=1$ and the Mackey formula imply that there exist standard basis elements $\left[\frac{G\times H}{U,\phi}\right]$ and $\left[\frac{H\times G}{V,\psi}\right]$ such that $U*V=\Delta(G)$. Proposition~\ref{first consequences}(a) then implies that $k_1(U)=\{1\}$ and $p_1(U)=G$. Now Proposition~\ref{eta and zeta}(c) implies that $G$ is isomorphic to a section of $H$. Similarly, $ba=1$ implies that $H$ is isomorphic to a section of $G$. Thus, $G$ and $H$ are isomorphic groups.

\smallskip
(b) By Part~(a), one may replace $\calC$ with a full subcategory as long as every isomorphism type of finite groups is represented and one obtains a functor category that is equivalent to $\calF$. Thus, we may assume that $\mathrm{Ob}(\calC)$ is a set, without changing the equivalence class of $\calF$.
\end{remark}

The following remark shows that $\calF$ also has some functorial properties and rigidity when changing the abelian group $A$.

\begin{remark}\label{isomorphic calFs}
(a) If $f\colon A\to A'$ is a homomorphism of abelian groups then one obtains an induced $k$-linear homomorphism $B^A_k(G,H)\to B_k^{A'}(G,H)$ for any two finite groups $G$ and $H$. Moreover, these homomorphisms induce a $k$-linear functor $\calC_k^A\to\calC_k^{A'}$, and by restriction along this functor a $k$-linear functor $\calF_k^{A'}\to\calF_k^A$ between the associated functor categories. If $f$ is an isomorphism then all these induced $k$-linear homomorphisms and functors are isomorphisms.

\smallskip
(b) The inclusion $\torA\subseteq A$ induces a $k$-linear isomorphism $B_k^{\torA}(G,H)\myiso B_k^A(G,H)$ for any two finite groups $G$ and $H$, and further $k$-linear isomorphisms $\calC_k^{\torA}\to\calC_k^A$ and $\calF_k^A\myiso \calF_k^{\torA}$.
\end{remark}

\begin{nothing}\label{noth LGV} We recall several constructions from
\cite[Section~2]{Bouc1996a}.

\noindent For a finite group $G$ let $E_G=E_k^A(G)$ denote the endomorphism
algebra of $G$ in $\calCkA$, i.e.,
\begin{equation*}
  E_G=E_k^A(G):=\End_{\calC}(G)= B_k^A(G,G)\,.
\end{equation*}
Clearly, $E_G$ is a $k$-algebra, and for any fibered biset functor $F$,
the $k$-module $F(G)$ has the structure of a left $E_G$-module. This
way one obtains a functor, given by evaluation at $G$,
\begin{equation*}
  \calE_G\colon\calFkA\to \lMod{E_G}
\end{equation*}
which maps a functor $F\in\calFkA$ to $F(G)$ and a natural
transformation $\eta\colon F\to F'$ to $\eta_G$. The evaluation
functor $\calE_G$ has a left adjoint which we now describe. To a left
$E_G$-module $V$ one associates the functor $L_{G,V}\in\calFkA$, which is
defined on an object $H$ of $\calC$ by
\begin{equation*}
  L_{G,V}(H):=\Hom_{\calC}(G,H)\otimes_{E_G} V = B_k^A(H,G)\otimes_{E_G} V\,,
\end{equation*}
where $\Hom_{\calC}(G,H)$ is considered as a right $E_G$-module via
composition of morphisms. If also $K$ is an object of $\calC$ and if
$\phi\in\Hom_{\calC}(K,H)$ then $L_{G,V}(\phi)$ is given by
composition with $\phi$ in the first factor of the tensor product.
This way one obtains a functor
\begin{equation*}
  L_{G,-}\colon \lMod{E_G}\to\calFkA
\end{equation*}
which is left adjoint to $\calE_G$.
\end{nothing}

Lemma~1 in \cite{Bouc1996a} specializes in our situation to the
following Lemma.

\begin{lemma}\label{lem J}
Let $V$ be a simple left $E_G$-module. Then the
$A$-fibered biset functor $L_{G,V}$ has a unique maximal subfunctor
$J_{G,V}$. Its evaluation at a finite group $H$ is given by
\begin{equation*}
  J_{G,V}(H) = \left\{ \sum_{i} x_i\otimes v_i \in B_k^A(H,G)\otimes_{E_G} V %L_{G,V}(H) 
  \ \big|\ 
  \forall y\in B_k^A(G, H)\colon \sum_{i} (y\cdotH x_i)(v_i)=0 \right\}\,.
\end{equation*}
Moreover, the simple head $S_{G,V}$ of $L_{G,V}$ satisfies
$S_{G,V}(G)\cong V$.
\end{lemma}

\begin{nothing}\label{noth essential algebra}
{\bf The essential algebra.}\quad
For a finite group $G$ we set
\begin{equation*}
  I_G=I_k^A(G):=\sum_{|H|<|G|} B_k^A(G,H)\cdotH B_k^A(H,G)\subseteq B_k^A(G,G)=E_G\,.
\end{equation*}
The sum runs over all finite groups $H$ of order smaller than $|G|$. Obviously, $I_G$ is an ideal of $E_G$ and we denote by $\Ebar_G:=E_G/I_G$ the factor $k$-algebra. It is called the {\em essential algebra} of $G$.
\end{nothing}

\begin{proposition}\mylabel{simple functor proposition}
{\rm (a)} Let $S\in\calF$ be a simple functor and let $G$ be a finite group
such that $V:=S(G)\neq\{0\}$. Then $V$ is a simple $E_G$-module and
$S\cong S_{G,V}$ in $\calF$.

\smallskip
{\rm (b)} Let $S\in\calF$ be a simple functor and let $G$ be a finite group of smallest order satisfying $S(G)\neq \{0\}$. Then the simple $E_G$-module $S(G)$ is annihilated by $I_G$. In particular, every simple functor $S\in\calF$ is isomorphic to $S_{G,V}$ for some finite group $G$ and some simple $E_G$-module $V$ which is annihilated by $I_G$.

\smallskip
{\rm (c)} Let $G$ be a finite group and let $V$ be a simple $E_G$-module which is annihilated by $I_G$. Then $G$ is a {\em minimal group} for $S_{G,V}$, i.e., $S_{G,V}(G)\neq \{0\}$ and for every finite group $H$ with $S_{G,V}(H)\neq\{0\}$ one has $|G|\le|H|$.

\smallskip
{\rm (d)} For finite groups $G,H$ and simple modules $V\in\lMod{E_G}$ and
$W\in\lMod{E_H}$ one has $S_{G,V}\cong S_{H,W}$ if and only if
$S_{G,V}(H)\cong W$ as $E_H$-modules.
\end{proposition}

\begin{proof}
Part (a) follows from a short argument given at the beginning of
Section~4 in \cite{Bouc1996a}. Part~(b) is immediate from the definitions and Part~(a). To prove Part~(c) assume that $|H|<|G|$. Then the description of $J_{G,V}(H)$ in Lemma~\ref{lem J} implies that $L_{G,V}(H)=J_{G,V}(H)$. Thus, $S_{G,V}(H)= \{0\}$. Finally, Part~(d) follows immediately from Part~(a) and the last statement in Lemma~\ref{lem J}.
\end{proof}

%%%%%%%%%%%%%%%%%%%%%%% SECTION 4 %%%%%%%%%%%%%%%%%%%%%%%%%%%%%%

\section{Idempotents in $E_G$}\label{sec idempotents}

%Proposition~2.4(b) implies that in order to classify the simple functors $S\in\calF$, we need to classify the simple $\Ebar_G$-modules. This goal will be achieved in Section~4. The goal of this section is to introduce and study a  subalgebra $E_G^c$ of $E_G$ which covers $\Ebar_G$.

In this section, we introduce idempotents in $E_G$ that will play an important role later.

\smallskip
Recall that $G$ acts on $\calM_G$ by
conjugation. We will denote the $G$-fixed point set by $\calM_G^G$.
A pair $(K,\kappa)\in\calM_G$ is $G$-fixed if and only if $K$ is a
normal subgroup of $G$ and $\kappa$ is a $G$-stable homomorphism. Note
that if also $H$ is a finite group, $(U,\phi)\in\calM_{G\times H}$
and $p_1(U)=G$ then, by Proposition~\ref{eta and zeta}(a) and (b), one has
$l_0(U,\phi)\in\calM_G^G$.

\begin{definition}\mylabel{idempotent definition 1}
Let $G$ be a finite group and let $(K,\kappa)\in\calM_G^G$. We
define the $A$-fibered $(G,G)$-biset $E_{(K,\kappa)}$ as
\begin{equation*}
  E_{(K,\kappa)}:=\left(\frac{G\times G}{\Delta_K(G),\phi_\kappa}\right)\,,
\end{equation*}
where
\begin{gather*}
  \Delta_K(G):=\{(g_1,g_2)\in G\times G\mid g_1K=g_2K\}= (K\times \{1\})\Delta(G)
  =(\{1\}\times K)\Delta(G)\\
  \text{and}\quad
  \phi_\kappa(g_1,g_2)=\kappa(g_2^{-1}g_1)=\kappa(g_1g_2^{-1})\,.
\end{gather*}
Note that $\Delta_K(G)$ is a subgroup of $G\times G$, since $K$ is
normal in $G$, and that $\phi_\kappa$ is a homomorphism, since
$\kappa$ is $G$-stable. Note also that $E_{(K,\kappa)}^{\mathrm{op}}\cong
E_{(K,\kappa)}$.
\end{definition}

\begin{proposition}\mylabel{idempotent proposition}
Let $G$ and $H$ be finite groups and let
$(K,\kappa),(K',\kappa')\in\calM_G^G$. Moreover, let
$(U,\phi)\in\calM_{G\times H}$ with $l(U,\phi)=(G,K',\kappa')$.

\smallskip
{\rm (a)} One has $l(\Delta_K(G),\phi_\kappa)=(G,K,\kappa)=r(\Delta_K(G),\phi_\kappa)$.

\smallskip
{\rm (b)} One has
\begin{equation*}
  E_{(K,\kappa)}\otimes_{AG}
  \left(\frac{G\times H}{U,\phi}\right) \cong
  \begin{cases}
    \ \emptyset\,, & \text{if $\kappa|_{K\cap K'}\neq \kappa'|_{K\cap
    K'}$,}\\
    \left(\frac{G\times H}{(K\times 1)U,\kappa\cdot\phi}\right)\,, &
    \text{if $\kappa|_{K\cap K'} = \kappa'|_{K\cap K'}$,}
  \end{cases}
\end{equation*}
where $(\kappa\cdot\phi)((k,1)u):=\kappa(k)\phi(u)$ for $k\in K$ and
$u\in U$. Moreover, in the second case, one has
\begin{equation*}
  l((K\times1)U,\kappa\cdot\phi)= (G, KK', \kappa\cdot\kappa')\,.
\end{equation*}
In particular one has
\begin{equation*}
  E_{(K,\kappa)}\otimes_{AG} E_{(K',\kappa')} \cong
  \begin{cases}
    \ \emptyset\,, & \text{if $\kappa|_{K\cap K'}\neq\kappa'|_{K\cap K'}$,}\\
    E_{(KK',\kappa\cdot\kappa')}\,, &
    \text{if $\kappa|_{K\cap K'}=\kappa'|_{K\cap K'}$.}
  \end{cases}
\end{equation*}

\smallskip
{\rm (c)} Assume that $(K,\kappa)\le(K',\kappa')$. Then
\begin{equation*}
  E_{(K,\kappa)}\otimes_{AG}\left(\frac{G\times H}{U,\phi}\right)
  \cong \left(\frac{G\times H}{U,\phi}\right)\,.
\end{equation*}
In particular, one has $E_{(K,\kappa)}\otimes_{AG}E_{(K',\kappa')}
\cong E_{(K',\kappa')}$ and $E_{(K,\kappa)}\otimes_{AG}E_{(K,\kappa)}
\cong E_{(K,\kappa)}$.

\smallskip
{\rm (d)} Assume that $p_2(U,\phi)=H$. Then
\begin{equation*}
  \left(\frac{G\times H}{U,\phi}\right) \otimes_{AH}
  \left(\frac{G\times H}{U,\phi}\right)^{\mathrm{op}} \cong
  E_{(K',\kappa')}
\end{equation*}
in $\GsetGA$.
\end{proposition}

\begin{proof}
Part~(a) is an easy verification. Parts~(b) and (d) follow immediately
from the explicit Mackey formula in Corollary~\ref{Mackey formula for fibered
bisets 2}. Part~(c) is a special case of part~(b).
\end{proof}

\begin{nothing}\mylabel{e's and f's}
For $(K,\kappa)\in\calM_G^G$ we set
\begin{equation*}
  e_{(K,\kappa)}:=[E_{(K,\kappa)}]\in B_k^A(G,G) = E_G\,.
\end{equation*}
Note that $e_{(1,1)}=1\in E_G$. By Proposition~\ref{idempotent proposition}, we have
\begin{equation}\label{product of e's}
  e_{(K,\kappa)}\cdot e_{(L,\lambda)} =
  \begin{cases}
    e_{(KL,\kappa\lambda)}, &\text{if $\kappa|_{K\cap L} =
    \lambda|_{K\cap L}$,}\\
    0, & \text{otherwise,}
  \end{cases}
\end{equation}
for all $(K,\kappa),(L,\lambda)\in\calM_G^G$. This implies that, we
obtain a commutative subalgebra
\begin{equation}\label{algebra generated by e's}
  \bigoplus_{(K,\kappa)\in\calM_G^G} ke_{(K,\kappa)}
\end{equation}
of $E_G$. For $(K,\kappa),(L,\lambda)\in\calM_G^G$,
let $\mu_{(K,\kappa),(L,\lambda)}^{\triangleleft}$ denote the
M\"obius coefficient with respect to the poset $\calM_G^G$. Since
$\kappa$ is determined by $\lambda$ if $(K,\kappa)\le(L,\lambda)$,
this coefficient is equal to the M\"obius coefficient
$\mu_{K,L}^\triangleleft$ of $K$ and $L$ with respect to the poset
of normal subgroups of $G$. For $(K,\kappa)\in\calM_G^G$, we
set
\begin{equation}\label{definition of f's}
  f_{(K,\kappa)}:=\sum_{(K,\kappa)\le(L,\lambda)\in\calM_G^G}
  \mu_{K,L}^\triangleleft e_{(L,\lambda)}
\end{equation}
and obtain by M\"obius inversion that
\begin{equation}\label{e's from f's}
  e_{(K,\kappa)}=\sum_{(K,\kappa)\le(L,\lambda)\in\calM_G^G}
  f_{(L,\lambda)}\,.
\end{equation}
It follows that also the elements $f_{(K,\kappa)}$,
$(K,\kappa)\in\calM_G^G$, form a $k$-basis of the algebra in
(\ref{algebra generated by e's}). Moreover, note that
\begin{equation}\label{sum equals 1}
  \sum_{(K,\kappa)\in\calM_G^G} f_{(K,\kappa)}= e_{(1,1)} =1 \in
  E_G\,.
\end{equation}
The next proposition shows that these basis elements are mutually
orthogonal idempotents.
\end{nothing}

\begin{proposition}\mylabel{e,f-relations 1}
For all $(K,\kappa),(L,\lambda)\in\calM_G^G$ one has
\begin{equation*}
  e_{(K,\kappa)}f_{(L,\lambda)} = f_{(L,\lambda)}e_{(K,\kappa)} =
  \begin{cases}
    f_{(L,\lambda)}, & \text{if $(K,\kappa)\le (L,\lambda)$,}\\
    0, & \text{otherwise,}
  \end{cases}
\end{equation*}
and
\begin{equation*}
  f_{(K,\kappa)}f_{(L,\lambda)} =
  \begin{cases}
    f_{(K,\kappa)}, & \text{if $(K,\kappa)=(L,\lambda)$,}\\
    0, & \text{otherwise.}
  \end{cases}
\end{equation*}
\end{proposition}

\begin{proof}
First note that equations (\ref{product of e's}) and (\ref{definition
of f's}) imply immediately that
$e_{(K,\kappa)}f_{(L,\lambda)}=f_{(L,\lambda)}$ when
$(K,\kappa)\le(L,\lambda)$. Recall also that $e_{(K,\kappa)}$ and
$f_{(L,\lambda)}$ commute.

Next we prove the remaining statements of the proposition by induction
on $d=d(K,\kappa)+d(L,\lambda)$, where $d(K,\kappa)$ is defined as the
largest $n\in\NN_0$ such that there exists a chain
$(K,\kappa)=(K_0,\kappa_0)<\cdots<(K_n,\kappa_n)$ in $\calM_G^G$.

If $d=0$ then $(K,\kappa)$ and $(L,\lambda)$ are maximal in
$\calM_G^G$. Thus, $f_{(K,\kappa)}=e_{(K,\kappa)}$ and
$f_{(L,\lambda)}=e_{(L,\lambda)}$. We only have to show that
$e_{(K,\kappa)}e_{(L,\lambda)}=0$ when $(K,\kappa)\neq(L,\lambda)$. So
assume that $e_{(K,\kappa)}e_{(L,\lambda)}\neq0$. Then equation
(\ref{product of e's}) implies that $\kappa|_{K\cap L}=\lambda|_{K\cap
L}$ and that there exists a pair $(KL,\kappa\lambda)\in\calM_G^G$
with $(K,\kappa)\le(KL,\kappa\lambda)\ge(L,\lambda)$. Since
$(K,\kappa)$ and $(L,\lambda)$ are maximal in $\calM_G^G$, we
obtain $(K,\kappa)=(KL,\kappa\lambda)=(L,\lambda)$.

Now assume that $d>0$ and that the proposition holds for smaller
values of $d$. We first show that if
$f_{(K,\kappa)}f_{(L,\lambda)}\neq0$ then $(K,\kappa)=(L,\lambda)$. In
fact $f_{(K,\kappa)}f_{(L,\lambda)}\neq0$ implies, after expanding
$f_{(K,\kappa)}$ and $f_{(L,\lambda)}$ according to equation
(\ref{definition of f's}) and using equation (\ref{product of e's}),
that $\kappa|_{K\cap L}=\lambda|_{K\cap L}$ and that
\begin{equation*}
  f_{(K,\kappa)}f_{(L,\lambda)} \in
  \bigoplus_{(KL,\kappa\lambda)\le(M,\mu)\in\calM_G^G}
  ke_{(M,\mu)}\,.
\end{equation*}
This implies that $e_{(KL,\kappa\lambda)}f_{(K,\kappa)}f_{(L,\lambda)}
= f_{(K,\kappa)}f_{(L,\lambda)}$ by Equation~(\ref{product of e's}). Assume that
$(K,\kappa)\neq(L,\lambda)$. Then $(K,\kappa)<(KL,\kappa\lambda)$ or $(L,\lambda)<(KL,\kappa\lambda)$, and
by induction we obtain $e_{(KL,\kappa\lambda)}f_{(K,\kappa)}=0$ or $e_{(KL,\kappa\lambda)}f_{(L,\lambda)}=0$. 
In either case we obtain $f_{(K,\kappa)}f_{(L,\lambda)} = e_{(KL,\kappa\lambda)}f_{(K,\kappa)}f_{(L,\lambda)}=0$, a contradiction.

Next we assume $(K,\kappa)=(L,\lambda)$. By the result of the previous paragraph and by
induction we have
\begin{equation*}
  e_{(K,\kappa)} = e_{(K,\kappa)}^2 =
  \Bigl(\sum_{(K,\kappa)\le(K',\kappa')\in\calM_G^G}
  f_{(K',\kappa')}\Bigr)^2
  = f_{(K,\kappa)}^2+\sum_{(K,\kappa)<(K',\kappa')\in\calM_G^G}
  f_{(K',\kappa')}\,.
\end{equation*}
Comparing this with Equation~(\ref{e's from f's}) we obtain
$f_{(K,\kappa)}^2=f_{(K,\kappa)}$.

Finally, Equation (\ref{e's from f's}) implies that
\begin{equation*}
  e_{(K,\kappa)}f_{(L,\lambda)} =
  \sum_{(K,\kappa)\le(K',\kappa')\in\calM_G^G}
  f_{(K',\kappa')}f_{(L,\lambda)}\,.
\end{equation*}
By induction and by what we already proved, the latter sum is equal to
$f_{(L,\lambda)}$ if $(K,\kappa)\le(L,\lambda)$ and equal to $0$
otherwise.
\end{proof}

In Section \ref{sec simple functors} we will need the following lemma.

\begin{lemma}\label{lem delete f}
Let $G$ and $H$ be finite groups and let $(U,\phi)\in\calM_{G\times H}$ with 
$p_1(U)=G$ and $p_2(U)=H$. Set $(K,\kappa):=l_0(U,\phi)$ and $(L,\lambda):=r_0(U,\phi)$. Then
\begin{equation*}
  \left[\frac{G\times H}{U,\phi}\right]\cdotH f_{(L,\lambda)} 
  = f_{(K,\kappa)}\cdotG \left[\frac{G\times H}{U,\phi}\right]\cdotH f_{(L,\lambda)} 
  = f_{(K,\kappa)}\cdotG \left[\frac{G\times H}{U,\phi}\right]
\end{equation*}
\end{lemma}

\begin{proof}
We only prove the first equation. The second one follows by a similar argument or by applying 
$-^{\mbox{\small op}}$. Set $x:=\left[\frac{G\times H}{U,\phi}\right]$. Note that $e_{(K,\kappa)} x = x$ (see~Proposition~\ref{idempotent proposition}(c)). Using the definition of $f_{(K,\kappa)}$, it suffices to show that $e_{(K',\kappa')} xf_{(L,\lambda)}=0$ for all $(K',\kappa')\in\calM_G^G$ with $(K',\kappa')>(K,\kappa)$. By Proposition~\ref{idempotent proposition}(b), we have $e_{(K',\kappa')}x = \left[\frac{G\times H}{V,\psi}\right]$ for some $(V,\psi)\in\calM_{G\times H}$ satisfying $l(V,\psi)=(G,K',\kappa')$ and $r(V,\psi)=(H,L',\lambda')$ for some $(L',\lambda')\in\calM_H^H$ with $(L,\lambda)\le (L',\lambda')$. By Proposition~\ref{eta and zeta}(c), we have $G/K\cong H/L$ and also $G/K'\cong H/L'$. Since $K<K'$, this implies $L<L'$. Thus, $e_{(K',\kappa')}xf_{(L,\lambda)}=e_{(K',\kappa')}xe_{(L',\lambda')}f_{(L,\lambda)} = 0$ by Proposition~\ref{e,f-relations 1}, and the proof is complete.
\end{proof}

%%%%%%%%%%%%%%%%%%% SECTION 5 %%%%%%%%%%%%%%%%%%%%%%%%%%%%%%%%%

\section{Linkage}\label{sec linkage}

In this section, we introduce an equivalence relation, that we call linkage, on the set $\mathcal M_G^G$ of 
$G$-fixed elements in $\mathcal M_G$. The linkage classes will be used in later sections, especially in the
determination of the parametrizing set for the simple fibered biset functors.

\begin{definition}\mylabel{linkage definition}
(a) Let $G$ and $H$ be finite groups and let $(K,\kappa)\in\calM_G^G$
and $(L,\lambda)\in\calM_H^H$. We say that $(G,K,\kappa)$ and
$(H,L,\lambda)$ are {\em linked} if there exists
$(U,\phi)\in\calM_{G\times H}$ with $l(U,\phi)=(G,K,\kappa)$ and
$r(U,\phi)=(H,L,\lambda)$. In this case we also write
$(G,K,\kappa)\mathop{\sim}\limits_{(U,\phi)}(H,L,\lambda)$ or just
$(G,K,\kappa)\sim(H,L,\lambda)$. Note that if also $I$ is a finite
group and $(M,\mu)\in\calM_I^I$ and if $(V,\psi)\in\calM_{H\times
I}$ is such that
$(H,L,\lambda)\mathop{\sim}\limits_{(V,\psi)}(I,M,\mu)$ then
$(G,K,\kappa)\mathop{\sim}\limits_{(U*V,\phi*\psi)}(I,M,\mu)$.
Therefore, the relation $\sim$ is an equivalence relation. 

\smallskip
(b) As a special case we also say that elements $(K,\kappa)$ and $(K',\kappa')$
of $\calM_G^G$ are {\em $G$-linked} if
$(G,K,\kappa)\sim(G,K',\kappa')$. We use again the
notation $(K,\kappa)\sim_G(K',\kappa')$ or just
$(K,\kappa)\sim(K',\kappa')$. Note that if
$(K,\kappa)\sim_G(K',\kappa')$ then $G/K\cong G/K'$ by
Proposition~\ref{eta and zeta} and therefore, $|K|=|K'|$. We write
$\{K,\kappa\}_G$ for the $G$-linkage class of $(K,\kappa)$.
\end{definition}

\begin{nothing}\mylabel{cocycle notation}
Let $G$ be a finite group and let $K\le Z(G)$. Let $\sigma\colon
G/K\to G$ be a section of the canonical epimorphism $\pi\colon G\to
G/K$ (i.e., $\pi\circ\sigma=\id_{G/K}$). Then $\sigma$ defines a
$2$-cocycle $\alpha\in Z^2(G/K,K)$ by the 
equation
\begin{equation}\label{cocycle section relation}
  \sigma(x)\sigma(y) = \alpha(x,y)\sigma(xy)\,,\quad \text{for $x,y\in
  G/K$.}
\end{equation}
When $\sigma$ runs through all possible sections of $\pi$ then
$\alpha$ runs through a full cohomology class $[\alpha]$ in
$H^2(G/K,K)$. We say that $\alpha$ {\em describes} the central
extension $1\to K\to G\to G/K\to 1$.
\end{nothing}

\begin{proposition}\mylabel{linkage criterion}
Let $G$ and $H$ be finite groups and let $(K,\kappa)\in\calM_G^G$
and $(L,\lambda)\in\calM_H^H$. Assume that $\kappa$ and $\lambda$
are faithful and let $\alpha\in Z^2(G/K,K)$ and $\beta\in Z^2(H/L,L)$
be cocycles which describe the central extensions $1\to K\to G\to
G/K\to 1$ and $1\to L\to H\to H/L\to 1$, respectively. Then the
following are equivalent:

\smallskip
{\rm (i)} $(G,K,\kappa)\sim(H,L,\lambda)$.

\smallskip
{\rm (ii)} There exists an isomorphism $\eta\colon H/L\to G/K$ such
that
\begin{equation*}
  [\kappa\circ\alpha]=[\lambda\circ\beta\circ(\eta^{-1}\times\eta^{-1})]
\end{equation*}
as elements in $H^2(G/K,\torA)$.
\end{proposition}

\begin{proof}
Let $\sigma\colon G/K\to G$ and $\tau\colon H/L\to H$ be sections of
the canonical epimorphisms such that
\begin{equation*}
  \sigma(x)\sigma(y) = \alpha(x,y)\sigma(xy)\quad\text{and}\quad
  \tau(w)\tau(z) = \beta(w,z)\tau(wz)
\end{equation*}
for all $x,y\in G/K$ and $w,z\in H/L$.

We first show that (i) implies (ii). Let
$(U,\phi)\in\calM_{G\times H}$ with $l(U,\phi)=(G,K,\kappa)$ and
$r(U,\phi) = (H,L,\lambda)$. Then we obtain an isomorphism $\eta\colon
H/L\to G/K$ from Proposition~\ref{eta and zeta}(c). We set
$\tau':=\tau\circ\eta^{-1}\colon G/K\to H$ and
$\beta':=\beta\circ(\eta^{-1}\times\eta^{-1})\in Z^2(G/K,L)$. For all $x,y\in G/K$
we obtain
\begin{align*}
  &  \kappa\bigl(\alpha(x,y)\bigr)
  \lambda\bigl(\beta'(x,y)\bigr)^{-1}\\
  = & \kappa\bigl(\sigma(x)\sigma(y)\sigma(xy)^{-1}\bigr)
  \lambda\bigl(\tau'(x)\tau'(y)\tau'(xy)^{-1}\bigr)^{-1}\\
  = & \phi\Bigl(\sigma(x)\sigma(y)\sigma(xy)^{-1},
  \tau'(x)\tau'(y)\tau'(xy)^{-1}\Bigr)\\
  = & \phi\bigl(\sigma(x),\tau'(x)\bigr)
  \phi\bigl(\sigma(y),\tau'(y)\bigr)
  \phi\bigl(\sigma(xy),\tau'(xy)\bigr)^{-1}\\
  = & \mu(x)\mu(y)\mu(xy)^{-1}\,,
\end{align*}
where $\mu\colon G/K\to\torA$ is defined by
$\mu(x)=\phi(\sigma(x),\tau'(x))$. Note here that, by the definition of
$\eta$ and $\tau'$, we have $(\sigma(x),\tau'(x))\in U$ for every 
$x\in G/K$.

\smallskip
Next we show that (ii) implies (i). We define
\begin{equation*}
  U:=\{(g,h)\in G\times H\mid gK=\eta(hL)\}\,.
\end{equation*}
Then $k_1(U)=K$, $k_2(U)=L$, $p_1(U)=G$ and $p_2(U)=H$. We only need
to show that $\kappa\times \lambda^{-1}\in (K\times L)^*$ can be
extended to a homomorphism $\phi\in U^*$. Consider the central extension
\begin{equation*}
  1\to K\times L \to U \to U/(K\times L) \to 1
\end{equation*}
and note that $p_1\colon U\to G$ induces an isomorphism $\pbar_1\colon
U/(K\times L)\to G/K$. Furthermore, note that, with 
\begin{equation*}
  \tau':=\tau\circ\eta^{-1}\colon G/K\to H\quad\text{and}\quad 
  \beta':=\beta\circ(\eta^{-1}\times\eta^{-1})\in Z^2(G/K,L)\,,
\end{equation*}
the function
\begin{equation*}
  \rho\colon G/K\to U\,,\quad
  x\mapsto\bigl(\sigma(x),\tau'(x)\bigr)\,,
\end{equation*}
is a section of the surjection $U\to U/(K\times
L)\mathop{\to}\limits^{\pbar_1} G/K$ and that
\begin{equation*}
  \gamma\colon G/K\times G/K \to K\times L\,, \quad (x,y)\mapsto
  \bigl(\alpha(x,y),\beta'(x,y)\bigr)\,,
\end{equation*}
defines a cocycle $\gamma\in Z^2(G/K,K\times L)$ such that
\begin{equation*}
  \rho(x)\rho(y) = \gamma(x,y)\rho(xy)
\end{equation*}
for all $x,y\in G/K$. This means that $\rho':=\rho\circ\pbar_1\colon
U/(K\times L)\to U$ is a section of $U\to U/(K\times L)$ with
corresponding cocycle $\gamma':=\gamma\circ(\pbar_1\times\pbar_1)\in
Z^2(U/(K\times L),K\times L)$. Now, $\kappa\times\lambda^{-1}\in
(K\times L)^*$ extends to a homomorphism $\phi\in U^*$ if and only if
in the following part of the Hochschild-Serre five term exact sequence (see \cite[Theorem~1.5.1]{Karpilovsky1987a}),
\begin{equation*}
  \cdots\ar\Hom(U,\torA)\ar \Hom(K\times L,\torA) \Ar{\delta}
  H^2(U/(K\times L),\torA)\ar \cdots\,,
\end{equation*}
the homomorphism $\kappa\times\lambda^{-1}\in\Hom(K\times L,\torA)$
belongs to the kernel of the connecting homomorphism $\delta$. However,
by \cite[Theorem~1.5.1]{Karpilovsky1987a}, one has $\delta(\kappa\times\lambda^{-1}) =
[(\kappa\times\lambda^{-1})\circ\gamma']$. Therefore, it suffices to
show that $[(\kappa\times\lambda^{-1})\circ\gamma]=1\in
H^2(G/K,\torA)$. But
\begin{equation*}
  \bigl((\kappa\times\lambda^{-1})\circ\gamma\bigr)(x,y) =
  \kappa\bigl(\alpha(x,y)\bigr) \cdot
  \lambda\bigl(\beta'(x,y)\bigr)^{-1} = \mu(x)\mu(y)\mu(xy)^{-1}
\end{equation*}
for some function $\mu\colon G/K\to\torA$, by the hypothesis in (ii).
\end{proof}

\begin{remark}\mylabel{rem linkage criterion}
(a) Let $G$ and $H$ be finite groups and let $(K,\kappa)\in\calM_G^G$
and $(L,\lambda)\in \calM_H^H$. Let $\Khat:=\ker(\kappa)$,
$\Lhat:=\ker(\lambda)$ and set
\begin{equation*}
  \Gbar:=G/\Khat\,, \quad \Kbar:=K/\Khat\,, \quad \Hbar:=H/\Lhat\,,
  \quad \Lbar:=L/\Lhat\,.
\end{equation*}
Furthermore, let $\kappabar\in\Kbar^*$ and $\lambdabar\in \Lbar^*$
denote the homomorphisms that inflate to $\kappa$ and $\lambda$,
respectively. It is easy to see that
\begin{equation*}
  (G,K,\kappa)\sim(H,L,\lambda) \iff
  (\Gbar,\Kbar,\kappabar)\sim(\Hbar,\Lbar,\lambdabar)\,.
\end{equation*}
Therefore, the question if $(G,K,\kappa)$ is linked to $(H,L,\lambda)$
can be reduced to the case where $\kappa$ and $\lambda$ are faithful
and can be answered by the criterion in Proposition~\ref{linkage
criterion}.

\smallskip
(b) Proposition~\ref{linkage criterion} is still true if $H^2(G/K,\torA)$ in (ii) 
is replaced by $H^2(G/K,A)$. The proof is exactly the same, mutatis mutandis. Note also that the natural map $H^2(G/K,\torA)\to H^2(G/K,A)$ is injective by the long exact cohomology sequence, since $H^1(G/K,A/\torA)\cong \Hom(G/K,A/\torA)$ is trivial.
\end{remark}

\begin{nothing}\mylabel{linked e's and f's}
(a) Recall that, for $(K,\kappa)\in\calM_G^G$, we denote by
$\{K,\kappa\}_G$ the $G$-linkage class of $(K,\kappa)$ in
$\calM_G^G$. The partial order on $\calM_G^G$ induces a partial
order on the linkage classes, defined for
$(K,\kappa),(L,\lambda)\in\calM_G^G$ by
\begin{equation*}
  \{K,\kappa\}_G\le \{L,\lambda\}_G
\end{equation*}
if and only if there exists $(K',\kappa')\in\{K,\kappa\}_G,
(L',\lambda')\in\{L,\lambda\}_G$ with $(K',\kappa')\le (L',\lambda')$.
In order to see that this relation is transitive it suffices to show
that if $(K,\kappa)\le(L,\lambda)$ in $\calM_G^G$ and
$(K',\kappa')\sim_G(K,\kappa)$ in $\calM_G^G$ then there exists
$(L',\lambda')\in\calM_G^G$ with
$(K',\kappa')\le(L',\lambda')\sim_G(L,\lambda)$. The existence of
$(L',\lambda')$ is seen as follows. Since
$(K,\kappa)\sim_G(K',\kappa')$, there exists $(U,\phi)\in\calM_{G\times
G}$ with $l(U,\phi)=(G,K,\kappa)$ and $r(U,\phi)=(G,K',\kappa')$.
We can write
\begin{equation*}
  E_{(L,\lambda)}\otimes_{AG}\left(\frac{G\times G}{U,\phi}\right) \cong
  \left(\frac{G\times G}{V,\psi}\right)
\end{equation*}
for some $(V,\psi)\in\calM_{G\times G}$. Using $(K,\kappa)\le (L,\lambda)$,
Proposition~\ref{idempotent proposition}(b) implies
$l(V,\psi)=(G,L,\lambda)$. Now Propositions~\ref{eta and zeta}(a)
and \ref{first consequences}(b) imply $r(V,\psi)=(G,L',\lambda')$
for some $(L',\lambda')\in\calM_G^G$ with $(K',\kappa')\le
(L',\lambda')\sim_G(L,\lambda)$, as desired.

\smallskip
(b) For $(K,\kappa)\in\calM_G^G$ we define the elements
\begin{equation*}
  e_{\{K,\kappa\}_G}:= \sum_{(K',\kappa')\in\{K,\kappa\}_G}
  e_{(K',\kappa')} \in B_k^A(G, G)
\end{equation*}
and
\begin{equation*}
  f_{\{K,\kappa\}_G}:= \sum_{(K',\kappa')\in\{K,\kappa\}_G}
  f_{(K',\kappa')} \in B_k^A(G, G)\,.
\end{equation*}
\end{nothing}

The following proposition now follows immediately from
Proposition~\ref{e,f-relations 1}.

\begin{proposition}\mylabel{e,f-relations 2}
Let $(K,\kappa),(L,\lambda)\in\calM_G^G$. Then
\begin{align*}
  e_{\{K,\kappa\}_G}f_{\{L,\lambda\}_G} =
  f_{\{L,\lambda\}_G}e_{\{K,\kappa\}_G}
  & = 0 \quad\text{unless
  $\{K,\kappa\}_G\le\{L,\lambda\}_G$,}\\
  e_{\{K,\kappa\}_G}f_{\{K,\kappa\}_G} =
  f_{\{K,\kappa\}_G}e_{\{K,\kappa\}_G}
  & =  f_{\{K,\kappa\}_G}\quad\text{and}\\
  f_{\{K,\kappa\}_G}f_{\{L,\lambda\}_G}
  & =
  \begin{cases}
    f_{\{K,\kappa\}_G}\,, & \text{if $\{K,\kappa\}_G =
    \{L,\lambda\}_G$,}\\
    0\,, & \text{otherwise.}
  \end{cases}
\end{align*}
\end{proposition}

%%%%%%%%%%%%%%% SECTION 6 %%%%%%%%%%%%%%%%%%%%%%%%%%%%%%%%%%%%

\section{The algebra $E_G^c$ and the group $\Gamma_{(G,K,\kappa)}$}\label{sec E^c}

This section is devoted to the study of the structure of the subalgebra $E_G^c$ of $E_G$ generated by the classes of
transitive $A$-fibered $(G,G)$-bisets whose stabilizing pairs have full projections. In Section~\ref{sec Ebar} it will be shown that this algebra covers the quotient $\Ebar_G$. In this section we will show that it is isomorphic to a
direct product of matrix algebras over certain group algebras $k\Gamma_{(G,K,\kappa)}$. These group algebras, together with the linkage classes, will be the main ingredients of the parametrization of the simple fibered biset functors in Section~\ref{sec simple functors}.

\begin{nothing}\mylabel{covering algebra} {\bf The algebra $E_G^c$ and the group $\Gamma_{(G,K,\kappa)}$.}\quad
(a) Let $G$ and $H$ be finite groups and let
$(U,\phi)\in\calM_{G\times H}$. We say that $(U,\phi)$ is {\em
covering} if $p_1(U)=G$ and $p_2(U)=H$. We denote by $\calMc_{G\times
H}$ the set of all covering pairs of $\calM_{G\times H}$. Note
that if $(U,\phi)$ as above is covering then
$l_0(U,\phi)\in\calM_G^G$ and $r_0(U,\phi)\in\calM_H^H$. If also
$K$ is a finite group and also $(V,\psi)\in\calMc_{H\times K}$ then, by the Mackey formula,
\begin{equation*}
  \left[\frac{G\times H}{U,\phi}\right] \cdotH
  \left[\frac{H\times K}{V,\psi}\right] =
  \begin{cases}
    \left[\frac{G\times K}{U*V,\phi*\psi}\right]\,, &
    \text{if $\phi_2=\psi_1$ on $k_2(U)\cap k_1(V)$,}\\
    0\, & \text{otherwise,}
  \end{cases}
\end{equation*}
with $(U*V,\phi*\psi)\in\calMc_{G\times K}$ in the first case. This
implies that the $k$-span
\begin{equation*}
  E_G^c = E_k^{A,c}(G) = B_k^{A,c}(G, G)
\end{equation*}
of the canonical basis elements $\left[\frac{G\times
G}{U,\phi}\right]$ of $E_G=B_k^A(G, G)$, with
$(U,\phi)\in\calMc_{G\times G}$, is a $k$-subalgebra of
$E_G=B_k^A(G, G)$. Note that $e_{(K,\kappa)},f_{(K,\kappa)},
e_{\{K,\kappa\}_G}, f_{\{K,\kappa\}_G}\in E_G^c$ for all
$(K,\kappa)\in\calM_G^G$.

\smallskip
(b) Let $G$ be a finite group and let $(K,\kappa)\in\calM_G^G$. It
follows from Propositions~\ref{first consequences}(d) and
\ref{idempotent proposition} that the standard basis elements
$\left[\frac{G\times G}{U,\phi}\right]$ of $E_G^c$ with
$l(U,\phi)=(G,K,\kappa)=r(U,\phi)$ form a finite group $\Gamma_{(G,K,\kappa)}$
under multiplication, with identity element $e_{(K,\kappa)}$ and inverses induced by taking
opposite fibered bisets.

\smallskip
(c) More generally, assume that $(K,\kappa)\in\calM_G^G$ and $(L,\lambda)\in\calM_H^H$ and set
\begin{equation*}
  \lGamma{(G,K,\kappa)}_{(H,L,\lambda)}:=\{\left[\frac{G\times H}{U,\phi}\right]\mid 
  l(U,\phi)=(G,K,\kappa),\ r(U,\phi)=(H,L,\lambda)\}\,.
\end{equation*}
This set is non-empty if and only if $(G,K,\kappa)$ and $(H,L,\lambda)$ are linked. Assume from now on that this is the case. Then it is a $(\Gamma_{(G,K,\kappa)},\Gamma_{(H,L,\lambda)})$-biset and each of the groups $\Gamma_{(G,K,\kappa)}$ and $\Gamma_{(H,L,\lambda)}$ acts transitively and freely on 
$\lGamma{(G,K,\kappa)}_{(H,L,\lambda)}$. This is easily verified by tensoring with opposites of standard basis elements and using Proposition~\ref{first consequences}(d). Therefore, each element $\left[\frac{G\times H}{U,\phi}\right]\in\lGamma{(G,K,\kappa)}_{(H,L,\lambda)}$ induces an isomorphism
\begin{equation*}
  \gamma\colon \Gamma_{(H,L,\lambda)}\liso\Gamma_{(G,K,\kappa)}\,,\quad
  y\mapsto \left[\frac{G\times H}{U,\phi}\right]\cdotH y\cdotH \left[\frac{G\times H}{U,\phi}\right]^{\mathrm{op}}\,,
\end{equation*}
and if also $\left[\frac{G\times H}{U',\phi'}\right]\in \lGamma{(G,K,\kappa)}_{(H,L,\lambda)}$ then the resulting isomorphism $\gamma'$ satisfies $\gamma'=c_z\circ\gamma$, where $z=\left[\frac{G\times H}{U',\phi'}\right]\cdotG \left[\frac{G\times H}{U,\phi}\right]^{\mathrm{op}}\in\Gamma_{(H,L,\lambda)}$. As $(\Gamma_{(G,K,\kappa)},\Gamma_{(H,L,\lambda)})$-biset we have $\lGamma{(G,K,\kappa)}_{(H,L,\lambda)}\cong \left[\frac{\Gamma_{(G,K,\kappa)}\times\Gamma_{(H,L,\lambda)}}{\Delta_\gamma}\right]$, where $\Delta_\gamma:=\{(\gamma(y),y)\mid y\in\Gamma_{(H,L,\lambda)}\}$.
As a consequence, one obtains a canonical bijection 
\begin{equation}\label{eqn irr bijection}
  \Irr(k\Gamma_{(H,L,\lambda)})\myiso\Irr(k\Gamma_{(G,K,\kappa)})
\end{equation}
induced by the category equivalence $k[\lGamma{(G,K,\kappa)}_{(H,L,\lambda)}] \otimes_{k\Gamma_{(H,L,\lambda})} -$ between the category of left $k\Gamma_{(H,L,\lambda)}$-modules and the category of left $k\Gamma_{(G,K,\kappa)}$-modules. Its inverse is given by tensoring with the bimodule $k[\lGamma{(H,L,\lambda)}_{(G,K,\kappa)}]=k[\lGamma{(G,K,\kappa)}^{\mathrm{op}}_{(H,L,\lambda)}]$. The bijection in (\ref{eqn irr bijection}) coincides with the one given by transport of structure via the isomorphism $\gamma$.
\end{nothing}

The goal of this section is the following theorem describing the
structure of the $k$-algebra $E_G^c=B_k^{A,c}(G, G)$.

\begin{theorem}\mylabel{covering algebra structure theorem}
There exists a $k$-algebra isomorphism
\begin{equation*}
 \bigoplus_{\{K,\kappa\}_G\in\calM_G^G/\sim} \Mat_{|\{K,\kappa\}_G|}(k\Gamma_{(G,K,\kappa)})
  \liso E_G^c
\end{equation*}
with the following property: For every $(K,\kappa)\in\calM_G^G$, writing $\{K,\kappa\}_G=\{(K_1,\kappa_1),\ldots,(K_n,\kappa_n)\}$, the element $f_{(K_i,\kappa_i)}\in E_G^c$ is mapped to the diagonal idempotent matrices $e_i=\dia(0,\ldots,0,1,0,\ldots,0)\in\Mat_{|\{K,\kappa\}_G|}(k\Gamma_{(G,K,\kappa)})$, $i=1,\ldots,|\{K,\kappa\}_G|$, in the $\{K,\kappa\}_G$-component.% In particular, the map
%\begin{equation*}
%  k\Gamma_{(G,K,\kappa)}\to f_{(K,\kappa)} E_G^c f_{(K,\kappa)}\,,\quad x\mapsto f_{(K,\kappa)}xf_{(K,\kappa)}\,,
%\end{equation*}
%is a $k$-algebra isomorphism.
\end{theorem}

Before we can prove the theorem we need a few auxiliary results.

\begin{lemma}\mylabel{orthogonality 1}
Let $(U,\phi)\in\calMc_{G\times G}$ and let
$(K,\kappa),(L,\lambda)\in\calM_G^G$. If $f_{\{K,\kappa\}_G}
\left[\frac{G\times G}{U,\phi}\right] f_{\{L,\lambda\}_G} \neq 0$ then
$\{K,\kappa\}_G=\{L,\lambda\}_G$.
\end{lemma}

\begin{proof}
If the above expression is non-zero then there exist
$(K',\kappa')\in\{K,\kappa\}_G$ and $(L',\lambda')\in\{L,\lambda\}_G$
such that $f_{(K',\kappa')} \left[\frac{G\times G}{U,\phi}\right]
f_{(L',\lambda')} \neq 0$ in $E_G^c$. Expanding $f_{(L',\lambda')}$ as
in Equation~(\ref{definition of f's}), we see that there exists
$(L'',\lambda'')\in\calM_G^G$ with
$(L',\lambda')\le(L'',\lambda'')$ and $f_{(K',\kappa')}
\left[\frac{G\times G}{U,\phi}\right] e_{(L'',\lambda'')} \neq 0$. So
also $\left[\frac{G\times G}{U,\phi}\right] e_{(L'',\lambda'')}$ is
non-zero and of the form $\left[\frac{G\times G}{V,\psi}\right]$ for
some $(V,\psi)\in\calMc_{G\times G}$ satisfying
$r_0(V,\psi)\ge(L'',\lambda'')$ and $(K'',\kappa''):=l_0(V,\psi)\ge
l_0(U,\phi)$ by Proposition~\ref{first consequences}(b). By
Proposition~\ref{idempotent proposition}, we now have
\begin{equation*}
  0 \neq f_{(K',\kappa')}
  \left[\frac{G\times G}{U,\phi}\right] e_{(L'',\lambda'')}
  = f_{(K',\kappa')} e_{(K'',\kappa'')}
  \left[\frac{G\times G}{U,\phi}\right] e_{(L'',\lambda'')}
\end{equation*}
which implies $(K'',\kappa'')\le(K',\kappa')$ by
Proposition~\ref{e,f-relations 1}. Altogether, we obtain
\begin{align*}
  \{L,\lambda\}_G & =\{L',\lambda'\}_G \le \{L'',\lambda''\}_G \le
  \{r_0(V,\psi)\}_G = \{l_0(V,\psi)\}_G\\
  & = \{K'',\kappa''\}_G \le \{K',\kappa'\}_G =\{K,\kappa\}_G\,.
\end{align*}
Similarly, we can prove $\{K,\kappa\}_G\le\{L,\lambda\}_G$.
\end{proof}

\begin{corollary}\mylabel{Bc decomposition}
Let $G$ be a finite group. The elements $f_{\{K,\kappa\}_G}$,
$\{K,\kappa\}_G\in\calM_G^G/\sim$, are mutually orthogonal central
idempotents of $E_G^c=B_k^{A,c}(G, G)$ and their sum is equal to $1$. In
particular, one has a decomposition
\begin{equation}\label{decomposition 1}
  E_G^c = \bigoplus_{\{K,\kappa\}_G\in\calM_G^G/\sim}
  f_{\{K,\kappa\}_G}E_G^c
\end{equation}
of $E_G^c$ into two-sided ideals.
\end{corollary}

\begin{proof}
We already know from Proposition~\ref{e,f-relations 2} that the
elements $f_{\{K,\kappa\}_G}$, $\{K,\kappa\}_G\in\calM_G^G/\sim$,
are mutually orthogonal idempotents of $E_G^c$
and we know from Equation~(\ref{sum equals 1}) that their sum is equal
to $1$. From this and from Lemma~\ref{orthogonality 1}, we obtain the three
decompositions
\begin{equation*}
  E_G^c = \bigoplus f_{\{K,\kappa\}_G} E_G^c =
  \bigoplus f_{\{K,\kappa\}_G} E_G^c f_{\{K,\kappa\}_G} =
  \bigoplus E_G^c f_{\{K,\kappa\}_G}
\end{equation*}
where each of the three sums runs over
$\{K,\kappa\}_G\in\calM_G^G/\sim$. Since
$f_{\{K,\kappa\}_G}E_G^cf_{\{K,\kappa\}_G}\subseteq
f_{\{K,\kappa\}_G}E_G^c$ and
$f_{\{K,\kappa\}_G}E_G^cf_{\{K,\kappa\}_G}\subseteq
E_G^cf_{\{K,\kappa\}_G}$, we obtain equality
\begin{equation*}
  f_{\{K,\kappa\}_G} E_G^c =
  f_{\{K,\kappa\}_G} E_G^c f_{\{K,\kappa\}_G} =
  E_G^c f_{\{K,\kappa\}_G}
\end{equation*}
for every $\{K,\kappa\}_G\in\calM_G^G/\sim$. Now let $b\in E_G^c$
be arbitrary. Then
\begin{equation*}
  \sum_{\{K,\kappa\}_G\in\calM_G^G/\sim} f_{\{K,\kappa\}_G}b = b =
  \sum_{\{K,\kappa\}_G\in\calM_G^G/\sim} bf_{\{K,\kappa\}_G}
\end{equation*}
and we obtain $f_{\{K,\kappa\}_G}b = bf_{\{K,\kappa\}_G}$ for every
$\{K,\kappa\}_G\in\calM_G^G/\sim$ by the equality of the above
decompositions. This completes the proof.
\end{proof}

\begin{nothing}\mylabel{two decompositions}
Besides the decomposition (\ref{decomposition 1}) of $E_G^c$ into ideals we have
a second natural decomposition of $E_G^c$, this time into $k$-submodules, again indexed
by $\calM_G^G/\sim$. For $\{K,\kappa\}_G\in\calM_G^G/\sim$, we
denote by $E_G^{c,\{K,\kappa\}}$
the $k$-span of all standard basis elements $\left[\frac{G\times
G}{U,\phi}\right]$ with $(U,\phi)\in\calMc_{G\times G}$ satisfying
$l_0(U,\phi)\in\{K,\kappa\}_G$. Note that this last condition is
equivalent to requiring $r_0(U,\phi)\in\{K,\kappa\}_G$. We have the
obvious decomposition
\begin{equation}\label{decomposition 2*}
  E_G = \bigoplus_{\{K,\kappa\}_G\in\calM_G^G/\sim} E_G^{c,\{K,\kappa\}}  
\end{equation}
into $k$-submodules. The next lemma provides a connection
between the latter decomposition and the one in (\ref{decomposition 1}).
\end{nothing}

\begin{lemma}\mylabel{comparing decompositions}
Let $G$ be a finite group and let $(K,\kappa)\in\calM_G^G$.

\smallskip
{\rm (a)} Let $(U,\phi)\in\calMc_{G\times G}$ with
$r_0(U,\phi)=(K,\kappa)$ and let $(L,\lambda)\in\calM_G^G$ with
$(K,\kappa)\not\le(L,\lambda)$. Then $\left[\frac{G\times
G}{U,\phi}\right] f_{(L,\lambda)}=0$.

\smallskip
{\rm (b)} One has
\begin{equation*}
  \bigoplus_{\{K,\kappa\}_G\le\{L,\lambda\}_G\in\calM_G^G/\sim} 
  E_G^{c,\{L,\lambda\}}
  =  \bigoplus_{\{K,\kappa\}_G\le\{L,\lambda\}_G\in\calM_G^G/\sim}
  E_G^c f_{\{L,\lambda\}_G}\,.
\end{equation*}

\smallskip
{\rm (c)} The projection map
\begin{equation*}
  \omega\colon E_G^{c,\{K,\kappa\}} \to E_G^c f_{\{K,\kappa\}_G}\,,
  \quad b\mapsto bf_{\{K,\kappa\}_G}\,,
\end{equation*}
with respect to the decomposition~(\ref{decomposition 1}) is an isomorphism of $k$-modules whose inverse 
is the projection map with respect to the decomposition~(\ref{decomposition 2*}).

\smallskip
{\rm (d)} Enumerating the elements of $\{K,\kappa\}_G$ as $(K_1,\kappa_1), \ldots,(K_n,\kappa_n)$, the isomorphism $\omega$ is the direct sum of the $k$-module isomorphisms 
\begin{equation*}
  \omega_{ij}\colon k[\lGamma{(G,K_i,\kappa_i)}_{(G,K_j,\kappa_j)}] \to f_{(K_i,\kappa_i)} E_G^c f_{(K_j,\kappa_j)}\,,\quad
  b_{ij}\mapsto  f_{(K_i,\kappa_i)} b_{ij} f_{(K_j,\kappa_j)}\,,
\end{equation*}
for $i,j\in\{1,\ldots,n\}$.

\smallskip
{\rm (e)} Let $i,j,l,m\in \{1,\ldots,n\}$, $b_{ij}\in k[\lGamma{(G,K_i,\kappa_i)}_{(G,K_j,\kappa_j)}]$ and $b_{lm}\in k[\lGamma{(G,K_l,\kappa_l)}_{(G,K_m,\kappa_m)}]$. If $j=l$ then $b_{ij}b_{lm}\in k[\lGamma{(G,K_i,\kappa_i)}_{(G,K_m,\kappa_m)}]$. In any case one has
\begin{equation}\label{eqn omega matrix property}
  \omega_{ij}(b_{ij})\omega_{lm}(b_{lm})=
  \begin{cases} 
     \omega_{il}(b_{ij}b_{lm}), & \text{if $j=l$,}\\
     0, & \text{if $j\neq l$.}
   \end{cases}
\end{equation}

\smallskip
{\rm (f)} The map
\begin{equation*}
  k[\Gamma_{(G,K,\kappa)}]\to f_{(K,\kappa)} E_G^c f_{(K,\kappa)}\,, \quad a\mapsto f_{(K,\kappa)} a f_{(K,\kappa)}\,,
\end{equation*}
is a $k$-algebra isomorphism.
\end{lemma}

\begin{proof}
(a) One has
\begin{equation*}
  \left[\frac{G\times G}{U,\phi}\right] f_{(L,\lambda)} =
  \left[\frac{G\times G}{U,\phi}\right] e_{(K,\kappa)} f_{(L,\lambda)}
  = 0
\end{equation*}
by Propositions~\ref{idempotent proposition}(c) and \ref{e,f-relations
1}.

\smallskip
(b) The left hand side of the equation is an ideal of $E_G^c$ by Proposition~\ref{first consequences}(b), and it contains
$f_{\{L,\lambda\}_G}$ for every $\{L,\lambda\}_G\ge\{K,\kappa\}_G$.
This shows that the left hand side contains the right hand side.
Conversely, let $(M,\mu)\in\calM_G^G$ with
$\{M,\mu\}_G\not\ge\{K,\kappa\}_G$. Then, by Part~(a), the left hand
side annihilates $f_{\{M,\mu\}_G}$. But the right hand side equals the
annihilator of the set of all $f_{\{M,\mu\}_G}$ with
$\{M,\mu\}_G\not\ge\{K,\kappa\}_G$, by Corollary~\ref{Bc
decomposition}. This shows that the left hand side is contained in the
right hand side.

\smallskip
(c) Note that the subsums
of both sides in the equation in Part~(b), that are indexed by $\{L,\lambda\}_G>\{K,\kappa\}_G$,
are also equal. In fact, this follows by applying (b) to the elements $\{L,\lambda\}_G$. Therefore, $E_G^{c,\{K,\kappa\}}$ and $E_G^c f_{\{K,\kappa\}_G}$ are both complements to the same submodule of the direct sum in (b). The assertion is now immediate.

\smallskip
(d) By definition we have a direct sum decomposition into $k$-submodules
\begin{equation}\label{eqn ij decomposition}
  E_G^{c,\{K,\kappa\}} = \bigoplus_{i,j=1}^n k[\lGamma{(G,K_i,\kappa_i)}_{(G,K_j,\kappa_j)}]\,.
\end{equation}
Since $f_{\{K,\kappa\}_G}=f_{(K_1,\kappa_1)}+\cdots + f_{(K_n,\kappa_n)}$ is an orthogonal idempotent decomposition and since $f_{\{K,\kappa\}_G}$ is central in $E_G^c$, we also have a decomposition
\begin{equation*}
  E_G^c f_{\{K,\kappa\}_G} = f_{\{K,\kappa\}_G} E_G^c f_{\{K,\kappa\}_G} 
  = \bigoplus_{i,j=1}^n f_{(K_i,\kappa_i)} E_G^c f_{(K_j,\kappa_j)}
\end{equation*}
into $k$-submodules. Moreover, by Part~(a) and its left-sided version we have, for $b_{ij}\in k[\lGamma{(G,K_i,\kappa_i)}_{(G,K_j,\kappa_j)}]$, 
\begin{equation*}
  \omega(b_{ij}) = f_{\{K,\kappa\}_G} b_{ij} f_{\{K,\kappa\}_G} = f_{(K_i,\kappa_i)} b_{ij} f_{(K_j,\kappa_j)} \in 
  f_{(K_i,\kappa_i)} E_G^c f_{(K_j,\kappa_j)}\,.
\end{equation*}
Since $\omega$ is an isomorphism of $k$-modules, also each $\omega_{ij}$ is an isomorphism of $k$-modules. 

\smallskip
(e) If $j=l$ then we have $b_{ij}b_{jm}\in k[\lGamma{(G,K_i,\kappa_i)}_{(G,K_m,\kappa_m)}]$ by  Proposition~\ref{first consequences}(d). Moreover, in this case, by the same argument as in the proof of Part~(d), and since $f_{\{K,\kappa\}_G}$ is central in $E_G^c$, we have 
\begin{align*}
  \omega_{ij}(b_{ij})\omega_{jm}(b_{jm}) 
  & = f_{(K_i,\kappa_i)} b_{ij} f_{(K_j,\kappa_j)} b_{jm} f_{(K_m,\kappa_m)} 
  = f_{(K_i,\kappa_i)} b_{ij} f_{\{K,\kappa\}_G} b_{jm} f_{(K_m,\kappa_m)} \\
  & = f_{(K_i,\kappa_i)} f_{\{K,\kappa\}_G} b_{ij}  b_{jm} f_{(K_m,\kappa_m)}
  = f_{(K_i,\kappa_i)}b_{ij}b_{jm} f_{(K_m,\kappa_m)} = \omega_{im}(b_{ij} b_{jm})\,.
\end{align*}
On the other hand, if $j\neq l$ then $\omega_{ij}(b_{ij})\omega_{lm}(b_{lm}) = f_{(K_i,\kappa_i)}b_{ij} f_{(K_j,\kappa_j)} f_{(K_l,\kappa_l}) b_{lm} f_{(K_m,\kappa_m)} = 0$, since $f_{(K_j,\kappa_j)}f_{(K_l,\kappa_l)} = 0$.

\smallskip
(f) This is an immediate consequence of Part~(e) by choosing $i\in\{1,\ldots,n\}$ such that $(K,\kappa) = (K_i,\kappa_i)$ and considering $\omega_{ii}$.
\end{proof}

Now we are ready to prove Theorem~\ref{covering algebra structure
theorem}.

\bigskip\noindent
\begin{nothing}  {\bf Proof} {\sl of Theorem}~\ref{covering algebra structure theorem}.
Let $(K,\kappa)\in\calM_G^G$. By Corollary~\ref{Bc decomposition}, it suffices to show that
there exists a $k$-algebra isomorphism between
$E_G^c f_{\{K,\kappa\}_G}$ and
$\Mat_n(k\Gamma_{(G,K,\kappa)})$ which maps $f_{(K_i,\kappa_i)}$ to the diagonal idempotent matrix $e_i$, where $(K_1,\kappa_1),(K_2,\kappa_2),\ldots,(K_n,\kappa_n)$ enumerate the elements of
$\{K,\kappa\}_G$. 

For each $i=1,\ldots,n$, we choose an element $(U_i,\phi_i)\in\calMc_{G\times
G}$ with $l_0(U_i,\phi_i)=(K,\kappa)$ and
$r_0(U_i,\phi_i)=(K_i,\kappa_i)$, and set $x_i:=\left[\frac{G\times
G}{U_i,\phi_i}\right]\in E_G^{c,\{K,\kappa\}}$, for
$i=1\ldots,n$. For any two elements $i,j\in\{1,\ldots,n\}$ we claim that the $k$-module homomorphism
\begin{equation*}
  \sigma_{ij}\colon k\Gamma_{(G,K,\kappa)} \to k[\lGamma{(G,K_i,\kappa_i)}_{(G,K_j,\kappa_j)}]\,,\quad
  a\mapsto\xop_i a x_j\,,
\end{equation*}
is an isomorphism. In fact, first of all, $\xop_i a x_j\in k[\lGamma{(G,K_i,\kappa_i)}_{(G,K_j,\kappa_j)}]$ by Proposition~\ref{first consequences}(d). Secondly, the map $b\mapsto x_i b \xop_j$ is an inverse to $\sigma_{ij}$, since $x_i\cdot \xop_i= e_{(K,\kappa)}$, $\xop_i\cdot x_i = e_{(K_i,\kappa_i)}$, and $e_{(K,\kappa)} a = a e_{(K,\kappa)} = a$ and $e_{(K_i,\kappa_i)} b = b e_{(K_j,\kappa_j)} = b$ by Proposition~\ref{idempotent proposition}(c) and (d). Moreover, we have, for $a,a'\in k\Gamma_{(G,K,\kappa)}$ and $i,j,l\in\{1,\ldots,n\}$,
\begin{equation}\label{eqn sigma matrix property}
  \sigma_{ij}(a)\sigma_{jl}(a') = \sigma_{il}(aa')\,,
\end{equation}
since $\xop_i a x_j \xop_j a' x_l = \xop_i a e_{(K_j,\kappa_j)}  a' x_l = \xop_i aa' x_l$, again by Proposition~\ref{idempotent proposition}(c) and (d). Taking the direct sum of the maps $\sigma_{ij}$ and using the direct sum decomposition in (\ref{eqn ij decomposition}) we obtain a $k$-module isomorphism
\begin{equation*}
   \sigma\colon \Mat_n(k\Gamma_{(G,K,\kappa)}) \to E_G^{c,\{K,\kappa\}}\,.
\end{equation*}
Thus, we have a $k$-module isomorphism
\begin{equation*}
  \rho:=\omega\circ\sigma \colon 
  \Mat_n(k\Gamma_{(G,K,\kappa)}) \myiso E_G^{c,\{K,\kappa\}} \myiso E_G^c f_{\{K,\kappa\}_G}\,,\quad 
  (a_{ij})\mapsto \sum_{i,j=1}^n \omega_{ij}(\sigma_{ij}(a_{ij}))\,.
\end{equation*}
For $(a_{ij}), (a'_{lm})\in\Mat_n(k\Gamma_{(G,K,\kappa)})$, Equations~(\ref{eqn omega matrix property}) and (\ref{eqn sigma matrix property}) now imply that $\rho$ is a $k$-algebra isomorphism:
\begin{align*}
  \rho((a_{ij}))\rho((a'_{lm})) 
  & = \Bigl(\sum_{i,j=1}^n\omega_{ij}(\sigma_{ij}(a_{ij}))\Bigr) \Bigl(\sum_{l,m=1}^n \omega_{lm}(\sigma_{lm}(a'_{lm}))\Bigr) \\
  & = \sum_{i,j,l,m=1}^n \omega_{ij}(\sigma_{ij}(a_{ij}))\cdot \omega_{lm}(\sigma_{lm}(a'_{lm})) 
  = \sum_{i,j,m=1}^n \omega_{ij}(\sigma_{ij}(a_{ij}))\cdot \omega_{jm}(\sigma_{jm}(a'_{jm}))  \\
  & = \sum_{i,j,m=1}^n \omega_{im}\bigl(\sigma_{ij}(a_{ij})\cdot\sigma_{jm}(a'_{jm})\bigr) 
  = \sum_{i,j,m=1}^n \omega_{im}\bigl(\sigma_{im}(a_{ij}a'_{jm}\bigr)) \\
  & = \sum_{i,m=1}^n \omega_{im}\bigl(\sigma_{im}(\sum_{j=1}^n a_{ij}a'_{jm})\bigr) 
  = \omega(\sigma( (a_{ij})\cdot(a'_{lm})) ) = \rho( (a_{ij})\cdot(a'_{lm}) )\,.
\end{align*}
Finally, one has $\rho(e_i) = \omega_{ii}(\sigma_{ii}(1)) = f_{(K_i,\kappa_i)} e_{(K_i,\kappa_i)}f_{(K_i,\kappa_i)} = f_{(K_i,\kappa_i)}$ as desired.
\qed
\end{nothing}

%%%%%%%%%%%%%%%% SECTION 7 %%%%%%%%%%%%%%%%%%%%%%%%%%%%%%%%%%%%

\section{The structure of $\Gamma_{(G,K,\kappa)}$}\label{sec structure of Gamma}

In this section we show that, for $(K,\kappa)\in\calM_G^G$, the group $\Gamma_{(G,K,\kappa)}$ is an extension of a certain subgroup of
the outer automorphism group $\Out(G)$ with the group $(G/K)^*$, under the assumption that $\kappa$ is faithful. By passing from $G$ to $G/\ker(\kappa)$ one could avoid this assumption at the cost of more complicated notation. In subsequent sections we will only need to consider the case where $\kappa$ is faithful.

\begin{nothing}\mylabel{study of Gamma}
Let $(K,\kappa)\in\calM_G^G$ and assume that $\kappa$ is faithful. Then $K\le Z(G)$. 
For $\eta\in\Aut(G/K)$ we set
\begin{equation*}
  U_\eta:=\{(g_1,g_2)\in G\times G\mid \eta(g_2K)=g_1K\}\,.
\end{equation*}
Clearly, $U_\eta$ satisfies
\begin{equation}\label{Ueta property}
  p_1(U_\eta)=p_2(U_\eta)=G\,,\quad\text{and}\quad
  k_1(U_\eta)=k_2(U_\eta)=K\,.
\end{equation}
Conversely, every subgroup $U\le G\times G$ satisfying (\ref{Ueta
property}) is of the form $U_\eta$ for some $\eta\in\Aut(G/K)$.
Moreover, if $(h_1,h_2)\in G\times G$ then
\begin{equation*}
  \lexp{(h_1,h_2)}{U_\eta}=U_{\eta'}\,,
\end{equation*}
where $\eta':= c_{h_1K\cdot\eta(h_2^{-1}K)}\circ\eta\in\Aut(G/K)$.
This implies that the map $\eta\mapsto U_\eta$ induces a bijection
between the outer automorphism group $\Out(G/K)$ and the $G\times
G$-conjugacy classes of subgroups $U\le G\times G$
satisfying~(\ref{Ueta property}). It is easy to see that, for
$\eta_1,\eta_2\in\Aut(G/K)$, one has
\begin{equation*}
  U_{\eta_1\circ\eta_2}=U_{\eta_1}*U_{\eta_2}\,.
\end{equation*}
Thus, we obtain a well-defined group homomorphism
\begin{equation*}
  \pi\colon \Gamma_{(G,K,\kappa)}\to\Out(G/K)\,,\quad
  \left[\frac{G\times G}{U_\eta,\phi}\right]\mapsto
  \etabar\,,
\end{equation*}
where $\etabar:=\eta\Inn(G/K)$.

\smallskip
Let $\alpha\in Z^2(G/K,K)$ be a cocycle that describes the central
extension $1\to K\to G\to G/K\to 1$. We denote by $\Aut^\circ(G/K)$
the subgroup of $\Aut(G/K)$ consisting of those elements $\eta$
satisfying
\begin{equation*}
  [\kappa\circ\alpha\circ(\eta^{-1}\times\eta^{-1})] =
  [\kappa\circ\alpha]\in H^2(G/K,\torA)\,.
\end{equation*}
In other words, $\Aut^\circ(G/K)$ is the stabilizer of $[\kappa\circ\alpha]$ in $\Aut(G/K)$ with respect to
the natural action of $\Aut(G/K)$ on $H^2(G/K,\torA)$. Note that the inner
automorphism group $\Inn(G/K)$ of $G/K$ fixes $[\alpha]\in H^2(G/K,K)$ and therefore also $[\kappa\circ\alpha]\in H^2(G/K,\torA)$. In fact, if $\sigma\colon G/K\to G$ is a section of the canonical epimorphism $\pi\colon G\to G/K$ such that (\ref{cocycle section relation}) holds then the cocycle $\alpha\circ (c_{z}^{-1}\times c_{z}^{-1})$ is defined by the section $\lexp{z}{\sigma}\colon x\mapsto g\sigma(z^{-1}xz)g^{-1}$, for every $z:=gK\in G/K$. We set
\begin{equation*}
  \Out^\circ(G/K):=\Aut^\circ(G/K)/\Inn(G/K)\le\Out(G/K)\,.
\end{equation*}  
It follows from Proposition~\ref{linkage criterion} that, for given
$\eta\in\Aut(G/K)$, there exists $\varphi\in(U_\eta)^*$ such that
$l_0(U_\eta,\phi)=r_0(U_\eta,\phi)=(K,\kappa)$ if and only if
$\eta\in\Aut^\circ(G/K)$. In other words, the image of $\pi$ is equal
to $\Out^\circ(G/K)$.

\smallskip
Recall that $U_{\id}=\Delta_K(G)=\Delta(G)(K\times 1)$. For every
$\theta\in(G/K)^*$ we can define a homomorphism
$\thetatilde\in(\Delta_K(G))^*$ by
\begin{equation*}
  \thetatilde(gk,g):= \kappa(k)\theta(gK)\,,
\end{equation*}
for $k\in K$ and $g\in G$. Note that $\tilde{1}=\phi_\kappa$ from Definition~\ref{idempotent definition 1}. 
It is straightforward to verify that the map
\begin{equation*}
  \iota\colon (G/K)^*\to \Gamma_{(G,K,\kappa)}\,,\quad
  \theta\mapsto \left[\frac{G\times G}{\Delta_K(G),\thetatilde}\right]\,,
\end{equation*}
is an injective group homomorphism and that $\ker(\pi)=\im(\iota)$.
\end{nothing}

The following proposition is now immediate from the preceding
discussion.

\begin{proposition}\mylabel{ses for Gamma}
Let $(K,\kappa)\in\calM_G^G$ and assume that $\kappa$ is faithful.
With the notation from~\ref{study of Gamma} one has a short exact
sequence
\begin{equation*}
  1\ar (G/K)^* \Ar{\iota} \Gamma_{(G,K,\kappa)} \Ar{\pi} \Out^\circ(G/K)
  \ar 1\,.
\end{equation*}
\end{proposition}

\begin{remark}
(a) If $\left[\frac{G\times G}{U,\phi}\right]\in\Gamma_{(G,K,\kappa)}$
with $U=U_\eta$ for some $\eta\in\Aut^\circ(G/K)$ and if $\theta\in(G/K)^*$ then, with the explicit
formula from Corollary~\ref{Mackey formula for fibered bisets 2}, it is
easy to verify that
\begin{equation*}
  \left[\frac{G\times G}{U,\phi}\right]
  \left[\frac{G\times G}{\Delta_K(G),\thetatilde}\right]
  \left[\frac{G\times G}{U,\phi}\right]^{\mathrm{op}} =
  \left[\frac{G\times G}
  {\Delta_K(G),\widetilde{\theta\circ\eta^{-1}}}\right]\,,
\end{equation*}
which shows that the induced action of $\Out^\circ(G/K)$ on $(G/K)^*$
is the usual canonical action.

\smallskip
(b) We do not know if the sequence in Proposition~\ref{ses for Gamma}
splits.
\end{remark}

%%%%%%%%%%%%%%%%%%%%%% SECTION 8 %%%%%%%%%%%%%%%%%%%%%%%%%%%%%%%%%%%%%

\section{Reduced pairs and the simple $\Ebar_G$-modules}\label{sec Ebar}

We keep the notation from the previous sections. Thus, $k$ denotes a commutative ring. We also fix a finite group $G$ for this section and we will classify the simple $\Ebar_G$-modules. This is the last step before the
classification of the simple fibered biset functors.

\begin{notation}
We define the subset $\calR_G=\calR_k^A(G)$ of $\calM_G^G$ by
\begin{equation*}
  \calR_G=\calR_k^A(G):=\{(K,\kappa)\in\calM_G^G\mid e_{(K,\kappa)}\notin I_G\}\,.
\end{equation*}
We call $(K,\kappa)\in\calM_G$ a {\em reduced} pair if $(K,\kappa)\in\calR_G$. Thus, every pair in $\calM_G\smallsetminus\calM_G^G$ is by definition not reduced. Note that if $(K,\kappa),(K',\kappa')\in\calM_G^G$ are $G$-linked (cf.~Definition~\ref{linkage definition}(b)) by a pair $(U,\phi)\in\calM_{G\times G}^c$ with $(K,\kappa)=l_0(U,\phi)$ and $(K',\kappa')=r_0(U,\phi)$ then $(K,\kappa)$ is reduced if and only if $(K',\kappa')$ is reduced. In fact,
\begin{equation*}
  e_{(K',\kappa')}= \left[\frac{G\times G}{U,\phi}\right]^{\mathrm{op}}\cdotG
  \left[\frac{G\times G}{U,\phi}\right] 
  = \left[\frac{G\times G}{U,\phi}\right]^{\mathrm{op}}\cdotG e_{(K,\kappa)} \cdotG
  \left[\frac{G\times G}{U,\phi}\right]\,,
\end{equation*}
by Proposition~\ref{idempotent proposition}(c) and (d). We write $\calR_G/\sim$ for the set of linkage classes in $\calR_G$.

Note also that for $(K,\kappa)\in\calM_G^G$, one has
\begin{equation}\label{calR is ideal}
  (K,\kappa)\le(K',\kappa')\in\calR_G\implies (K,\kappa)\in\calR_G\,,
\end{equation}
by Proposition~\ref{idempotent proposition}(c).

In the sequel we will denote the image of an element $b\in E_G$ in $\Ebar_G$ by $\bbar$.
\end{notation}

\begin{lemma}\mylabel{Basis of IG}
The ideal $I_G$ of $E_G$ is generated as a $k$-module by the standard basis elements $\left[\frac{G\times G}{U,\phi}\right]$ with $(U,\phi)\in\calM_{G\times G}$ satisfying 

\smallskip
{\rm (i)} $p_1(U)\neq G$ or

\smallskip
{\rm (ii)} $p_1(U)=G$ and $l_0(U,\phi)\notin \calR_G$.
\end{lemma}

\begin{proof}
First we show that every element in $I_G$ can be written as a $k$-linear combination of elements $\left[\frac{G\times G}{U,\phi}\right]$ with $(U,\phi)\in\calM_{G\times G}$ satisfying (i) or (ii). To that end, let $H$ be a finite group with $|H|<|G|$ and let $(V,\psi)\in\calM_{G\times H}$ and $(W,\mu)\in\calM_{H\times G}$. It suffices to show that $\left[\frac{G\times H}{V,\psi}\right]\cdotH\left[\frac{H\times G}{W,\mu}\right]$ can be written as such a linear combination. But every summand occurring in this product is of the form $\left[\frac{G\times G}{V*W',\psi*\mu'}\right]$ for some $(W',\mu')\in\calM_{H\times G}$. We may assume that $p_1(V*W')=G$. Otherwise (i) holds and we are done. So we have $l(V*W',\psi*\mu')=(G,K,\kappa)$ for some $(K,\kappa)\in\calM_G^G$. We need to show that $(K,\kappa)\notin\calR_G$. By Proposition~\ref{first consequences}(a) and (b), we have $l(V,\psi)=(G,K',\kappa')$ with $(K',\kappa')\le (K,\kappa)$. Thus, if $(K,\kappa)\in\calR_G$ then (\ref{calR is ideal}) implies that $(K',\kappa')\in\calR_G$ and we obtain the contradiction
\begin{equation*}
  e_{(K',\kappa')}=\left[\frac{G\times H}{V,\psi}\right] \cdotH
  \left[\frac{G\times H}{V,\psi}\right]^{\mathrm{op}} \in I_G\,.
\end{equation*}

Conversely, assume that $(U,\psi)\in\calM_{G\times G}$ with $p_1(U)<G$. Then the decomposition in Proposition~\ref{decomposition} implies that $\left[\frac{G\times G}{U,\phi}\right]\in I_G$. Also, if $p_1(U)=G$ and $(K,\kappa):=l_0(U,\phi)\notin\calR_G$ then $e_{(K,\kappa)}\in I_G$ and $\left[\frac{G\times G}{U,\phi}\right] =
e_{(K,\kappa)}\cdotG \left[\frac{G\times G}{U,\phi}\right] \in I_G$.
\end{proof}

\begin{remark}\mylabel{remark on basis of IG}
Since $I_G$ is stable under the map $-^{\mathrm{op}}\colon E_G\to E_G$, we may replace the conditions (i) and (ii) in the previous lemma also by

\smallskip
(i') $p_2(U)\neq G$ or

\smallskip
(ii') $p_2(U)=G$ and $r_0(U,\phi)\notin\calR$.
\end{remark}

\begin{proposition}\mylabel{prop structure of Ebar}
{\rm (a)} One has
\begin{equation*}
  E_G^c+I_G=E_G
\end{equation*}
and
\begin{equation}\label{eqn intersection}
  E_G^c\cap I_G = \bigoplus_{\{K,\kappa\}_G\in\calM_G^G/\sim \atop (K,\kappa)\notin\calR_G}
  f_{\{K,\kappa\}_G} E_G^c\,.
\end{equation}

\smallskip
{\rm (b)} The canonical epimorphism $E_G\to\Ebar_G$ maps the subalgebra
\begin{equation*}
  \bigoplus_{\{K,\kappa\}_G\in\calR_G/\sim} f_{\{K,\kappa\}_G} E_G^c
\end{equation*}
of $E_G^c$ isomorphically onto $\Ebar_G$. 

\smallskip
{\rm (c)} For each $(K,\kappa)\in\calR_G$, the map 
\begin{equation*}
  k\Gamma_{(G,K,\kappa)}\to \fbar_{(K,\kappa)}\Ebar_G\fbar_{(K,\kappa)}\,,\quad 
  a\mapsto \fbar_{(K,\kappa)}\abar\fbar_{(K,\kappa)}\,,
\end{equation*}
is a $k$-algebra isomorphism.
\end{proposition}

\begin{proof}
(a) Lemma~\ref{Basis of IG} and Remark~\ref{remark on basis of IG} imply immediately that a standard basis element $\left[\frac{G\times G}{U,\phi}\right]$ of $E_G$ belongs to $E_G^c$ if $(U,\phi)\in\calM_{G\times G}^c$, and it belongs to $I_G$ if $(U,\phi)\notin\calM_{G\times G}^c$. This implies the first equation.

Lemma~\ref{Basis of IG} implies that $E_G^c\cap I_G$ is generated as $k$-module by the canonical basis elements $\left[\frac{G\times G}{U,\phi}\right]$ with $(U,\phi)\in\calM_{G\times G}^c$ with $l_0(U,\phi)\notin\calR_G$. Thus,
\begin{equation*}
  E_G^c\cap I_G = \bigoplus_{\{K,\kappa\}_G\in\calM_G^G/\sim \atop (K,\kappa)\notin\calR_G}
  E_G^{c,\{K,\kappa\}}\,.
\end{equation*}
But, by the property in (\ref{calR is ideal}) and by Lemma~\ref{comparing decompositions}, the latter direct sum is equal to the direct sum in Equation~(\ref{eqn intersection}).

\smallskip
Part~(b) follows immediately from Part~(a) and Corollary~\ref{Bc decomposition}, and Part~(c) follows immediately from Part~(b) and Lemma~\ref{comparing decompositions}(f).
\end{proof}

Now we are ready to classify the simple $\Ebar_G$-modules. Set
\begin{equation*}
 \calS_G=\calS_k^A(G):= \{((K,\kappa),[V])\mid (K,\kappa)\in\calR_G, [V]\in\Irr(k\Gamma_{(G,K,\kappa)})\}\,.
\end{equation*}
We call two elements $(K,\kappa,[V]),(K',\kappa',[V'])\in\calS_G$ {\em equivalent}, if $(K,\kappa)$ and $(K',\kappa')$ are $G$-linked and $[V]$ corresponds to $[V']$ via the canonical bijection $\Irr(k\Gamma_{(G,K',\kappa')})\myiso\Irr(k\Gamma_{(G,K,\kappa)})$ from \ref{covering algebra}(c). If $\calRtilde_G$ denotes a set of representatives of the linkage classes of $\calR_G$ then 
\begin{equation*}
  \calStilde_G:=\{((K,\kappa),[V])\mid (K,\kappa)\in\calRtilde_G, [V]\in\Irr(k\Gamma_{(G,K,\kappa)})\}
\end{equation*}
is a set of representatives of the equivalence classes of $\calS_G$. By the canonical isomorphism from Proposition~\ref{prop structure of Ebar}(c), we can view each simple $k\Gamma_{(G,K,\kappa)}$-module $V$ as a simple $\fbar_{(K,\kappa)}\Ebar_G\fbar_{(K,\kappa)}$-module, and we can view $\Ebar_G \fbar_{(K,\kappa)}$ as $(\Ebar_G,k\Gamma_{(G,K,\kappa)})$-bimodule.

\begin{corollary}\label{simples of Ebar 3}
With the above notation, the map
\begin{equation}\label{simples of Ebar 2}
  ((K,\kappa),[V])\mapsto \Vtilde:=\Ebar_G \fbar_{(K,\kappa)} \otimes_{k  \Gamma_{(G,K,\kappa)}} V\,.
\end{equation}
induces a bijection between the set of equivalence classes of $\calS_G$ and $\Irr(\Ebar_G)$.
\end{corollary}

\begin{proof}
By the isomorphism in Proposition~\ref{prop structure of Ebar}(b), one obtains a bijection  between $\Irr(\Ebar_G)$ and the set of pairs $((K,\kappa),[W])$ with $(K,\kappa)\in\calRtilde_G$ and $[W]\in\Irr(f_{\{K,\kappa\}_G}E_G^c)$. By the isomorphism $\rho\colon \Mat_n(k\Gamma_{(G,K,\kappa)})\to f_{\{K,\kappa\}_G}E_G^c$ used in the proof of Theorem~\ref{covering algebra structure theorem}, one obtains a bijection between $\Irr(f_{\{K,\kappa\}_G}E_G^c)$ and $\Irr(\Mat_n(k\Gamma_{(G,K,\kappa)}))$. Finally, since $\Mat_n(k\Gamma_{(G,K,\kappa)})$ and $k\Gamma_{(G,K,\kappa)}$ are Morita equivalent, $\Irr(\Mat_n(k\Gamma_{(G,K,\kappa)}))$ is in bijection with $\Irr(k\Gamma_{(G,K,\kappa)})$. If $(K,\kappa)\in\calRtilde_G$, $[V]\in\Irr(k\Gamma_{(G,K,\kappa}))$, and $(K,\kappa)=(K_i,\kappa_i)$ in the enumeration of $\{K,\kappa\}_G$ in Theorem~\ref{covering algebra structure theorem}, then we use the Morita equivalence given by tensoring with the bimodule $\Mat_n(k\Gamma_{(G,K,\kappa)})e_i$, where $e_i:=\dia(0,\ldots,0,1,0,\ldots,0)$, and $k\Gamma_{(G,K,\kappa)}$ is identified with $e_i\Mat_n(k\Gamma_{(G,K,\kappa)})e_i$ via $a\mapsto e_iae_i$. Thus, $[V]\in\Irr(k\Gamma_{(G,K,\kappa)})$ corresponds to the class of $\Mat_n(k\Gamma_{(G,K,\kappa)})e_i\otimes_{k\Gamma_{(G,K,\kappa)}} V$ in $\Irr(\Mat_n(k\Gamma_{(G,K,\kappa)}))$. Choosing $x_i:=e_{(K,\kappa)}$ in the definition of the isomorphism $\rho$ (in the proof of Theorem~\ref{covering algebra structure theorem}), the isomorphism $\rho$ transports this latter irreducible module to the irreducible module $f_{\{K,\kappa\}_G}E_G^c f_{(K,\kappa)}\otimes_{k\Gamma_{(G,K,\kappa)}} V = E_G^c f_{(K,\kappa)}\otimes_{k\Gamma_{(G,K,\kappa)}} V$, since $f_{\{K,\kappa\}_G}E_G^c f_{(K,\kappa)}=E_G^c f_{\{K,\kappa\}_G} f_{(K,\kappa)}= E_G^c f_{(K,\kappa)}$. And finally, via the isomorphism in Proposition~\ref{prop structure of Ebar}(b), the latter simple module corresponds to the module $\Vtilde$ in the statement of the corollary.
\end{proof}

The following proposition gives a necessary condition and a sufficient condition for a pair $(K,\kappa)\in\calM_G^G$ to be reduced. For special cases of $A$, we will see in Section~\ref{sec completely reduced pairs} that the sufficient condition in Part~(b) is also necessary. 

\begin{proposition}\mylabel{conditions for reduced}
Let $(K,\kappa)\in\calM_G^G$.

\smallskip
{\rm (a)} If $(K,\kappa)\in\calR_G$ then $\kappa$ is faithful and $K\le Z(G)$. 

\smallskip
{\rm (b)} If $\kappa$ is faithful and $K\le G'$ then $(K,\kappa)\in\calR_G$.
\end{proposition}

\begin{proof}
(a) By the decomposition in Proposition~\ref{decomposition} we obtain that $\ker(\kappa)=1$. Then, $K\le Z(G)$ by Proposition~\ref{eta and zeta}(b).

\smallskip
(b) Assume that $e_{(K,\kappa)}\in I_G$. Then, by the definition of $I_G$ there exists a finite group $H$ with $|H|<|G|$ and pairs $(V,\psi)\in\calM_{G\times H}$ and $(W,\mu)\in\calM_{H\times G}$ such that the standard basis element $e_{(K,\kappa)}$ occurs as a summand in $\left[\frac{G\times H}{V,\psi}\right]\cdotH \left[\frac{H\times G}{W,\mu}\right]$. By the Mackey formula in Corollary~\ref{Mackey formula for fibered bisets 2}, it follows that there exists a pair $(W',\mu')\in\calM_{H\times G}$ such that $(\Delta_K(G),\phi_\kappa)=(V*W',\psi*\mu')$. By
Proposition~\ref{first consequences}, it follows that $l(V,\psi)=(G,K',\kappa')$ with $(K',\kappa')\le(K,\kappa)$. Now, by Proposition~\ref{idempotent proposition}(b), we have $e_{(K,\kappa)}\cdotG\left[\frac{G\times H}{V,\psi}\right] = \left[\frac{G\times H}{V',\psi'}\right]$ with $(V',\psi')\in\calM_{G\times H}$ satisfying $l(V',\psi')=(G,K,\kappa)$. Proposition~\ref{eta and zeta}(e) applied to $(V',\psi')$ implies that $|G|\le|H|$, a contradiction.  
\end{proof}

\begin{example}
The converse of Part~(b) in the previous Proposition does not hold in general: Let $A:=\{\pm 1\}$, $G$ the cyclic group of order $4$, $K$ its subgroup of order $2$, and $\kappa\colon K\to A$ the unique injective homomorphism. Then an easy computation shows that $e_{(K,\kappa)}\notin I_G$. Thus, $(K,\kappa)$ is reduced in $G$, but $K$ is not contained in $G'=\{1\}$.
\end{example}

We conclude this section with some additional results around the notion of being reduced.

\begin{proposition}\label{prop reduced and linked}
Let $G$ be a finite group and let $(K,\kappa)\in\calM_G^G$. The following are equivalent:

\smallskip
{\rm (i)} $(K,\kappa)\notin\calR_G$.

\smallskip
{\rm  (ii)} There exist a finite group $H$ with $|H|<|G|$ and $(L,\lambda)\in\calM_H^H$ such that $(G,K,\kappa)\sim(H,L,\lambda)$.
\end{proposition}

\begin{proof}
(i) $\Rightarrow$ (ii): If $(K,\kappa)\notin\calR_G$ then $e_{(K,\kappa)}\in I_G$ and there exists a group $H$ with $|H|<|G|$ and pairs $(U,\phi)\in\calM_{G\times H}$, $(V,\psi)\in\calM_{H\times G}$ such that the standard basis element $e_{(K,\kappa)}$ of $B_k^A(G,G)$ occurs with non-zero coefficient in the tensor product $\left[\frac{G\times H}{U,\phi}\right] \cdotH \left[\frac{H\times G}{V,\psi}\right]$. The Mackey formula and Proposition~\ref{first consequences}(a),(b) imply that $p_1(U)=G$ and that $(K,\kappa)\ge l_0(U,\phi)$. We set $\Htilde:=p_2(U)$ and write $e_{(K,\kappa)}\cdotG \left[\frac{G\times \Htilde}{U,\phi}\right] = \left[\frac{G\times \Htilde}{\Utilde,\phitilde}\right]$ with $(\Utilde,\phitilde)\in\calM_{G\times\Htilde}$. Then
\begin{equation*}
  |\Htilde|\le |H|<|G|,\quad p_2(\Utilde)=\Htilde,\quad\text{and}\quad l(\Utilde,\phitilde)=(G,K,\kappa)
\end{equation*}
by Proposition~\ref{idempotent proposition}(b). Thus, $(G,K,\kappa)\mathop{\sim}\limits_{(\Utilde,\phitilde)} (\Htilde,\Ltilde,\tilde{\lambda})$ with $(L,\tilde{\lambda}):=r(\Utilde,\phitilde)\in\calM_{\Htilde}^{\Htilde}$.

\smallskip
(ii)$\Rightarrow$(i): Let $(U,\phi)\in\calM_{G\times H}$ with $l(U,\phi)=(G,K,\kappa)$ and $r(U,\phi)=(H,L,\lambda)$. Then Proposition~\ref{idempotent proposition}(d), implies $e_{(K,\kappa)}= \left[\frac{G\times H}{U,\phi}\right]\cdotH \left[\frac{G\times H}{U,\phi}\right]^{\mathrm{op}}\in I_G$, and therefore $(K,\kappa)\notin\calR_G$.
\end{proof}

Since the condition in Proposition~\ref{prop reduced and linked}(ii) is independent of $k$, we have the following immediate corollary.

\begin{corollary}\label{cor calR independent of k}
Let $G$ and $H$ be finite groups. The subset $\calR_k^A(G)$ of $\calM_G$ is independent of $k$. 
\end{corollary}

\begin{definition}\label{def twosided reduced}
Let $G$ and $H$ be finite groups. A pair $(U,\phi)\in\calM_{G\times M}^c$ is called {\em reduced} if $l_0(U,\phi)\in\calR_G$ and $r_0(U,\phi)\in\calR_H$. By Corollary~\ref{cor calR independent of k}, this notion does not depend on $k$.
\end{definition}

The following proposition justifies this terminology.

\begin{proposition}\label{prop twosided reduced}
Let $G$ and $H$ be finite groups, let $(U,\phi)\in\calM^c_{G\times H}$, and set $(K,\kappa):=l_0(U,\phi)$ and $(L,\lambda):=r_0(U,\phi)$. The following are equivalent:

\smallskip
{\rm (i)} $(U,\phi)$ is reduced.

\smallskip
{\rm (ii)} $|G|=|H|$ and $\left[\frac{G\times H}{U,\phi}\right]$ does not factor through a group of order smaller than $|G|=|H|$, i.e., there exists no group $I$ of order smaller than $|G|=|H|$ such that $\left[\frac{G\times H}{U,\phi}\right]\in B^A_k(G,I)\mathop{\cdot}\limits_{I} B^A_k(I,H)$.
\end{proposition}

\begin{proof}
(i)$\Rightarrow$(ii):  Since $\left[\frac{G\times H}{U,\phi}\right]\cdotH\left[\frac{G\times H}{U,\phi}\right]^{\mathrm{op}} = e_{(K,\kappa)}$ by Proposition~\ref{idempotent proposition}(d), and since $(K,\psi)\in\calR_G$, we obtain $|H|\ge|G|$. Similarly one proves that $|G|\ge|H|$ to obtain $|G|=|H|$. If $\left[\frac{G\times H}{U,\phi}\right]$ factors through a group of order smaller than $|G|$, so does $e_{(K,\kappa)}=\left[\frac{G\times H}{U,\phi}\right]\cdotH\left[\frac{G\times H}{U,\phi}\right]^{\mathrm{op}}$ which contradicts $(K,\psi)\in\calR_G$.

(ii)$\Rightarrow$(i): By Proposition~\ref{idempotent proposition}(c) we have $\left[\frac{G\times H}{U,\phi}\right]=e_{(K,\psi)}\cdotG \left[\frac{G\times H}{U,\phi}\right] \cdotH e_{(L,\lambda)}$. Now, (ii) immediately implies (i).
\end{proof}

%%%%%%%%%%%%%%%%%%%%%% SECTION 9 %%%%%%%%%%%%%%%%%%%%%%%%%%%%%%%%%%%%%

\section{Simple $A$-Fibered Biset Functors}\label{sec simple functors}

We keep the notation from the previous sections: $A$ is an abelian
group, $k$ is a commutative ring, $\calC=\calC_k^A$ is the category introduced in Section~\ref{sec fibered biset functors}, and $\calF=\calF_k^A$ is the category of $k$-linear functors from $\calC_k^A$ to $\lMod{k}$ . For a finite group $G$, we continue to write $E_G:=B_k^A(G,G)=\End_{\calC}(G,G)$ and $\Ebar_G=E_G/I_G$, where $I_G$ denotes the ideal of $E_G$ of all morphisms factoring through groups of order strictly smaller than $|G|$. The goal of this section is the classification of the simple objects in $\calF$.

\begin{nothing}\mylabel{quadruples} Recall from
Proposition~\ref{simple functor proposition} that for every finite group
$G$ and every simple $E_G$-module $V$ one obtains a simple functor $S_{G,V}\in\calF$ and that every
simple functor arises that way. Also, recall from Corollary~\ref{simples of Ebar 3} that if
$(K,\kappa)\in\calR_G$ and if $V$ is an irreducible
$k\Gamma_{(G,K,\kappa)}$-module then
\begin{equation}\label{Vtilde}
  \Vtilde:= \Ebar_G \fbar_{(K,\kappa)}\otimes_{k\Gamma_{(G,K,\kappa)}} V
\end{equation}
is an irreducible $E_G$-module. Therefore, we obtain a
simple functor
\begin{equation*}
  S_{(G,K,\kappa,V)}:=S_{G,\Vtilde}\in\calF\,.
\end{equation*}
It is clear that if $V'\cong V$ as $k\Gamma_{(G,K,\kappa)}$-modules then
$\Vtilde\cong\Vtilde'$ as $E_G$-modules,
$L_{G,\Vtilde}\cong L_{G,\Vtilde'}$ as functors in $\calF$, and
therefore, $S_{G,\Vtilde}\cong S_{G,\Vtilde'}$ in $\calF$ as the
unique simple factors of $L_{G,\Vtilde}$ and $L_{G,\Vtilde'}$,
respectively. We denote by
$\Irr(\calF)$ the set of isomorphism classes $[S]$ of
simple functors $S\in\calF$. Moreover, we set
\begin{equation*}
  \calS=\calSkA:=\{(G,K,\kappa,[V])\mid G\in\mathrm{Ob}(\calC),
  (K,\kappa)\in\calR_G, [V]\in\Irr(k\Gamma_{(G,K,\kappa)})\}\,.
\end{equation*}
Two quadruples $(G,K,\kappa,[V])$ and $(H,L,\lambda,[W])$ in
$\calS$ will be called {\em linked} if
$(G,K,\kappa)\sim(H,L,\lambda)$, cf.~Definition~\ref{linkage definition}(a), and
\begin{equation*}
  V\cong k[\lGamma{(G,K,\kappa)}_{(H,L,\lambda)}]\otimes_{k\Gamma_{(H,L,\lambda)}} W
\end{equation*}
as $k\Gamma_{(G,K,\kappa)}$-modules, cf.~\ref{covering algebra}(c). In this case we write
$(G,K,\kappa,[V])\sim(H,L,\lambda,[W])$. Note that linkage is an
equivalence relation on $\calS$. We denote by $\calSbar=\calSbarkA$ the set
of linkage classes $[G,K,\kappa,[V]]$ of $\calS$. By all the above,
we have now defined a function
\begin{equation}\label{eqn omega}
  \omega\colon \calS\to\Irr(\calF)\,,\quad
  (G,K,\kappa,[V])\mapsto [S_{(G,K,\kappa,V)}]\,.
\end{equation}
\end{nothing}

The goal of this section is to prove the following theorem.

\begin{theorem}\mylabel{parametrization of simples}
The function $\omega$ from (\ref{eqn omega}) induces a
bijection
\begin{equation*}
  \omegabar\colon \calSbar \to \Irr(\calF)\,,\quad
  [G,K,\kappa,[V]]\mapsto [S_{(G,K,\kappa,V)}]\,.
\end{equation*}
\end{theorem}

Before we can prove the theorem we will need three Lemmas as preparation. 

\begin{lemma}\label{lem simple functors 1}
Let $G$ and $H$ be finite groups and let $(K,\kappa)\in\calM_G^G$ and $(L,\lambda)\in\calM_H^H$.

\smallskip
{\rm (a)} If $M\in \calF$ and $\lGamma{(G,K,\kappa)}_{H,L,\lambda}\neq \emptyset$ then, for every $x\in \lGamma{(G,K,\kappa)}_{H,L,\lambda}$, the map
\begin{equation*}
  f_{(L,\lambda)} M(H) \to f_{(K,\kappa)} M(G)\,,\quad m\mapsto xm\,,
\end{equation*}
is an isomorphism of $k$-modules with inverse $m\mapsto \xop m$.

\smallskip
{\rm (b)} Assume that $(K,\kappa)\in\calR_G$ and $(L,\lambda)\in\calR_H$ and that $(G,K,\kappa)\sim(H,L,\lambda)$. Then $|G|=|H|$.
\end{lemma}

\begin{proof}
(a) This follows immediately from the relations 
\begin{equation}\label{eqn x e f relations}
  x \xop = e_{(K,\kappa)},\quad \xop x = e_{(L,\lambda)},\quad e_{(K,\kappa)}f_{(K,\kappa)} = f_{(K,\kappa)},\quad
  e_{(L,\lambda)}f_{(L,\lambda)}=f_{(L,\lambda)}.
\end{equation}

\smallskip
(b) Let $(U,\phi)\in \calM_{G\times H}$ satisfy $l(U,\phi)=(G,K,\kappa)$ and $r(U,\phi)=(H,L,\lambda)$ and assume that $|G|>|H|$. Then $e_{(K,\kappa)}=\left[\frac{G\times H}{U,\phi}\right]\cdotH \left[\frac{G\times H}{U,\phi}\right]^{\rm{op}} = e_{(K,\kappa)}\in I_G$, contradicting $(K,\kappa)\in\calR_G$. Similarly, we obtain a contradiction if we assume $|G|<|H|$.
\end{proof}

\begin{lemma}\label{lem simple functors 2}
Let $(H,L,\lambda,[W])\in\calS$ and let $\Wtilde$ be the irreducible $\Ebar_H$-module defined as in (\ref{Vtilde}).

\smallskip
{\rm (a)}  If $G$ is a finite group with $|G|=|H|$ then $I_G\cdot S_{H,\Wtilde}(G)=\{0\}$ and $S_{H,\Wtilde}(G)$ can be viewed as $\Ebar_G$-module.

\smallskip
{\rm (b)} One has a $k\Gamma_{(H,L,\lambda)}$-module isomorphism $f_{(L,\lambda)}S_{H,\Wtilde}(H)\cong W$. Here, $f_{(L,\lambda)}S_{H,\Wtilde}(H)$ is viewed as $k\Gamma_{(H,L,\lambda)}$-module via the $k$-algebra isomorphism $k\Gamma_{(H,L,\lambda)} \to f_{(L,\lambda)} E_H^c f_{(L,\lambda)}$ from Lemma~\ref{comparing decompositions}(f).
\end{lemma}

\begin{proof}
(a) We may assume that $S_{H,\Wtilde}(G)\neq \{0\}$. Then, with $H$, also $G$ is minimal for $S_{H,\Wtilde}$, since $|G|=|H|$. The result now follows from Proposition~\ref{simple functor proposition}(b).

\smallskip
(b) Recall from Lemma~\ref{lem J} that $S_{H,\Wtilde}(H)\cong \Wtilde = \Ebar_H\fbar_{(L,\lambda)}\otimes_{k\Gamma_{(H,L,\lambda)}} W$ as $E_H$-modules. Thus, using Part~(a) for $G=H$, we have $f_{(L,\lambda)}S_{H,\Wtilde}(H)=\fbar_{(L,\lambda)} S_{H,\Wtilde}(H)\cong \fbar_{(L,\lambda)}\Ebar_H\fbar_{(L,\lambda)}\otimes_{k\Gamma_{(H,L,\lambda)}} W$ as $\fbar_{(L,\lambda)}\Ebar_H\fbar_{(L,\lambda)}$-modules. Using the $k$-algebra isomorphism in Proposition~\ref{prop structure of Ebar}(c), we obtain the desired isomorphism.
\end{proof}

Let $G$ and $H$ be finite groups. Generalizing the notation from~\ref{covering algebra}(a), we define $B_k^{A,c}(G,H)$ as the $k$-span of the standard basis elements $\left[\frac{G\times H}{U,\phi}\right]$ of $B_k^A(G,H)$ with $(U,\phi)\in\calM_{G\times H}^c$, i.e., such that $p_1(U)=G$ and $p_2(U)=H$.

\begin{lemma}\label{lem simple functors 3}
Let $G$ and $H$ be finite groups and let $(K,\kappa)\in\calM_G^G$ and $(L,\lambda)\in\calM_H^H$.

\smallskip
{\rm (a)} If $|G|=|H|$ then the map
\begin{equation*}
  \alpha\colon k[\lGamma{(G,K,\kappa)}_{(H,L,\lambda)}] \to f_{(K,\kappa)} B^{A,c}_k(G, H) f_{(L,\lambda)},\quad 
  b\mapsto f_{(K,\kappa)} b f_{(L,\lambda)}\,,
\end{equation*}
is an isomorphism of $(k\Gamma_{(G,K,\kappa)},k\Gamma_{(H,L,\lambda)})$-bimodules, where $f_{(K,\kappa)} B^{A,c}_k(G,H) f_{(L,\lambda)}$ is viewed as $(k\Gamma_{(G,K,\kappa)},k\Gamma_{(H,L,\lambda)})$-bimodule via the isomorphism from Lemma~\ref{comparing decompositions}(f).

\smallskip
{\rm (b)} Assume that $(H,L,\lambda,[W])\in\calS$ and that $|G|=|H|$. Then there exists an epimorphism
\begin{equation*}
  k[\lGamma{(G,K,\kappa)}_{(H,L,\lambda)}] \otimes_{k\Gamma_{(H,L,\lambda)}} W  \to  f_{(K,\kappa)} S_{H,\Wtilde}(G)\end{equation*}
of $k\Gamma_{(G,K,\kappa)}$-modules, where $f_{(K,\kappa)} S_{H,\Wtilde}(G)$ is viewed as $k\Gamma_{(G,K,\kappa)}$-module via the isomorphism from Lemma~\ref{comparing decompositions}(f).
\end{lemma}

\begin{proof}
(a) First we treat the case that $\lGamma{(G,K,\kappa)}_{(H,L,\lambda)}$ is non-empty and we pick an element $x$ from it. Consider the diagram of left $k\Gamma_{(G,K,\kappa)}$-module homomorphisms,
\begin{diagram}
  \movevertex(-40,0){k[\lGamma{(G,K,\kappa)}_{(H,L,\lambda)}]} & 
  \movearrow(5,0){\Ear[70]{-\cdotH \xop}} & 
  \movevertex(30,0){k\Gamma_{(G,K,\kappa)}} &&
  \movearrow(-40,0){\Sar{\alpha}} & & 
  \movearrow(30,0){\saR{\wr}} &&
  \movevertex(-50,0){f_{(K,\kappa)} B_k^{A,c}(G,H) f_{(L,\lambda)}} & 
  \Ear[60]{-\cdotH \xop f_{(K,\kappa)}} & 
  \movevertex(30,0){f_{(K,\kappa)} E_G^c f_{(K,\kappa)}} &&
\end{diagram}
where the right vertical map is the isomorphism $a\mapsto f_{(K,\kappa)} a f_{(K,\kappa)}$ from Lemma~\ref{comparing decompositions}(f). By Lemma~\ref{lem delete f} we have
\begin{equation}\label{eqn x f relations}
  f_{(K,\kappa)} x f_{(L,\lambda)} = x f_{(L,\lambda)}\quad\text{and}\quad 
  f_{(L,\lambda)} \xop f_{(K,\kappa)} = \xop f_{(K,\kappa)}\,.
\end{equation}
This implies that the diagram is commutative. And together with the relations in (\ref{eqn x e f relations}) it implies that the top horizontal map is an isomorphism with inverse $-\cdotG x$ and that the bottom horizontal map is an isomorphism with inverse $-\cdotG x f_{(L,\lambda)}$. Thus, also $\alpha$ is an isomorphism.

Next assume that $\lGamma{(G,K,\kappa)}_{(H,L,\lambda)}$ is empty. Let $(U,\phi)\in\calM^c_{G\times H}$ and set $b:=\left[\frac{G\times H}{U,\phi}\right] \in B_k^{A,c}(G,H)$, a general standard basis element. It suffices to show that $f_{(K,\kappa)} b f_{(L,\lambda)} = 0$. Set $(K',\kappa'):=l_0(U,\phi)\in \calM_G^G$ and $(L',\lambda'):=r_0(U,\phi)\in \calM_H^H$.
Then $b=e_{(K',\kappa')} b e_{(L',\lambda')}$. By Proposition~\ref{e,f-relations 1}, we may assume that $(K',\kappa')\le (K,\kappa)$ and $(L',\lambda')\le (L,\lambda)$. We will show that a pair $(U,\phi)$ with these conditions cannot exist. Assume it does. Then \begin{equation*}
   e_{(K,\kappa)}\left[\frac{G\times H}{U,\phi}\right] e_{(L,\lambda)} =  \left[\frac{G\times H}{V,\psi}\right]
\end{equation*}
for some $(V,\psi)\in\calM_{G\times H}$ satisfying $l(V,\psi)=(G,K,\kappa)$ and $l(V,\psi)=(H,L,\lambda)$ by Proposition~\ref{idempotent proposition}(b). Thus $(V,\psi)\in\lGamma{(G,K,\kappa)}_{(H,L,\lambda)}$, a contradiction.

\smallskip
(b) Recall from \ref{noth LGV} that, for every finite group $I$, the evaluation $L_{H,\Wtilde}(I)$ is given by 
\begin{equation*}
  L_{H,\Wtilde}(I)=B_k^A(I,H)\otimes_{E_H}\Wtilde = 
  B_k^A(I,H)\otimes_{E_H} \Ebar_H \fbar_{(L,\lambda)}\otimes_{k\Gamma_{(H,L,\lambda)}} W\,.
\end{equation*}
Similar to the functor $L_{H,\Wtilde}$ we define a functor $M_{H,W}\in\calF$ via
\begin{equation*}
  M_{H,W}(I)= B_k^A(I,H)\otimes_{E_H} E_H f_{(L,\lambda)}\otimes_{k\Gamma_{(H,L,\lambda)}} W 
  = B_k^A(I,H) f_{(L,\lambda)}\otimes_{k\Gamma_{(H,L,\lambda)}} W \,.
\end{equation*}
The functor is defined on morphisms in the same way as $L_{H,\Wtilde}$, namely via composition in the category $\calC$ from the left. Since the $(E_H,k\Gamma_{(H,L,\lambda)})$-bimodule $\Ebar_H \fbar_{(L,\lambda)}$ is a factor module of the $(E_H,k\Gamma_{(H,L,\lambda)})$-bimodule  $E_H f_{(L,\lambda)}$, we obtain an epimorphism of functors $M_{H,W}\to L_{H,\Wtilde}$. Composing it with the natural epimorphism $L_{H,\Wtilde}\to S_{H,\Wtilde}$ we obtain an epimorphism $\pi\colon M_{H,W}\to S_{H,\Wtilde}$ of $A$-fibered biset functors. In particular we have an epimorphism
\begin{equation*}
  \pi_G\colon B_k^A(G,H)f_{(L,\lambda)}\otimes_{k\Gamma_{(H,L,\lambda)}} W \to S_{H,\Wtilde}(G)
\end{equation*}
of $E_G$-modules. If $(U,\phi)\in\calM_{G\times H}\smallsetminus\calM^c_{G\times H}$ then $\pi_G(\left[\frac{G\times H}{U,\phi}\right]\otimes w) = 0$ for all $w\in W$. In fact, this follows from the decomposition of $\left[\frac{G\times H}{U,\phi}\right]$ as in   Proposition~\ref{decomposition} and from $|G|=|H|$, since $S_{H,\Wtilde}(I)=\{0\}$ for all finite groups $I$ with $|I|< |G|=|H|$. Thus, we also obtain an epimorphism
\begin{equation*}
  B_k^{A,c}(G,H)f_{(L,\lambda)}\otimes_{k\Gamma_{(H,L,\lambda)}} W \to S_{H,\Wtilde}(G)
\end{equation*}
of $E_G^c$-modules, and after multiplying with the idempotent $f_{(K,\kappa)}$ we obtain an epimorphism
\begin{equation*}
  f_{(K,\kappa)} B_k^{A,c}(G,H)f_{(L,\lambda)}\otimes_{k\Gamma_{(H,L,\lambda)}} W \to f_{(K,\kappa)} S_{H,\Wtilde}(G)
\end{equation*}
of $f_{(K,\kappa)} E_G^c f_{(K,\kappa)}$-modules. Using the isomorphism from Part~(a) and the $k$-algebra isomorphism $k\Gamma_{(G,K,\kappa)} \to f_{(K,\kappa)} E_G^c f_{(K,\kappa)}$ from Proposition~\ref{prop structure of Ebar}(c), we obtain the desired epimorphism of $k\Gamma_{(G,K,\kappa)}$-modules.
\end{proof}

\begin{nothing}{\bf Proof of Theorem~\ref{parametrization of simples}.}
(a) First we show that the map $\omega$ in (\ref{eqn omega}) is surjective.
So let $S\in\calF$ be simple. Choose a finite group
$G$ of minimal order  such that $S(G)\neq\{0\}$. Then, by
Proposition~\ref{simple functor proposition}(a), $U:=S(G)$ is an
irreducible $E_G$-module and $S\cong S_{G,U}$. Since
$S(H)=0$ for all finite groups $H$ with $|H|<|G|$, the module $U$ is
annihilated by $I_G$ and comes via
inflation from an irreducible $\Ebar_G$-module which we again denote by $U$. By Corollary~\ref{simples of Ebar 3}, we obtain $U\cong \Ebar_G \fbar_{(K,\kappa)}
\otimes_{k\Gamma_{(G,K,\kappa)}} V$ for some $(K,\kappa)\in\calR_G$
and some irreducible $k\Gamma_{(G,K,\kappa)}$-module $V$. Thus, $U\cong \Vtilde$ as defined in (\ref{Vtilde}) and $S\cong S_{G,U}\cong S_{G,\Vtilde}=S_{(G,K,\kappa,V)}$.

\smallskip
(b) Next we show that if $(G,K,\kappa,[V]), (H,L,\lambda,[W])\in\calS$ are linked then $S_{(G,K,\kappa,V)}\cong S_{(H,L,\lambda,W)}$. First, $(G,K,\kappa,[V])\sim(H,L,\lambda,[W])$ implies $|G|=|H|$, by Lemma~\ref{lem simple functors 1}(b). Further, by Lemma~\ref{lem simple functors 2}(a), $S_{H,\Wtilde}(G)$ is an $\Ebar_G$-module. Since $f_{(L,\lambda)}S_{H,\Wtilde}(H)\cong W$, by Lemma~\ref{lem simple functors 2}(b), and since $(G,K,\kappa)\sim(H,L,\lambda)$, we obtain $\fbar_{(K,\kappa)}S_{H,\Wtilde}(G)\neq0$, by Lemma~\ref{lem simple functors 1}(a). Moreover, since $(G,K,\kappa,[V])\sim (H,L,\lambda,[W])$, we have $V\cong k[\lGamma{(G,K,\kappa)}_{(H,L,\lambda)}] \otimes_{k\Gamma_{(H,L,\lambda)}} W$ and Lemma~\ref{lem simple functors 3}(b) implies $\Hom_{k\Gamma_{(G,K,\kappa)}}(V,\fbar_{(K,\kappa)}S_{H,\Wtilde}(G))\neq\{0\}$. But
\begin{align*} 
  & \Hom_{k\Gamma_{(G,K,\kappa)}}\bigl(V,\fbar_{(K,\kappa)}S_{H,\Wtilde}(G)\bigr)
  \cong  \Hom_{k\Gamma_{(G,K,\kappa)}}\bigl(V,\Hom_{\Ebar_G}(\Ebar_G \fbar_{(K,\kappa)}, S_{H,\Wtilde}(G))\bigr) \\
  & \qquad \cong \ \Hom_{\Ebar_G}\bigl(\Ebar_G \fbar_{(K,\kappa)}\otimes_{k\Gamma_{(G,K,\kappa)}} V, S_{H,\Wtilde}(G)\bigr) 
  \cong  \Hom_{\Ebar_G} (\Vtilde, S_{H,\Wtilde}(G))\,.
\end{align*}
Thus, $\Hom_{\Ebar_G}(\Vtilde,S_{H,\Wtilde}(G))\neq\{0\}$. Since both $\Vtilde$ and $S_{H,\Wtilde}(G)$ are simple $E_G$-modules, we obtain $\Vtilde\cong S_{H,\Wtilde}(G)$. Now Proposition~\ref{simple functor proposition}(d) implies that $S_{G,\Vtilde}=S_{H,\Wtilde}$.
  
\smallskip
(c) Finally, we show that if $(G,K,\kappa,[V]),(H,L,\lambda,[W])\in\calS$ satisfy $S_{(G,K,\kappa,[V])}\cong S_{(H,L,\lambda,[W])}$ then $(G,K,\kappa,V)\sim(H,L,\lambda,W)$. Note that, since $G$ is a minimal subgroup for $S_{G,\Vtilde}$ and $H$ is a minimal subgroup for $S_{H,\Wtilde}$ (by Proposition~\ref{simple functor proposition}(d)), we have $|G|=|H|$. Thus, $I_G$ annihilates $S_{H,\Wtilde}(G)$, and $S_{H,\Wtilde}(G)$ is an $\Ebar_G$-module, by Lemma~\ref{lem simple functors 2}(a). Since $S_{G,\Vtilde}\cong S_{H,\Wtilde}$, we have isomorphisms $\fbar_{(K,\kappa)} S_{H,\Wtilde}(G) \cong \fbar_{(K,\kappa)} S_{G,\Vtilde}(G) \cong V$ as $k\Gamma_{(G,K,\kappa)}$-modules, by Proposition~\ref{lem simple functors 2}(b). By Lemma~\ref{lem simple functors 3}(b), there exists an epimorphism $k[\lGamma{(G,K,\kappa)}_{(H,L,\lambda)}]\otimes_{k\Gamma_{(H,L,\lambda)}} W\to \fbar_{(K,\kappa)} S_{H,\Wtilde}(G)\cong V$ of $k\Gamma_{(G,K,\kappa)}$-modules. In particular, $\lGamma{(G,K,\kappa)}_{(H,L,\lambda)}\neq\emptyset$. Since $k[\lGamma{(G,K,\kappa)}_{(H,L,\lambda)}] \otimes_{k\Gamma_{(H,L,\lambda)}} W$ is a simple $k\Gamma_{(G,K,\kappa)}$-module (see \ref{covering algebra}(c)), we obtain $k[\lGamma{(G,K,\kappa)}_{(H,L,\lambda)}] \otimes_{k\Gamma_{(H,L,\lambda)}} W \cong V$ as $k\Gamma_{(G,K,\kappa)}$-modules. Thus, $(G,K,\kappa, [V])\sim(H,L,\lambda,[W])$, and the proof of Theorem~\ref{parametrization of simples} is complete.
\qed
\end{nothing}

The next proposition shows that the evaluation $S(H)$ of a simple functor $S$ parametrized by the quadruple $(G,K,\kappa,[V])$ vanishes, unless $H$ has a section that is related to the triple $(G,K,\kappa)$ in very strong sense.

\begin{proposition}\mylabel{non-vanishing of simple functors}
Let $(G,K,\kappa,[V])\in\calS$ and let $\Vtilde$ be the associated
irreducible $E_G$-module from (\ref{Vtilde}). Assume that $H$ is a finite group such that $S_{G,\Vtilde}(H)\neq \{0\}$. Then there exist subgroups
$H_1\trianglelefteq H_2\le H$ such that $I:=H_2/H_1$ has the following property: There exists a pair
$(L,\lambda)\in\calM_I^I$ with faithful $\lambda$, such that 
$(G,K,\kappa)\sim(I,L,\lambda)$, 
$G/K\cong I/L$, $K\cap G'\cong L\cap I'$, and $|I|\ge|G|$.
\end{proposition}

\begin{proof} Recall that $S_{G,\Vtilde}=L_{G,\Vtilde}/J_{G,\Vtilde}$.
Since $S_{G,\Vtilde}(H)\neq \{0\}$, we have a proper inclusion
\begin{equation*}
  J_{G,\Vtilde}(H)\subset L_{G,\Vtilde}(H)= B_k^A(H,G)\otimes_{E_G} \Vtilde 
  = B_k^A(H,G)\otimes_{E_G} \Ebar_G\fbar_{(K,\kappa)} \otimes_{k\Gamma_{(G,K,\kappa)}} V\,.
\end{equation*}
By the explicit description of $J_{G,\Vtilde}(H)$ in Lemma~\ref{lem J}, and since $\Ebar_G\fbar_{(K,\kappa)}=\fbar_{\{K,\kappa\}}\Ebar_G\fbar_{(K,\kappa)}$, there exist elements $(U,\phi)\in\calM_{G\times H}$ and $(W,\psi)\in\calM_{H\times G}$ such that the standard basis elements 
\begin{equation*}
  x:=\left[\frac{G\times H}{U,\phi}\right]\in B_k^A(G,H)\quad\text{and}\quad 
  y:=\left[\frac{H\times G}{W,\psi}\right]\in B_k^A(H,G)
\end{equation*}
satisfy $x\cdotH y\cdotG f_{\{K,\kappa\}}\notin I_G$. This implies that there exists $(\Ktilde,\kappatilde)\in\{K,\kappa\}_G$ such that $x\cdotH y\cdotG f_{(\Ktilde,\kappatilde)}\notin I_G$. By the decomposition in Proposition~\ref{decomposition} we obtain immediately that $r(W,\psi)=(G,\Kttilde,\kappattilde)$ with faithful $\kappattilde$. Moreover, since $x\cdotH y\cdotG f_{(\Ktilde,\kappatilde)}\notin I_G$ and $y =y\cdotG e_{(\Kttilde,\kappattilde)}$, Proposition~\ref{e,f-relations 1} implies that $(\Kttilde,\kappattilde)\le(\Ktilde,\kappatilde)$. Again by Proposition~\ref{e,f-relations 1}, we have $y \cdotG f_{(\Ktilde,\kappatilde)} = y\cdotG e_{(\Ktilde,\kappatilde)} f_{(\Ktilde,\kappatilde)}$. Thus, replacing $y$ with the standard basis element $y\cdotG e_{(\Ktilde,\kappatilde)}$ and using the dual version of Proposition~\ref{idempotent proposition}(b), we may assume that $r(W,\psi)=(G,\Ktilde,\kappatilde)$ with $(\Ktilde,\kappatilde)\in\{K,\kappa\}_G$. Since $p_2(W)=G$ and $\kappa$ is faithful, one can decompose $y$ according to Proposition~\ref{decomposition} as
\begin{equation*}
  y=\Ind_{H_2}^H\mathop{\cdot}\limits_{H_2} \Inf_{I}^{H_2} \mathop{\cdot}\limits_{I} 
  \left[\frac{I\times G}{\Wtilde,\psitilde}\right]
\end{equation*}
with $H_2:=p_1(W)$ and $H_1:=\ker(\psi_1)$, where $\psi_1\in k_1(W)^*$ is defined as in \ref{noth M_G}, and $I:=H_2/H_1$. Note that $l(\Wtilde,\psitilde)=(I,L,\lambda)$ with faithful $\lambda\in L^*$ and that $r(\Wtilde,\psitilde)=r(W,\psi)=(G,\Ktilde,\kappatilde)$. Thus we have $(I,L,\lambda)\sim(G,\Ktilde,\kappatilde)\sim(G,K,\kappa)$. Parts (c) and (d) of Proposition~\ref{eta and zeta} now imply that $H/L\cong G/K$ and that $L\cap I'\cong K\cap G'$. Finally, since $x\cdotH y\cdotG f_{\{K,\kappa\}}\notin I_G$, and $y$ factors through $I$ we also have $|G|\le |I|$.
\end{proof}

\begin{remark}\label{rem Romero}
It will frequently happen that $G$ and $H$ are non-isomorphic finite groups and that $(K,\kappa)\in\calR_G$ and $(L,\lambda)\in\calR_H$ are reduced pairs such that $(G,K,\kappa)\sim(H,L,\lambda)$. All one needs is a pair $(U,\phi)\in\calM_G$ with $l(U,\phi)=(G,K,\kappa)$ and $r(U,\phi)= (H,L,\lambda)$. If $W$ is an irreducible $k\Gamma_{(H,L,\lambda)}$-module and $V:=k[\lGamma{(G,K,\kappa)}_{(H,L,\lambda)}]\otimes_{k\Gamma_{(H,L,\lambda)}} W$ is the corresponding irreducible $k\Gamma_{(G,K,\kappa)}$-module, then $S:=S_{G,\Vtilde}\cong S_{H,\Wtilde}$ and the minimal groups $G$ and $H$ for $S$ are non-isomorphic, answering a question of Serge Bouc (cf.~\cite[Conjecture~2.16]{Romero2012a}) to the negative.
See for instance \cite[Example~12]{Romero2012a}, giving an example of $(U,\phi)\in\calM_{G\times H}$ where $G$ is the quaternion group of order $8$ and $H$ the dihedral group of order $8$. For this example to work, $A$ needs to be an abelian group which contains an element of order $4$. A pair $(U,\phi)$ with $\phi$ of order $4$ is explicitly constructed. In this example, $(K,\kappa)=(Z(G),\kappa)$ and $(L,\lambda)=(Z(H),\lambda)$, where $\kappa$ and $\lambda$ are injective homomorphisms to $A$. Note that $(K,\kappa)$ and $(L,\lambda)$ are reduced by Proposition~\ref{conditions for reduced}(b).
\end{remark}

\section{Reduced pairs revisited}\label{sec completely reduced pairs}

In the previous section, we parametrized the simple fibered biset functors by equivalence classes of quadruples $(G,K,\kappa,[V])$, where $(K,\kappa)$ is reduced in $\calM_G$. The indirect definition of being reduced makes it very difficult to determine which pairs $(K,\kappa)$ are reduced in $\calM_G$. The goal of this section is to establish a more explicit characterization (see Corollary~\ref{reduced pairs for special A}) of being reduced under additional assumptions on $A$. These assumptions are satisfied in all the cases of interest to us; for instance when $A$ is the multiplicative group of an algebraically closed field $F$.

\begin{hypothesis}\mylabel{hypothesis}
The group $A$ has the following property:
There exists a (unique) set $\pi$ of primes such that for every
$n\in\NN$, the $n$-torsion part $\{a\in A\mid a^n=1\}$ of $A$ is
cyclic of order $n_\pi$. Here, $n_\pi$ denotes the $\pi$-part of $n$.
\end{hypothesis}

\begin{remark}\mylabel{hypothesis remark}
(a) By $\pi'$ we will denote the subset of primes which is
complementary to $\pi$. For an abelian group $B$ we
denote the $\pi$-torsion subgroup of $B$ by $B_\pi$. 

\smallskip
(b) Note that Hypothesis~\ref{hypothesis} implies that $\torA$ is
divisible, i.e., for every $a\in\torA$ and every $n\in\NN$ there
exists $b\in\torA$ such that $b^n=a$.

\smallskip
(c) The important part of $A$ is its torsion subgroup $\torA$. Nothing in what follows changes if
we replace $A$ with $\torA$, see Proposition~\ref{prop A and torA}(c). But we want to keep the freedom
to choose $A$ as the multiplicative group of a field or of an integral
domain. Then Hypothesis~\ref{hypothesis} means that, for any given
prime $p$, if the field has a primitive $p$-th root of unity then it
must have a primitive $p^n$-th root of unity for every $n$. It is not
difficult to see that Hypothesis~\ref{hypothesis} is equivalent to
having an isomorphism $\torA\cong(\CC^\times)_\pi\cong(\QQ/\ZZ)_\pi$,
but we want to keep the freedom to choose the ring or field more
naturally depending on the situation.
\end{remark}

The following propositions list some consequences of
Hypothesis~\ref{hypothesis} on $A$.

\begin{proposition}\mylabel{prop A and torA}
Assume that $\torA$ is divisible and that, for every $n\in\NN$, the $n$-torsion group $A_n$ is finite. Furthermore, let $G$ be a finite group acting trivially on $A$.

\smallskip
{\rm (a)} The group $B^2(G,A)$ of $2$-coboundaries has a complement in the group $Z^2(G,A)$ of $2$-cocycles.

\smallskip
{\rm (b)} The group $H^2(G,A)$ is finite and has exponent dividing $|G|$.

\smallskip
{\rm (c)} The canonical map $H^2(G,\torA)\to H^2(G,A)$ is an isomorphism and every $2$-cohomology class in $H^2(G,A)$ can be represented by a $2$-cocycle with values in $A_{|G|}$.
\end{proposition}

\begin{proof}
Parts (a) and (b) are proved in the same way as Theorem~11.15 in \cite{Isaacs1976a}, using Lemma~\cite[Lemma~11.14]{Isaacs1976a}. The proof of Theorem~11.15 in \cite{Isaacs1976a} also shows that $C^2(G,A)_{|G|}\cdot B^2(G,A) = Z^2(G,A)$. This proves the surjectivity of the canonical map $H^2(G,\torA)\to H^2(G,A)$ and the existence part of the statement in Part~(c). The injectivity of the above map follows immediately from the long cohomology sequence, see also Remark~\ref{rem linkage criterion}(b).
\end{proof}

Recall the notation $B^*:=\Hom(B,A)$ for any abelian group $B$.

\begin{proposition}\mylabel{hypothesis consequence}
Assume Hypothesis~\ref{hypothesis}, let $S$ be a finite group and let
$B$ be a finite abelian group.

\smallskip
{\rm (a)} One has $|B^*|=|B|_\pi$.

\smallskip
{\rm (b)} For every subgroup $C$ of $B$, the sequence of natural maps
\begin{equation*}
  1\to (B/C)^* \to B^* \to C^* \to 1
\end{equation*}
is exact.

\smallskip
{\rm (c)} One has $\bigcap_{\mu\in B^*} \ker(\mu)= B_{\pi'}$.

\smallskip
{\rm (d)} The canonical map
\begin{equation*}
  \epsilon\colon B_\pi\to (B^*)^*\,,\quad b\mapsto (\mu\mapsto
  \mu(b))\,,
\end{equation*}
is an isomorphism.

\smallskip
{\rm (e)} Let $T\le S$ be an abelian subgroup with $|T|=|T|_\pi$. The
restriction map $S^*\to T^*$ is surjective if and only if $S'\cap
T=\{1\}$.

\smallskip
{\rm (f)} The cohomology group $H^2(S,\torA)$ is finite and the group
$B^2(S,\torA)$ of coboundaries has a complement in the group
$Z^2(S,\torA)$ of cocycles. Here, we assume that $S$ acts trivially on
$\torA$.
\end{proposition}

\begin{proof}
(a) This follows immediately from Hypothesis~\ref{hypothesis}
and the structure theorem for finite abelian groups.

\smallskip
(b) Since $\Hom(-,A)$ is left exact, the sequence is exact everywhere,
except possibly at $C^*$. By Part~(a), this implies that the
restriction map $B^*\to C^*$ has image of order
\begin{equation*}
  |B^*|/|(B/C)^*| = |B|_\pi / |B/C|_\pi = |C|_\pi = |C^*|\,.
\end{equation*}
Thus, the map $B^*\to C^*$ is surjective.

\smallskip
(c) Since $A$ has trivial $\pi'$-torsion, it is clear that $B_{\pi'}\subseteq\ker(\mu)$
for every homomorphism $\mu\in B^*$.
Conversely, assume that $b\in B$ has order $n$ with $n_\pi\neq 1$. Let
$a\in A$ be an element of order $n_\pi$ (which exists by
Hypothesis~\ref{hypothesis}). Then $\nu\colon \langle
b\rangle\to A$, $b\mapsto a$, defines a group homomorphism. By
Part~(b), $\nu$ can be extended to a homomorphism $\mu\in B^*$. Thus,
we have $\mu(b)=a\neq 1$ for some $\mu\in B^*$.

\smallskip
(d) By Part~(c), the map $\epsilon$ is injective. By Part~(a), we have
$|B_\pi|=|B|_\pi=|B^*|=|B^*|_\pi = |(B^*)^*|$. Thus, $\epsilon$ is also
surjective.

\smallskip
(e) Assume first that $S'\cap T$ contains an element $t\neq 1$. Since $T$ is a $\pi$-group,
there exists a homomorphism $\nu_1\colon\langle t\rangle \to A$ with
$\nu_1(t)\neq 1$. By Part~(b), this homomorphism can be extended to a
homomorphism $\nu_2\in T^*$. Then $\nu_2(t)\neq 1$ and, since $t\in S'$, the map $\nu_2$ is not
in the image of $S^*\to T^*$.

Now assume that $S'\cap T=\{1\}$. Then $S'T/S'$ is canonically
isomorphic to $T$ and every homomorphism $S'T/S'\to A$ can be extended
to a homomorphism $S/S'\to A$, by Part~(b). Thus, $S^*\to T^*$ is
surjective.

\smallskip
(f) This follows immediately from Proposition~\ref{prop A and torA}.
\end{proof}

\begin{notation}\mylabel{cohomology notation}
Let $S$ be a group and let $B$ be an abelian group. We set
\begin{equation*}
  M(S):=H^2(S,\torA)\,,
\end{equation*}
regarding $\torA$ endowed with the trivial $S$-action. If
$A=\CC^\times$, the unit group of the complex numbers, then $M(S)$ is
the well-known Schur multiplier of $S$. There exists a natural group
homomorphism
\begin{align*}
  \Psi\colon H^2(S,B) & \to \Hom(B^*,M(S))\,,\\
  [\alpha] & \mapsto (\mu\mapsto [\mu\circ\alpha])\,.
\end{align*}
Here, we denote by $[\alpha]$ the cohomology class of a cocycle
$\alpha$. As usual, for an abelian group $C$ we denote by
\begin{equation*}
  \Ext(C,B)
\end{equation*}
the subgroup of $H^2(C,B)$ consisting of cohomology classes of
cocycles $\alpha\colon C\times C\to B$ satisfying
$\alpha(c_1,c_2)=\alpha(c_2,c_1)$ for all $c_1,c_2\in C$. These
cohomology classes correspond to abelian extensions of $C$ by $B$.
Here we assume again that $C$ acts trivially on $B$. Note that one has
a natural group homomorphism
\begin{align*}
  \iota\colon\Ext(S/S',B) & \to H^2(S,B)\,,\\
  [\alpha] & \mapsto [\alpha\circ(\nu\times\nu)]\,,
\end{align*}
where $\nu\colon S\to S/S'$ denotes the natural epimorphism. It is
shown in \cite[Lemma~2.1.17]{Karpilovsky1987a} that $\iota$ is
injective and that the image of $\iota$ is equal to the subgroup
\begin{equation*}
  H_0^2(S,B)\le H^2(S,B)
\end{equation*}
which is defined as the set of cohomology classes whose corresponding group
extensions $1\to B\to X\to S\to 1$ satisfy $B\le Z(X)$ and $B\cap X'=\{1\}$.
\end{notation}

\begin{remark}\mylabel{cohomology diagram}
Assume that $G$ is a group and that $K$ is a normal abelian subgroup of
$G$. Moreover set $\Ktilde:=K\cap G'$. Then the maps defined in
\ref{cohomology notation} give rise to a commutative diagram

\begin{diagram}
  & & & & \movevertex(30,0){\Hom((K/\Ktilde)^*,M(G/K))} &&
  & & & & \movevertex(30,0){\naR{\epsilon_4}} &&
  \movevertexleft{\llap{$1\to$}\Ext(G/KG',K)} & \Ear{\iota_1} & H^2(G/K,K) & \Ear{\Psi_1}
  & \movevertexright{\Hom(K^*,M(G/K))\rlap{$\to 1$}} &&
  \movevertexright{\naR{\epsilon_1}} & & \naR{\epsilon_2} & &
  \movevertex(30,0){\naR{\epsilon_3}} &&
  \movevertexleft{\llap{$1\to$}\Ext(G/KG',\Ktilde)} & \Ear{\iota_2} & H^2(G/K,\Ktilde)
  & \Ear{\Psi_2} & \movevertexright{\Hom(\Ktilde^*,M(G/K))\rlap{$\to 1$}} &&
  & & & & \movevertex(30,15){\nar[20]} &&
  & & & & \movevertex(30.5,35){1} &&
\end{diagram}
Here, $\epsilon_1$, $\epsilon_2$ and $\epsilon_3$ are induced by the
inclusion $\Ktilde\to K$ and $\epsilon_4$ is induced by the natural
epimorphism $K\mapsto K/\Ktilde$. Also, in the domains of the maps
$\iota_1$ and $\iota_2$, we identify $(G/K)/(G/K)'$ with $G/KG'$ in the
obvious way.
\end{remark}

\begin{proposition}\mylabel{cohomology proposition}
Assume that $A$ satisfies Hypothesis~\ref{hypothesis}. Let $G$ be a
finite group, let $K$ be a subgroup of $Z(G)$ with $|K|_\pi=|K|$, and set $\Ktilde:=K\cap G'$.

\smallskip
{\rm (a)} The two rows and the right column in the diagram in
Remark~\ref{cohomology diagram} are exact.

\smallskip
{\rm (b)} Let $\alpha\in Z^2(G/K,K)$ be a cocycle describing the central
extension $1 \to K \to G\to G/K \to 1$. Then, in the diagram in
Remark~\ref{cohomology diagram}, one has
$(\epsilon_4\circ\Psi_1)([\alpha])=1$.
\end{proposition}

\begin{proof}
(a) The column is exact since, $\Hom(-,M(G/K))$ is left exact and
since $1\to (K/\Ktilde)^*\to K^* \to \Ktilde^* \to 1$
is exact by Proposition~\ref{hypothesis consequence}(b).

We only show exactness of the first row. The exactness proof for the second
one is analogous. For the exactness of the first row we refer to the proof of 
\cite[Theorem~2.1.19]{Karpilovsky1987a} in the case
$A=\CC^\times$. The same proof works for $A$ satisfying Hypothesis~\ref{hypothesis}. 
The injectivity of $\iota_1$ and that $\im(\iota_1)= H^2_0(G/K,K)$ is proved in
\cite[Lemma~2.1.17]{Karpilovsky1987a}. The exactness at $H^2(G/K,K)$
can be proved with the same arguments as in
\cite[Theorem~2.1.19]{Karpilovsky1987a}, using
Proposition~\ref{hypothesis consequence}(e) and that $K$ is a
$\pi$-group. The surjectivity of $\Psi_1$ is proved with the same
arguments as in the proof of \cite[Theorem~2.1.19]{Karpilovsky1987a},
using the results in Proposition~\ref{hypothesis consequence}(d) and (f).

\smallskip
(b) There exists a function $\rho\colon G/K\to G$ such that
\begin{equation*}
  \rho(x)\rho(y)=\alpha(x,y)\rho(xy)\,,
\end{equation*}
for all $x,y\in G/K$. Every element in $(K/\Ktilde)^*$ is of the form
$\mubar$ for a homomorphism $\mu\colon K\to A$ with
$\mu|_{\Ktilde}=1$. By the definition of $\Psi_1$ and $\epsilon_4$, it suffices 
to show that $[\mu\circ\alpha]=1$ in
$M(G/K)$, for each such $\mu$. From Proposition~\ref{hypothesis consequence}(e) we know
that $\mubar$ can be extended to a homomorphism $G/\Ktilde\to A$. Then, also $\mu$ can be extended to a homomorphism $\mutilde\colon G\to A$.
Applying $\mutilde$ to the above equation and setting
$\nu(x):=\mutilde(\rho(x))$, for $x\in G/K$, we obtain
\begin{equation*}
  \nu(x)\nu(y)=\mu(\alpha(x,y))\nu(xy)
\end{equation*}
for all $x,y\in G/K$, showing that $[\mu\circ\alpha]=1$.
\end{proof}

The motivation for the following definition is given in
Proposition~\ref{proposition on alphan}. It will be heavily used in
Lemma~\ref{squeezing lemma}.

\begin{definition}\mylabel{definition of alphan}
Let $S$ be a group and let $B$ be an abelian group (endowed with the
trivial $S$-action). For $n\in\NN$, let $F(S^{2n},B)$ denote the
abelian group of functions from $S^{2n}$ to $B$. For each $\alpha\in Z^2(S,B)$, we define functions $\alpha_n\in F(S^{2n},B)$, $n\in\NN$, recursively by setting
\begin{equation*}
  \alpha_1(s_1,s_2):= \alpha(s_1,s_2) \alpha(s_2,s_1)^{-1}
  \alpha(s_2s_1,s_1^{-1}s_2^{-1})^{-1} \alpha(s_1s_2,s_1^{-1}s_2^{-1})
  \alpha(1,1)^{-1}\,,
\end{equation*}
for $s_1,s_2\in S$, and by
\begin{align*}
  \alpha_n(s_1,\ldots,s_{2n}):= & \alpha_{n-1}(s_1,\ldots,s_{2n-2})
  \alpha_1(s_{2n-1},s_{2n})\cdot \\
  & \cdot
  \alpha([s_1,s_2]\cdots[s_{2n-3},s_{2n-2}],[s_{2n-1},s_{2n}])\,,
\end{align*}
for $n\ge 2$.
\end{definition}

\begin{proposition}\mylabel{proposition on alphan}
Let $S$ be a group, let $B$ be an abelian group (with trivial
$S$-action), let $\alpha\in Z^2(S,B)$ and let $n\in\NN$. The function
$\alpha_n\colon S^{2n}\to B$ has the following properties:

\smallskip
{\rm (a)} If
\begin{equation}\label{ses}
 1 \ar B \Ar{\iota} T \Ar{\pi} S \ar 1
\end{equation}
is a short exact sequence of groups with $\iota(B)\le Z(T)$ and if
$\sigma\colon S\to T$ is a section of $\pi$ (i.e.,
$\pi\circ\sigma=\id_S$) such that
\begin{equation}\label{section cocycle relation}
  \sigma(s_1)\sigma(s_2) = \iota(\alpha(s_1,s_2)) \sigma(s_1s_2)
\end{equation}
for all $s_1,s_2\in S$, then one has
\begin{align}\label{property of alphan}
  [\sigma(s_1),\sigma(s_2)] & \cdots [\sigma(s_{2n-1}),\sigma(s_{2n})]
  \\ \notag
  & = \iota(\alpha_n(s_1,s_2,\ldots,s_{2n})) \sigma([s_1,s_2]\cdots
  [s_{2n-1},s_{2n}])\,,
\end{align}
for all $s_1,\ldots,s_{2n}\in S$.

\smallskip
{\rm (b)} One has $(\alpha\beta)_n=\alpha_n\beta_n$ for all
$\alpha,\beta\in Z^2(S,B)$.

\smallskip
{\rm (c)} If $f\colon \Stilde\to S$ is a group homomorphism then
$(\alpha\circ(f\times f))_n = \alpha_n\circ(f\times\cdots\times f)$.

\smallskip
{\rm (d)} If also $C$ is an abelian group with trivial $S$-action and $f\colon B\to C$ is a group homomorphism 
then $(f\circ \alpha)_n = f\circ\alpha_n$.

\smallskip
{\rm (e)} If $S$ is abelian and $\alpha$ is symmetric (i.e.,
$\alpha(s_1,s_2)=\alpha(s_2,s_1)$ for all $s_1,s_2\in S$) then
$\alpha_n$ is the constant function with value $\alpha(1,1)^{-1}$.

\smallskip
{\rm (f)} If there exists a function $\mu\colon S \to B$ such that
$\alpha(s,t)=\mu(s)\mu(t)\mu(st)^{-1}$ for all $s,t\in S$ then
\begin{equation*}
  \alpha_n(s_1,\ldots,s_{2n}) =
  \mu([s_1,s_2]\cdots[s_{2n-1},s_{2n}])^{-1}\,,
\end{equation*}
for all $s_1,\ldots,s_{2n}\in S$.
\end{proposition}

\begin{proof}
(a) Assume that we have a short exact sequence as in (\ref{ses}) with
$\iota(B)\le Z(T)$ and a section $\sigma$ of $\pi$ such that
Equation~(\ref{section cocycle relation}) holds for all $s_1,s_2\in
S$. Then Equation~(\ref{section cocycle relation})
implies
\begin{equation}\label{section of 1}
  \sigma(1)=\iota(\alpha(1,1))
\end{equation}
and
\begin{equation}\label{section inverse}
  \sigma(s)^{-1} = \iota(\alpha(s,s^{-1})^{-1} \alpha(1,1)^{-1})
  \sigma(s^{-1})
\end{equation}
for all $s\in S$. We prove Equation~(\ref{property of alphan}) by induction on $n$. For $n=1$, Equation~(\ref{section cocycle relation}) yields
\begin{align*}
  [\sigma(s_1),\sigma(s_2)] & = \sigma(s_1)\sigma(s_2) \bigl(\sigma(s_2)\sigma(s_1)\bigr)^{-1}\\ 
  & = \sigma(s_1s_2)\iota\alpha(s_1,s_2)\sigma(s_2s_1)^{-1}\iota\alpha(s_2,s_1)^{-1}\,.
\end{align*}
Now applying Equation~(\ref{section inverse}), for $s=s_2s_1$, to the third factor in the last expression, we obtain the desired formula. For the induction step from $n-1$ to $n$ we see by first applying the induction hypothesis and the case $n=1$, then Equation~(\ref{section cocycle relation}), that
\begin{align*}
  & [\sigma(s_1),\sigma(s_2)]\cdots[\sigma(s_{2n-1}),\sigma(s_{2n})] \\
  =\ & \iota\alpha_{n-1}(s_1,\ldots,s_{2n-2}) \sigma\bigl([s_1,s_2]\cdots[s_{2n-3},s_{2n-2}]\bigr) 
             \iota\alpha_1(s_{2n-1},s_{2n})\sigma([s_{2n-1},s_{2n}]) \\
  =\ & \sigma\bigl([s_1,s_2]\cdots[s_{2n-1},s_{2n}]\bigr) 
            \iota\alpha\bigl([s_1,s_2]\cdots[s_{2n-3},s_{2n-2}],[s_{2n-1},s_{2n}]\bigr)\cdot \\
            & \qquad\qquad \cdot \iota\alpha_{n-1}(s_1,\ldots,s_{2n-2})\iota\alpha_1(s_{2n-1},s_{2n})\,.
\end{align*}

\smallskip
Parts (b), (c), (d), and (f) follow immediately from the definition of $\alpha_n$ using induction on $n$. And Part~(e) is an easy consequence of Equations~(\ref{property of alphan}) and (\ref{section of 1}).
\end{proof}

\begin{lemma}\mylabel{squeezing lemma}
Assume Hypothesis~\ref{hypothesis} on $A$. Let $G$ be a finite group,
let $K\le Z(G)$ be a $\pi$-group, let $\kappa\in K^*$ be faithful and
set $\Ktilde:=K\cap G'$. Then
there exists a short exact sequence of groups
\begin{equation}\label{sestilde}
  1 \ar \Ktilde \Ar{\iotatilde} \Gtilde \Ar{\pitilde} G/K \ar 1
\end{equation}
with $\iotatilde(\Ktilde)\le Z(\Gtilde)\cap\Gtilde'$ and an element
$(M,\mu)\in\calM_{G\times\Gtilde}$ such that, after identifying
$\Ktilde$ and $\iotatilde(\Ktilde)$ via $\iotatilde$, one has
\begin{equation*}
  p_1(M)=G\,,\ k_1(M)= K\,,\ \mu_1=\kappa\,, \quad \text{and}\quad
  p_2(M)=\Gtilde\,, \ k_2(M)=\Ktilde\,, \ \mu_2=\kappa|_{\Ktilde}\,.
\end{equation*}
In particular, with $\kappatilde:=\kappa|_{\Ktilde}$, one has $(\Ktilde,\kappatilde)\in\calR_{\Gtilde}$, by Proposition~\ref{conditions for reduced}(b).
\end{lemma}

\begin{proof}
Consider the short exact sequence
\begin{equation*}
  1\ar K\Ar{\iota} G\Ar{\pi} G/K \to 1
\end{equation*}
with $\iota$ the inclusion and $\pi$ the canonical surjection. Let
$\sigma\colon G/K\to G$ be a section of $\pi$ and let $\alpha\in
Z^2(G/K,K)$ denote the cocycle satisfying Equation~(\ref{section cocycle
relation}) for all $s_1,s_2\in G/K$. We consider the commutative
diagram from Remark~\ref{cohomology diagram}. By
Proposition~\ref{cohomology proposition}(b) we have
$\epsilon_4(\Psi_1(\alphabar))=1$. As the right hand
column and bottom row of the diagram are exact, there exists $\beta\in
Z^2(G/K,\Ktilde)$ such that
$\epsilon_3(\Psi_2([\beta]))=\Psi_1([\alpha])$. It follows that
$\epsilon_2([\beta])^{-1}\cdot[\alpha]\in\ker(\Psi_1)$ and, by the
exactness of the top row, there exists a symmetric $\gamma\in
Z^2(G/KG',K)$ such that
$\iota_1([\gamma])\cdot\epsilon_2([\beta])=[\alpha]$. After
changing $\beta$ and $\gamma$ by a coboundary, we may assume that
$\beta(1,1)=1$ and that $\gamma(1,1)=1$. After changing the section
$\sigma$, we may also assume that
\begin{equation*}
  \alpha(s_1,s_2)=\gamma(\sbar_1,\sbar_2)\beta(s_1,s_2)
\end{equation*}
for all $s_1,s_2\in G/K$, where $\sbar$ denotes the image of an
element $s\in G/K$ under the natural epimorphism $G/K\to G/KG'$. It
follows that $\alpha(1,1)=1$ and that $\sigma(1)=1$. By
Proposition~\ref{proposition on alphan}(e), we
have
\begin{equation}\label{vanishing of gamma}
  \gamma_n(\sbar_1,\ldots,\sbar_{2n})=1
\end{equation}
and by parts (b) and (c) of the same proposition, we
obtain
\begin{equation}\label{alpha = beta}
  \alpha_n(s_1,\ldots,s_{2n})=\beta_n(s_1,\ldots,s_{2n})
\end{equation}
for all $s_1,\ldots,s_{2n}\in G/K$ and all $n\in\NN$.

\smallskip
Next, we define the group $\Gtilde$, using the cocycle $\beta$, as the
set $\Ktilde\times G/K$ with multiplication
\begin{equation*}
  (\ktilde_1,s_1)(\ktilde_2,s_2):=(\beta(s_1,s_2)\ktilde_1\ktilde_2,s_1s_2)
\end{equation*}
for $\ktilde_1,\ktilde_2\in\Ktilde$ and $s_1,s_2\in G/K$. We obtain a
short exact sequence as in (\ref{sestilde}) with
$\iotatilde(\ktilde):=(\ktilde,1)$ and $\pitilde(\ktilde,s):=s$, for
$\ktilde\in\Ktilde$ and $s\in G/K$. With the section
\begin{equation*}
  \sigmatilde\colon G/K\to \Gtilde\,, \quad s\mapsto (1,s)\,,
\end{equation*}
of $\pitilde$ we obtain
\begin{equation}\label{section cocycle relation tilde}
  \sigmatilde(s_1)\sigmatilde(s_2)=
  \iotatilde(\beta(s_1,s_2))\sigmatilde(s_1s_2)
\end{equation}
for all $s_1,s_2\in G/K$. Note that since $\beta(1,1)=1$, also
$\beta(s,1)=\beta(1,s)=1$ for all $s\in S$. This implies that $(1,1)$
is the identity of $\Gtilde$ and that $\iotatilde(\Ktilde)\le
Z(\Gtilde)$.

\smallskip
Next, we define $M\le G\times\Gtilde$ by
\begin{equation*}
  M:=\{(g,(\ktilde,s))\in G\times\Gtilde\mid \pi(g)=s\}\,.
\end{equation*}
We will identify $\Ktilde$ with $\iotatilde(\Ktilde)$ and view
$\Ktilde$ as a subgroup of $\Gtilde$. It is now clear from the
definition of $M$ that
\begin{equation*}
  p_1(M)= G\,,\quad p_2(M)=\Gtilde\,,\quad k_1(M)=K\,, \quad
  k_2(M)=\Ktilde\,.
\end{equation*}

\smallskip
Next, we want to define $\mu\in M^*$ such that $\mu$ extends
$\kappa\times 1\in (K\times1)^*$. By Proposition~\ref{hypothesis
consequence}(e), it suffices to show that $(K\times1)\cap M'=1$. So
let $k\in K$ and assume that $(k,(1,1))\in M'$. Then we can write
$(k,(1,1))$ as a product of $n$ commutators
$[(g_i,(\ktilde_i,s_i)),(g'_i,(\ktilde'_i,s'_i))]$ of elements in $M$,
$i=1,\ldots,n$. Note that, since the above elements
belong to $M$, we have $\pi(g_i)=s_i$ and $\pi(g'_i)=s'_i$ for
$i=1,\ldots,n$. Note also that, since $K$ is central in $G$ and
$\Ktilde$ is central in $\Gtilde$, we may replace $g_i$ by
$\sigma(\pi(g_i))=\sigma(s_i)$, $g'_i$ by $\sigma(s_i)$ and we may
replace $\ktilde_i$ and $\ktilde'_i$ by $1$. This way, we see that we
can write
\begin{equation*}
  (k,(1,1)) =
  [(\sigma(s_1),\sigmatilde(s_1)),(\sigma(s'_1),\sigmatilde(s'_1))]\cdots
  [(\sigma(s_n),\sigmatilde(s_n)),(\sigma(s'_n),\sigmatilde(s'_n))]
\end{equation*}
for certain element $s_1,s'_1,\ldots,s_n,s'_n\in G/K$. It follows from
Proposition~\ref{proposition on alphan} and Equation~(\ref{alpha = beta}) 
that
\begin{align}\label{commutator product}
  & [(\sigma(s_1),\sigmatilde(s_1)),(\sigma(s'_1),\sigmatilde(s'_1))]\cdots
  [(\sigma(s_n),\sigmatilde(s_n)),(\sigma(s'_n),\sigmatilde(s'_n))]\\
  \notag
  = &
  \Bigl(\alpha_n(s_1,\ldots,s'_n)\sigma([s_1,s'_1]\cdots[s_n,s'_n]),
  \bigl(\beta_n(s_1,\ldots,s'_n),
  \sigmatilde([s_1,s'_1]\ldots[s_n,s'_n])\bigr)\Bigr)
  \\
  \notag
  = &
  \Bigl(\beta_n(s_1,\ldots,s'_n)\sigma([s_1,s'_1]\cdots[s_n,s'_n]),
  \bigl(\beta_n(s_1,\ldots,s'_n),
  \sigmatilde([s_1,s'_1]\ldots[s_n,s'_n])\bigr)\Bigr)\,.
\end{align}
This implies that $\sigmatilde([s_1,s'_1]\ldots[s_n,s'_n])=1$ and that
$\beta_n(s_1,\ldots,s'_n)=1$. Moreover, the first of the last two equation implies that
$[s_1,s'_1]\cdots[s_n,s'_n]=1$ and that
$\sigma([s_1,s'_s]\cdots[s_n,s'_n])=\sigma(1)=1$. Altogether, we
obtain that $k=1$.

\smallskip
Now we know that there exists $\mu\in M^*$ with $\mu|_{K\times
1}=\kappa\times 1$ and we choose any such $\mu$. Then clearly $\mu_1=\kappa$.

\smallskip
Next, we show that
\begin{equation*}
  \Delta(\Ktilde)=\{(\ktilde,(\ktilde,1))\mid\ktilde\in\Ktilde\} \le
  M'\,.
\end{equation*}
Let $\ktilde\in\Ktilde$. Since $\Ktilde\in G'$, there exist
elements $g_1,g'_1,\ldots,g_n,g'_n\in G$ such that
$\ktilde=[g_1,g'_1]\cdots[g_n,g'_n]$. Since $K$ is central in $G$, we
may assume that $g_i=\sigma(s_i)$ and $g'_i=\sigma(s'_i)$ for
$i=1,\ldots,n$ and elements $s_1,\ldots,s'_n\in S$. Thus, we have
\begin{align*}
  \ktilde & =
  [\sigma(s_1),\sigma(s'_1)]\cdots[\sigma(s_n),\sigma(s'_n)] \\
  & = \beta_n(s_1,\ldots,s'_n)\sigma([s_1,s'_1]\cdots[s_n,s'_n])\,,
\end{align*}
by Proposition~\ref{proposition on alphan} and Equation~(\ref{alpha =
beta}). This implies that $\sigma([s_1,s'_1]\cdots[s_n,s'_n])\in\Ktilde$ and,
consequently, that $[s_1,s'_1]\cdots[s_n,s'_n]=1$. Thus $\sigma([s_1,s'_1]\cdots[s_n,s'_n])=1$ and
$\ktilde=\beta_n(s_1,\ldots,s'_n)$. With Equation~(\ref{commutator product}),
we see now that
\begin{equation*}
  (\ktilde,(\ktilde,1)) =
  [(\sigma(s_1),\sigmatilde(s_1)),(\sigma(s'_1),\sigmatilde(s'_1))]\cdots
  [(\sigma(s_n),\sigmatilde(s_n)),(\sigma(s'_n),\sigmatilde(s'_n))]\,,
\end{equation*}
since $\sigmatilde(1)=1$. Therefore $\Delta(\Ktilde)\le M'$.

\smallskip
Since $\Delta(\Ktilde)\le M'$, we have
\begin{equation*}
  1=\mu(\ktilde,(\ktilde,1))=\mu_1(\ktilde)\mu_2^{-1}(\ktilde)
\end{equation*}
for all $\ktilde\in\Ktilde$. This implies that
$\mu_2=\mu_1|_{\Ktilde}=\kappa|_{\Ktilde}$.
Finally, since $\Delta(\Ktilde)\le M'$, we have
$(\ktilde,(\ktilde,1))\in (G\times\Gtilde)'=G'\times\Gtilde'$ and
therefore $(\ktilde,1)\in\Gtilde'$ for all $\ktilde\in\Ktilde$. Thus
$\Ktilde\le\Gtilde'$.
\end{proof}

\begin{definition}
Let $G$, $K$, $\kappa$, $\Gtilde$, and $(M,\mu)$ be as in
Lemma~\ref{squeezing lemma}. By abuse of notation, we define the
$A$-fibered $(G,\Gtilde)$-biset
\begin{equation*}
  \Ins_{\Gtilde}^G:=\left(\frac{G\times\Gtilde}{M,\mu}\right)\,,
\end{equation*}
the {\em insertion} from $\Gtilde$ to $G$ (which inserts the section
$K/\Ktilde$ into $\Gtilde$), and we define the $A$-fibered
$(\Gtilde,G)$-biset
\begin{equation*}
  \Del^G_{\Gtilde}:=(\Ins_{\Gtilde}^G)^{\mathrm{op}}\,,
\end{equation*}
the {\em deletion} from $G$ to $\Gtilde$ (which deletes the section
$K/\Ktilde$ from $G$).

\smallskip
Note that $(G,K,\kappa)$ does not determine $\Gtilde$ and $(M,\mu)$
uniquely, so that $\Ins_{\Gtilde}^G$ and $\Del^G_{\Gtilde}$ are not well-defined elements. 
But all that matters for our purposes is the existence of $\Gtilde$ and $(M,\mu)$ with
the properties from Lemma~\ref{squeezing lemma}, for
given $(G,K,\kappa)$.
\end{definition}

Next we decompose the biset $X$ from Proposition~\ref{decomposition}
further. Recall the properties of $X$ mentioned in the paragraph
following Proposition~\ref{decomposition}.

Recall from Definition~\ref{def twosided reduced} that, for finite groups $G$ and $H$, we call a pair $(U,\phi)\in\calM^c_{G\times H}$ {\em reduced} if $l_0(U,\phi)\in\calR_G$ and $r_0(U,\phi)\in\calR_H$.

\begin{proposition}\mylabel{decomposition 2}
Let $G$ and $H$ be finite groups, let $(U,\phi)\in\calM_{G\times H}$, and set $X:=\left(\frac{G\times H}{U,\phi}\right)$. 
Assume that $p_1(U)=G$, $p_2(U)=H$, set $K:=k_1(U)$, $L:=k_2(U)$, and assume that $\kappa:=\phi_1\in K^*$ and $\lambda:=\phi_2\in L^*$ are faithful. Then there exists a decomposition
\begin{equation*}
  X\cong \Ins_{\Gtilde}^G\otimes_{A\Gtilde} Y \otimes_{A\Htilde}
  \Del^H_{\Htilde}
\end{equation*}
with the following properties:

\smallskip
{\rm (a)} The group $\Gtilde$ is a central extension of $G/K$ by
$\Ktilde:=K\cap G'$ and the group $\Htilde$ is a central extension of $H/L$ by
$\Ltilde:= L\cap H'$.

\smallskip
{\rm (b)} One has an isomorphism
$Y\cong\left(\frac{\Gtilde\times\Htilde}{\Utilde,\phitilde}\right)$
for a reduced pair $(\Utilde,\phitilde)\in\calM_{\Gtilde\times\Htilde}$
satisfying
\begin{equation*}
  k_1(\Utilde)=\Ktilde\,,\ k_2(\Utilde)=\Ltilde\,,\
  \phitilde_1=\kappa|_{\Ktilde}\,,\ \phitilde_2=\lambda|_{\Ltilde}\,.
\end{equation*}
\end{proposition}

\begin{proof}
We define $Y:=\Del^G_{\Gtilde}\otimes_{AG} X \otimes_{AH}
\Ins_{\Htilde}^G$. All the statements of the proposition follow from Lemma~\ref{squeezing
lemma}, Proposition~\ref{idempotent proposition}(b) and (c) and from
the associativity of the tensor product.
\end{proof}

Assuming Hypothesis~\ref{hypothesis}, we have now a simple criterion for a pair to be reduced, thus simplifying the classification of simple fibered biset functors.

\begin{corollary}\mylabel{reduced pairs for special A}
Assume that $A$ satisfies Hypothesis~\ref{hypothesis}, let $G$ be a finite group, and let $(K,\kappa)\in\calM_G^G$. Then $(K,\kappa)$ is reduced in $G$ if and only if $K$ is a cyclic $\pi$-subgroup contained in $Z(G)\cap G'$ and $\kappa$ is faithful.
\end{corollary}

\begin{proof}
It was already proved in Proposition~\ref{conditions for reduced}(b) that the condition is sufficient for $(K,\kappa)$ to be reduced in $G$. Conversely, assume that $(K,\kappa)$ is reduced in $G$. By Proposition~\ref{conditions for reduced}(a), we know that $\kappa$ is faithful and $K\le Z(G)$. This implies that $K$ is a cyclic $\pi$-group, by the hypothesis on $A$. Assume that $K$ is not contained in $G'$. Applying Proposition~\ref{decomposition 2} to the the pair $(\Delta_K(G), \phi_\kappa)$ shows that $e_{(K,\kappa)}\in I_G$, since $|\Gtilde|<|G|$.
\end{proof}

Combining Propositions~\ref{decomposition} and \ref{decomposition 2},
we obtain the following refined decomposition of an arbitrary transitive fibered biset under Hypothesis~\ref{hypothesis}.

\begin{theorem}\mylabel{decomposition 3}
Assume that $A$ satisfies Hypothesis~\ref{hypothesis}. Let $G$ and $H$ be finite groups, let 
$(U,\phi)\in\calM_{G\times H}$, and let
\begin{equation*}
  \Khat\le\Ktilde\le K\le P\le G\,,\ \Lhat\le L\tilde\le L\le Q\le H\,,\  \kappa:=\phi_1\in K^*\,,\ \text{and}\ \lambda:=\phi_2\in L^*
\end{equation*} 
be defined as in Proposition~\ref{eta and zeta}.

There exists a decomposition
\begin{equation*}
  \left(\frac{G\times H}{U,\phi}\right)\cong
  \Ind_P^G\otimes\Inf_{P/\Khat}^P\otimes\Ins_{\Gtilde}^{P/\Khat}
  \otimes X \otimes
  \Del^{Q/\Lhat}_{\Htilde}\otimes\Def^Q_{Q/\Lhat}\otimes\Res^H_Q\,,
\end{equation*}
where $X\cong\left(\frac{\Gtilde\times
\Htilde}{\Utilde,\phitilde}\right)$ is a transitive $A$-fibered
$(\Gtilde,\Htilde)$-biset such that $(\Utilde,\phitilde)\in\calM^c_{\Gtilde\times\Htilde}$ is reduced, where $\Gtilde$ is a central extension of $P/K$ by $\Ktilde/\Khat$ and $\Htilde$ is a central extension of $Q/L$
by $\Ltilde/\Lhat$, and $\phitilde_1=\kappabar|_{\Ktilde/\Khat}$ and
$\phitilde_2=\lambdabar|_{\Ltilde/\Lhat}$. Here, $\kappabar\in (K/\Khat)^*$ and $\lambdabar\in(L/\Lhat)^*$ are
defined as the unique homomorphisms inflating to $\kappa$ and
$\lambda$, respectively.
\end{theorem}

%\begin{definition}\mylabel{reduced pairs (U,phi)}
%Let $G$ and $H$ be finite groups. 
%
%\smallskip
%(a) We call a pair $(K,\kappa)\in\calR_G(A)$ {\em completely reduced} (in $G$) if $\kappa$ is faithful and $K\le G'$. Note that in this case $K$ is also central in $G$. Recall from Proposition~\ref{conditions for reduced}(b) that every completely reduced pair is also reduced.
%
%\smallskip
%(b) We call a pair $(U,\phi)\in\calM^c_{G\times H}(A)$ {\em completely reduced}, if $l_0(U,\phi)\in\calM_G(A)$ and $r_0(U,\phi)\in\calM_H(A)$ are completely reduced.
%\end{definition}
%
%
%\begin{remark}
%Note that by Proposition~\ref{eta and zeta}(c) and (d) the existence of a reduced pair $(U,\phi)\in\calM_{G\times H}(A)$ implies that $G/K\cong H/L$ and $K\cong L$. In particular $|G|=|H|$. Note that by Example~\ref{special example} the groups $G$ and $H$ need not be isomorphic.
%\end{remark}

%%%%%%%%%%%%%%% SECTION 11 %%%%%%%%%%%%%%%%%%%%%%%%%%%%%%%%%%%%%%

\section{Examples}\label{sec examples}
In this section, we look at several examples of fibered biset functors. We give a criterion for a fibered biset functor to be simple and we apply it to some functors. As before, $A$ is a multiplicatively written abelian group and $k$ denotes a commutative ring.

\smallskip
More precisely, we will identify the simple functors $S_{(\{1\},\{1\},1,k)}$ for some choices of $A$ and $k$ as well-known objects in the representation theory of finite groups. Note that if $G$ is the trivial group then $(K,\kappa)=(\{1\},1)$ is the only reduced pair and $\Gamma_{(\{1\},\{1\},1)}$ is again the trivial group, so that $k$ is naturally a $k\Gamma_{(\{1,\},\{1\},1)}$-module. 

\smallskip
For any functor $F\in\calFkA$ and any class $\calH\subseteq\mathrm{Ob}(\calC)$ of finite groups we define subfunctors $\calI_{F,\calH}$ and $\calK_{F,\calH}$ of $F$ as follows. For a finite group $G$ we set
\begin{gather*}
  \calI_{F,\calH}(G):=\sum_{\substack{H\in\calH \\ x\in B_k^A(G,H)}} \im\bigl(F(x)\colon F(H)\to F(G)\bigr) \subseteq F(G)\,, \\
  \calK_{F,\calH}(G):=\bigcap_{\substack{H\in\calH \\ y\in B_k^A(H,G)}} \ker\bigl(F(y)\colon F(G)\to F(H)\bigr) \subseteq F(G)\,.
\end{gather*}
It is clear that $\calI_{F,\calH}$ and $\calK_{F,\calH}$ are subfunctors of $F$.

\medskip
The following proposition gives a criterion for $F$ to be simple. It holds for more general functor categories with the same proof. See also \cite[Theorem 3.1]{TW} for a similar result for Mackey functors.

\begin{proposition}\label{prop simplicity criterion}
Let $F\in\calFkA$ be an $A$-fibred biset functor over $k$ and let $H$ be a finite group such that $F(H)\neq \{0\}$. Then, $F$ is a simple functor if and only if the following three conditions are satisfied:

\smallskip
{\rm (i)} $F(H)$ is a simple $E_H$-module.

\smallskip
{\rm (ii)} $\calI_{F,\{H\}}=F$.

\smallskip
{\rm (iii)} $\calK_{F,\{H\}}=0$.
\end{proposition}

\begin{proof}
We first assume that $F$ is a simple functor. Then $F(H)$ is a simple $E_H$-module by Proposition~\ref{simple functor proposition}(a). Further note that $\calI_{F,\{H\}}(H)=F(H)$ and $\calK_{F,\{H\}}(H)=\{0\}$ (use $x=y=\id_H$ in the definition of $\calI_{F,\{H\}}(H)$ and $\calK_{F,\{H\}}(H)$). Thus $\calI_{F,\{H\}}$ is a non-zero subfunctor of $F$ and $\calK_{F,\{H\}}$ is a proper subfunctor of $F$. Since $F$ is simple, the properties (ii) and (iii) follow.

\smallskip
Next assume that the conditions (i)--(iii) are satisfied and let $L\subseteq F$ be a subfunctor of $F$. Then, $L(H)$ is an $E_H$-submodule of $F(H)$. Since the latter is simple, by (i), we have $L(H)=\{0\}$ or $L(H)=F(H)$. First assume that $L(H)=\{0\}$. We'll show that this implies that $L=0$. In fact, for every finite group $G$ we then have
\begin{align*}
  \{0\} & = \calK_{F,\{H\}}(G) = \bigcap_{y\in B_k^A(H,G)} \ker\bigl(F(y)\colon F(G)\to F(H)\bigr) \\
  & \supseteq L(G) \cap \bigcap_{y\in B_k^A(H,G)} \ker\bigl(F(y)\colon F(G)\to F(H)\bigr) \\
  & = \bigcap_{\substack{H\in\calH \\ y\in B_k^A(H,G)}} \ker\bigl(F(y)\colon L(G)\to F(H)\bigr) = L(G)\,,
\end{align*}
since $F(y)(L(G))\subseteq L(H)=\{0\}$ for all $y\in B_k^A(H,G)$. Finally, we assume that $L(H)=F(H)$ and will show that this implies $L=F$. In fact, for every finite group $G$ we then have
\begin{align*}
  F(G) & = \calI_{F,\{H\}}(G) = \sum_{x\in B_k^A(G,H)} \bigl(F(x)\colon F(H)\to F(G\bigr)) \\
  & = \sum_{x\in B_k^A(G,H)} \im\bigl(F(x)\colon L(H)\to F(G)\bigr) \subseteq L(G)\,,
\end{align*}
since $F(x)(L(H))\subseteq L(G)$, for all $x\in B_k^A(G,H)$.
\end{proof}

\bigskip
\pagebreak
\begin{center}\sc\large \ref{sec examples}A. The Functor $B_k^A$\end{center}

\bigskip\noindent
Canonically identifying a finite group $G$ with $G\times \{1\}$, we have
\begin{equation*}
B^A_k(G)=  B_k^A(G\times\{1\}) = \Hom_{\calCkA}(\{1\},G)\,.
\end{equation*}
Thus, we can view $B^A_k$ as the Yoneda functor $\Hom_{\calCkA}(\{1\}, -)\colon \calCkA\to\kMod$.
For any $A$-fibered $(G,H)$-biset $X$, the map $B^A_k([X])\colon B^A_k(H)\to B^A_k(G)$ is induced by composition with $[X]$ on the left. The functor $B^A_k$ is called the {\em$A$-fibered Burnside functor} over $k$.
Note also that $B^A_k$ is isomorphic in $\calFkA$ to the functor $L_{\{1\},k}$, where $\{1\}$ is the trivial group and $k$ is the regular $\Ebar_{\{1\}}$-module, identifying $\Ebar_{\{1\}}\cong E_{\{1\}}\cong k$ as $k$-algebras, cf.~\ref{noth LGV} and \ref{noth essential algebra}.

\smallskip
Given a transitive $A$-fibered $(G,\{1\})$-biset $\big( \frac{G\times \{1\}}{U,\phi} \big)$, we simply write $[U,\phi]_G$ for its image in the $B^A_k(G)$. With this notation, the $k$-module $B^A_k(G)$ is freely generated by the elements $[U,\phi]_G$ as $(U,\phi)$ runs through a set of representatives of the $G$-orbits of $\mathcal M_G$. 

\begin{theorem}\label{thm example BA}
Assume that $k$ is a field. The functor $B^A_k$ in $\calFkA$ is a projective and indecomposable object. It is a projective cover of the simple functor $S_{(\{1\},\{1\},1,k)}$.
\end{theorem}

\begin{proof}
The first statement is an easy consequence of the Yoneda-Lemma. The second statement follows from the first, the observation that the functors $L_{\{1\},k}$ and $B^A_k$ are isomorphic, and the fact that $S_{\{1\},k}=S_{(\{1\},\{1\},1,k)}$ is the only simple factor of the functor $L_{\{1\},k}$.
\end{proof}

In the following subsections we identify the simple functor $S_{(\{1\},\{1\},1,k)}$ for various choices of $A$ and $k$.

\bigskip
\begin{center}\sc\large \ref{sec examples}B. The Functor $kR_{\CC}$\end{center}

\bigskip\noindent
Let $A = \mathbb C^\times$. For a finite group $G$, we denote by $R_{\mathbb C}(G)$ the character ring of $\mathbb C[G]$-modules. Recall that there is a linearization map
\begin{equation*}
  \lin_G\colon B^{\CC^\times}(G)\to R_{\CC}(G), \quad [H,\phi]_G\mapsto \ind_H^G(\phi)\,,
\end{equation*}
and recall that it is surjective, by Brauer's induction theorem.
It follows from \cite[Theorem 1.1]{Bouc2010c} and the explicit formula in Corollary \ref{Mackey formula for fibered bisets 2} that 
\begin{equation*}
  \lin_{G\times H}([X])\otimes_{\CC H}\lin_{H\times K}([Y]) = \lin_{G\times K}([X]\cdotH[Y])
\end{equation*}
in $R_{\CC}(G\times K)$. This implies first that $G\mapsto R_{\CC}(G)$ gives rise to a $\CC^\times$-fibered biset functor $R_{\CC^\times}$ by mapping the standard basis element $\left[\frac{G\times H}{U,\phi}\right]$ of $B^{\CC^\times}(G,H)$ to the map $R_{\CC}(H)\to R_{\CC}(G)$, $[M]\mapsto [\Ind_{U}^{G\times H}(\CC_\phi)\otimes_{\CC H} M]$; and secondly that the maps $\lin_G$ form a morphism of fibered biset functors. Clearly, one can extend scalars from $\ZZ$ to any commutative ring $k$ and $\lin\colon kB^{\CC^\times}\to k R_{\CC}$ becomes a morphism in $\calF_k^{\CC^\times}$.

\begin{theorem}
Let $k$ be an arbitrary field. The $\CC^\times$-fibered biset functor $kR_{\CC}$ over $k$ is isomorphic to the simple functor $S_{(\{1\},\{1\},1,k)}$. Moreover, $\lin\colon kB^{\CC^\times}\to k R_{\CC}$ is a projective cover in the category $\calF_k^{\CC^\times}$.
\end{theorem}

\begin{proof}
It suffices to show that $kR_{\CC}\in\calF_k^{\CC^\times}$ is a simple functor. Since $\lin\colon kB^{\CC^\times}\to kR_{\CC}$ is surjective, Theorem~\ref{thm example BA} then implies the rest. To see that $kR_{\CC}$ is simple we apply the criterion from Proposition~\ref{prop simplicity criterion} with $H$ chosen as the trivial group $\{1\}$. We will show that the three conditions (i)-(iii) hold.

\smallskip
Clearly, $k R_{\CC}(\{1\})\cong k$ is a simple module for $E_{\{1\}}\cong k$ and the first condition holds.
To see that the second condition holds, we need to show that for every irreducible character $\chi$ of $G$ there exists $x\in B_k^{\CC^\times}(G,\{1\}) = B_k^{\CC^\times}(G)$ such that $(kR_{\CC}(x))(1)=\chi$, where $1\in k= kR_{\CC}(\{1\})$. But, with the above identification, $(kR_{\CC}(x))(1)=\lin_G(x)$. And since $\lin_G$ is surjective, the second condition holds. Finally we show that the third condition holds, namely that $\calK_{kR_{\CC},\{1\}}(G)=0$ for any finite group $G$. In other words, we need to show that for any non-zero $k$-linear combination $f$ of simple $\CC G$-modules, there exists 
$x\in B_k^{\CC^\times}(\{1\},G)=B^{\CC^\times}_k(G)$ such that $xf=(R_{\CC}(x))(f)\neq 0\in kR_{\CC}(\{1\})= k$. However, for $x\in B_k^{\CC^\times}(G)$ and any $\CC G$-module $V$, with the above identifications, we have $x[V]=\lin_G(x)\cdotG [V]\in R_{\CC}(\{1\})$, where $\cdotG\colon R_{\CC}(G)\times R_{\CC}(G)\to k$ is induced by $(W,V)\mapsto \dim_{\CC}(\Wtilde\otimes_{\CC G}V)$, for left $\CC G$-modules $V$ and $W$, where $\Wtilde$ is the right $\CC G$-module with $W$ as underlying space and $G$-action $wg:=g^{-1}w$ for $w\in W$ and $g\in G$. Thus, choosing $x\in B_k^{\CC^\times}(G)$ such that $\lin_G(x)$ is the dual of an irreducible constituent occurring in $f$ with non-zero coefficient $\alpha\in k$, we obtain $xf=\alpha\neq 0$.
\end{proof}

\bigskip
\begin{center}\sc\large \ref{sec examples}C. The Functor $kR_{\QQ}$\end{center}

\bigskip\noindent
Let $A=\mathbb Q^\times$ and let $R_{\mathbb Q}(G)$ denote the character ring of finitely generated $\QQ G$-modules, where $G$ is a finite group. Since $\mathbb Q^\times\subset \mathbb C^\times$, the $\mathbb C^\times$-fibered biset functor $R_{\CC}$ 
restricts to a $\mathbb Q^\times$-fibered biset functor. Considering $R_\QQ(G)$ as subgroup of $R_\CC(G)$, we see that $R_\QQ$ becomes a subfunctor of $R_\CC$ as $\QQ^\times$-biset functor and thus is a $\QQ^\times$-fibered biset functo in its on right.

On the other hand, any $\mathbb Q^\times$-fibered biset functor restricts to an ordinary biset functor. Further, recall that, by \cite[Proposition 4.4.8]{BoucBook}, for a field $k$ of characteristic zero, the functor $kR_{\mathbb Q}=k\otimes R_{\QQ}$ is simple
as a biset functor. Therefore, the $\mathbb Q^\times$-fibered biset functor $kR_{\mathbb Q}$ is also simple. Since $kR_{\QQ}(\{1\})\neq \{0\}$, the trivial group $\{1\}$ is a minimal group for this functor and the classification 
of simple fibered biset functors from Section~\ref{sec simple functors} implies that it is parameterized by the quadruple $(\{1\},\{1\},1,k)$. The map $\lin_G\colon B_k^{\CC^\times}(G)\to kR_{\CC^\times}(G)$ from \ref{sec examples}B restricts to the map 
\begin{equation*}
  \lin_G\colon B_k^{\QQ^\times}(G)\to R_{\QQ}(G)\,, \quad [H,\phi]_G\mapsto \ind_H^G(\phi)\,,
\end{equation*}
and therefore defines a morphism of $\QQ^\times$-fibered biset functors. It must be surjective, since $kR_{\QQ^\times}$ is a simple functor in $\calF_k^{\QQ^\times}$. Thus, using Theorem~\ref{thm example BA}, we now have the following result.

\begin{theorem}\label{thm: the simple (1,1)2}
Let $k$ be a field of characteristic zero. Then there is an isomorphism
\begin{equation*}
S_{(\{1\},\{1\},1,k)} \cong kR_{\mathbb Q}
\end{equation*}
of $\mathbb Q^\times$-fibered biset functors. Moreover, $\lin\colon kB^{\QQ^\times}\to k R_{\QQ}$ is a projective cover in the category $\calF_k^{\QQ^\times}$.

\end{theorem}

\bigskip
\begin{center}\sc\large \ref{sec examples}D. The Functor of Trivial Source Modules\end{center}

\bigskip\noindent
Let $F$ be an algebraically closed field of characteristic $p>0$. For any finite group $G$, we denote by $T(G)$ the ring of trivial source $F[G]$-modules, also known as $p$-permutation $FG$-modules, see \cite[Section~5.5]{Benson1995}, \cite{Broue1985}, or \cite[\S 27]{Thevenaz1995}. Again we have a linearization map
\begin{equation*}
 \lin_G\colon B^{F^\times}(G) \rightarrow T(G)\,,\quad [H,\phi]_G\mapsto [\Ind_H^G(F_\phi)]\,. 
\end{equation*}
With the same arguments as at the beginning of Subsection~11B one obtains that $T$ is an $F^\times$-fibered biset functor and that the maps $\lin_G$ form a morphism of $F^\times$-fibered biset functors. Note that $\lin_G$ is surjective (see \cite{Dress1975}, \cite{Boltje1998b}, or \cite{BoltjeKuelshammer2005}). Now let $k$ be a field. Then the functor $kT:=k\otimes T$ is a quotient of the functor $B_k^{F^\times}$. Since $B^A_k$ has a unique simple quotient isomorphic to $S_{(\{1\},\{1\},1,k)}$ (see Theorem~\ref{thm example BA}), also $kT$ has a unique quotient functor, isomorphic to $S_{(\{1\},\{1\},1,k)}$. 

We will now use Proposition~\ref{prop simplicity criterion} applied to $H:=\{1\}$ in order to show that $kT$ is not a simple functor. More precisely we show that $\calK_{kT,\{1\}}\neq0$. In fact, let $G=C_p\times C_p$ be the elementary abelian $p$-group of rank $2$ and let $H_1,\ldots, H_{p+1}$ denote its subgroups of order $p$. Then, the map $\lin_G\colon B^{F^\times}(G)=B(G)\to T(G)$, $[G/P]\mapsto\Ind_P^G(F)$, is an isomorphism, since $G$ is a $p$-group. Here $B(G)$ denotes the Burnside ring of $G$. Moreover, identifying $B^{F^\times}(\{1\},G)=B(\{1\},G)$ with $B(G)$, the map $kT([G/P])\colon kT(G)\to kT(\{1\})=k$ is given by $[G/Q]\mapsto |P\backslash G/Q|\cdot 1_k$, for $Q\le G$. Now it is easy to see that the element $p[G/G]- ([G/H_1]+\cdots+[G/H_{p+1}])+[G/\{1\}]\in kB(G)\cong kT(G)$ is a non-zero element in $\calK_{kT,\{1\}}(G)$. Thus, we have the following result.

\begin{proposition}\label{prop ex trivial source functor}
Let $k$ be a field. Then the $F^\times$-fibered biset functor $kT$ is not simple. It has a unique quotient functor and this quotient is isomorphic to $S_{(\{1\},\{1\},1,k)}$
\end{proposition}

%%%%%%%%%%%%%%%%%%%%%% BIBLIOGRAPHY %%%%%%%%%%%%%%%%%%%%%%%%%%%%%%%%%%

%%%%%%%%%%%%%%%%%%%%%% END %%%%%%%%%%%%%%%%%%%%%%%%%%%%%%%%%%%%%%%%%%%


\begin{thebibliography}{XXXXX}
\bibitem[Ba04]{Barker} {\sc L.~Barker:} Fibred permutation sets and the idempotents and units of monomial Burnside rings. 
{\sl J.~Alg.}~{\bf 281} (2004), 535--566.
\bibitem[Be95]{Benson1995} {\sc D.~J.~Benson:}
Representations and cohomology. I. Basic representation theory of finite groups and associative algebras. Second edition. Cambridge Studies in Advanced Mathematics, 30. Cambridge University Press, Cambridge, 1998.
\bibitem[Bol89]{Boltje1989} {\sc R.~Boltje:} A canonical Brauer induction formula. {\sl Ast\'erisque} {\bf 181--182} (1990), 31--59.
\bibitem[Bol98a]{Boltje1998} {\sc R.~Boltje:} A general theory of canonical induction formulae. {\sl J.~Alg.}~{\bf 206} (1998) 293--343.
\bibitem[Bol98b]{Boltje1998b} {\sc R.~Boltje:} Linear source modules and trivial source modules. Group representations: cohomology, group actions and topology (Seattle, WA, 1996), 7--30, Proc. Sympos. Pure Math., 63, Amer. Math. Soc., Providence, RI, 1998. 
\bibitem[Bol01]{Boltje2001a} {\sc R.~Boltje:}
  Monomial resolutions.
  {\sl J. Algebra\/}~{\bf 246} (2001), 811--848.
\bibitem[BK05]{BoltjeKuelshammer2005}{\sc R.~Boltje, B.~K\"ulshammer:} Explicit and canonical Dress induction. {\sl Algebr. Represent. Theory\/}~{\bf 8} (2005), 731--746.
\bibitem[Bou96]{Bouc1996a} {\sc S.~Bouc:} Foncteurs d'ensembles munis d'une
double action. {\sl J.~Alg.\/}~{\bf 183} (1996), 664--736.
\bibitem[Bou10a]{Bouc2010a} {\sc S.~Bouc:} The functor of units of Burnside rings for p-groups. {\sl Comment.\ Math.\ Helv.}~{\bf 82} (2007), 583--615. 
\bibitem[Bou10b]{BoucBook} {\sc S.~Bouc:} Biset functors for finite groups. Lecture Notes in Mathematics, 1990. Springer-Verlag, Berlin, 2010.
\bibitem[Bou10c]{Bouc2010c} {\sc S.~Bouc:} Bisets as categories, and tensor product of induced bimodules. 
{\sl Appl.\ Categ.\ Structures\/}~{\bf 18} (2010), 517--521.
\bibitem[BT00]{BoucThevenaz2000a} {\sc S.~Bouc, J.~Th\'evenaz:}
  The group of endo-permutation modules.
  {\sl Invent. math.\/}~{\bf 139} (2000), 275--349.
\bibitem[BT15]{BoucThevenaz2015} {\sc S.~Bouc:, J.~Th\'evenaz:}
   The representation theory of finite sets and correspondences. Preprint 2015. 	arXiv:1510.03034.
\bibitem[Br85]{Broue1985} {\sc M.~Brou\'e:}
On Scott modules and p-permutation modules: an approach through the Brauer morphism. 
{\sl Proc. Amer. Math. Soc.} {\bf 93} (1985), 401--408. 
\bibitem[D71]{Dress}{\sc A.~Dress:} The ring of monomial representations I. Structure theory. {\sl J.~Alg.}~{\bf 18}
(1971) 137--157.
\bibitem[D75]{Dress1975} {\sc A.~Dress:}
Modules with trivial source, modular monomial representations and a modular version of Brauer's induction theorem. 
{\sl Abh. Math. Sem. Univ. Hamburg\/}~{\bf 44} (1975), 101--109 (1976).
\bibitem[I76]{Isaacs1976a} {\sc M.~Isaacs:} Character theory of finite groups. Pure and Applied Mathematics, No. 69. Academic Press, New York-London, 1976.
\bibitem[K87]{Karpilovsky1987a} {\sc G.~Karpilovsky:} 
The Schur multiplier. London Mathematical Society Monographs. New Series, 2. The Clarendon Press, Oxford University Press, New York, 1987.
\bibitem[R12]{Romero2012a} {\sc N.~Romero:} Simple modules over Green biset functors.
{\sl J.~Alg.\/}~{\bf 367} (2012), 203--221.
\bibitem[R13]{Romero2013a}{\sc N.~Romero:} On fibered biset functors with fibres of order prime and four. 
{\sl J.~Alg.\/}~{\bf 387} (2013), 185--194.
\bibitem[T95]{Thevenaz1995} {\sc J.~Th\'evenaz:}
$G$ -algebras and modular representation theory. Oxford Mathematical Monographs. Oxford Science Publications. The Clarendon Press, Oxford University Press, New York, 1995.
\bibitem[T07]{Thevenaz2007} {\sc J.~Th\'evenaz:} Endo-permutation modules, a guided tour. Group representation theory, 115--147, EPFL Press, Lausanne, 2007.
\bibitem[TW]{TW} {\sc J.~Th\'evenaz, P.~Webb:} The structure of Mackey functors.
{\sl Trans. Amer. Math. Soc.}~{\bf 347} (1995), 1865--1961.
\end{thebibliography}
\end{document}